\documentclass{article}%
\usepackage{amsfonts}
\usepackage{amsmath}
\usepackage{amssymb}
\usepackage{graphicx}%
\providecommand{\U}[1]{\protect\rule{.1in}{.1in}}
\newtheorem{theorem}{Theorem} [section]

\newtheorem{conclusion}[theorem]{Conclusion}

\newtheorem{definition}[theorem]{Definition}
\newtheorem{example}[theorem]{Example}

\newtheorem{proposition}[theorem]{Proposition}
\newtheorem{remark}[theorem]{Remark}

\newenvironment{proof}[1][Proof]{\noindent\textbf{#1.} }{\ \rule{0.5em}{0.5em}}
\begin{document}
\author{Robert Shwartz\\Department of Mathematics\\Ariel University, Israel\\robertsh@ariel.ac.il}
\date{}
\title{OGS canonical forms and exchange laws for the $I$ and for the $A$-type Coxeter groups}
\maketitle
\begin{abstract}
We consider a generalization of the fundamental theorem of finitely generated abelian groups for some non-abelian groups, which is called $OGS$. First, we consider the dihedral group $Dih(A)$, which is a non-abelian extension of an abelian group $A$ by an involution. Then, we focus on a special case of $Dih(A)$, where $A$ is cyclic, which is the two-generated Coxeter group $I_{2}(m)$. We mention interesting connections between the reduced Coxeter presentation and a particular $OGS$ canonical presentation, which we call the standard $OGS$ canonical presentation. These connections motivate us to offer a generalization of the standard $OGS$ to the $A$-type Coxeter group, which can be considered as the dual family to the $I$-type Coxeter groups. The $A$-type  Coxeter graph is connected, with edges limited to simply laced, but the number of vertices is not limited, and depends on $n$. The $(n-1)$-generated $A$-type Coxeter groups can be considered as the symmetric group $S_{n}$ for an arbitrary $n$. We mention the standard and the dual-standard $OGS$ of $S_n$, where, The standard $OGS$ canonical form of $S_{n}$  has a special interest in combinatorics, since in 2001, R. M. Adin, and Y. Roichman has proved that sum of the exponents in the canonical form is coincide with the major-index of the permutation, which is equi-disributed with the Coxeter length. In this paper we extend the results of Adin and Roichman very significantly, where we show interesting properties of the exchange laws, we define standard $OGS$ elementary factorization, which connects between the standard $OGS$ and the descent set of a permutation. Then, by using the standard $OGS$ elementary factorization, we find  a new explicit formula for the Coxeter length of an element of $S_n$, and we give a new algorithm for the standard $OGS$ canonical form and the descent set of the inverse element of an arbitrary element of $S_n$.



\end{abstract}

\section{Introduction}

One of the most important theorems of Linear Algebra is  that every vector-space $V$ over a field $\mathbb{F}$ has a basis, i.e. there are elements $v_{1}, v_{2}, \ldots v_{n}$ in $V$, such that every vector $v$ in $V$ has a unique presentation of a form:
$$v=\alpha_{1}\cdot v_{1}+\alpha_{2}\cdot v_{2}+\cdots +\alpha_{n}\cdot v_{n}, ~~~\alpha_{1}, \alpha_{2}, \ldots, \alpha_{n}\in \mathbb{F}$$
Thus, the vector $v$ in $V$ can be expressed by its $n$ coordinates $(\alpha_{1}, \alpha_{2}, \ldots, \alpha_{n})$. There is a generalization of the basis for finitely generated abelian groups.  Let $A$ be a finitely generated abelian group, then by the fundamental theorem of finitely generated abelian groups there exists generators $a_{1}, a_{2}, \ldots a_{n}$, such that every element $a$ in $A$ has a unique presentation of a form:
$$g=a_{1}^{i_{1}}\cdot a_{2}^{i_{2}}\cdots a_{n}^{i_{n}}.$$
where, $i_{1}, i_{2}, \ldots, i_{n}$ are $n$ integers such that for  $1\leq k\leq n$, $0\leq i_{k}<|g_{k}|$, where $a_{k}$ has a finite order of $|a_{k}|$ in $A$, and $i_{k}\in \mathbb{Z}$, where $a_{k}$ has infinite order in $A$. The mentioned  presentation is  the canonical presentation of $a\in A$ according to the basis  $\langle a_{1}, a_{2}, \ldots, a_{n}\rangle$. Since $A$ is an abelian group, the following exchange laws hold: $a_{k}\cdot a_{j}=a_{j}\cdot a_{k}$, for each $1\leq j, k\leq n$. The exchange law allows us to put each product of two elements of the group $A$ onto the mentioned canonical form.  Therefore, we can write every element $g$ of an abelian group $A$ as a $n$-tuple of integers $(i_1, i_2, \ldots i_n)$, where $i_k\in \mathbb{Z}_{|a_{k}|}$ for $1\leq r\leq n$ (we denote by  $\mathbb{Z}_{r}$ the set of integers modulo $r$, and we denote by by $\mathbb{Z}_{\infty}$ the set of the integers).  Thus,
$$(i_{1}, i_{2}, \ldots i_{n})+(j_{1}, j_{2},\ldots j_{n})=(i_{1}+j_{1}, i_{2}+j_{2}, \ldots, i_{n}+j_{n}),$$
where, $i_{k}+j_{k}$ is the group operation in the additive group of $\mathbb{Z}_{|a_{k}|}$.
In this paper we consider a generalization of the canonical form by a given basis as it arises from the fundamental theorem of Abelian groups to the non-abelian case in the following way:
\begin{definition}\label{ogs}
Let $G$ be a non-abelian group. The ordered sequence of $n$ elements $\langle g_{1}, g_{2}, \ldots, g_{n}\rangle$ is called an $Ordered ~~ Generating ~~ System$ of the group $G$ or by shortened notation, $OGS(G)$, if every element $g\in G$ has a unique presentation in the a form
$$g=g_{1}^{i_{1}}\cdot g_{2}^{i_{2}}\cdots g_{n}^{i_{n}},$$
where, $i_{1}, i_{2}, \ldots, i_{n}$ are $n$ integers such that for  $1\leq k\leq n$, $0\leq i_{k}<r_{k}$, where  $r_{k} | |g_{k}|$  in case the order of $g_{k}$ is finite in $G$, or   $i_{k}\in \mathbb{Z}$, in case  $g_{k}$ has infinite order in $G$.
The mentioned canonical form is called $OGS$ canonical form.  For every $q>p$, $1\leq x_{q}<r_{q}$, and $1\leq x_{p}<r_{p}$ the relation
$$g_{q}^{x_{q}}\cdot g_{p}^{x_{p}} = g_{1}^{i_{1}}\cdot g_{2}^{i_{2}}\cdots g_{n}^{i_{n}},$$
is called exchange law.
\end{definition}
In contrast to finitely generated abelian groups, the existence of an $OGS$ is generally not true for every finitely generated non-abelian group. Even in case of two-generated infinite non-abelian groups it is not too hard to find counter examples. For example, the Baumslag-Solitar groups $BS(m,n)$ \cite{BS}, where $m\neq \pm1$ or $n\neq \pm1$, or most of the cases of the one-relator free product of a finite cyclic group generated by $a$, with a finite two-generated group generated by $b, c$ with the relation $a^{2}\cdot b\cdot a\cdot c=1$ \cite{S}, do not have an $OGS$. Even the question of the existence of an $OGS$ for a general finite non-abelian group is still open. Moreover, contrary to the abelian case where the exchange law is just $g_{q}\cdot g_{p}=g_{p}\cdot g_{q}$, in most of the cases of non-abelian groups with the existence of an $OGS$, the exchange laws are very complicated. Although there are some specific non-abelian groups where the exchange laws are very convenient and have very interesting properties (For example, in the case of $PSL_{2}(q)$ there is an $OGS$ which is closely connected to the $BN-pair$ decomposition, where the exchange laws yield  some interesting recursive sequences over finite fields \cite{SY}). In this paper we deal with the two most significant classes of  Coxeter groups (namely the $I$-type and the $A$-type), which  have an $OGS$ canonical presentation strongly connected to the presentation in Coxeter generators, and  with very interesting and surprising exchange laws. The paper is divided as follows: In section \ref{dih}, we describe an $OGS$ canonical form and the related exchange laws for the family of the dihedral groups $Dih(A)$, which are non-abelian extensions of an abelian group $A$ by a cyclic group of order $2$. Then, we focus in the case of $Dih(A)$, where the abelian group $A$ is cyclic, since then the group  is a Coxeter group with two generators. We show a connection between the canonical form according to an $OGS$, which we call standard $OGS$, and the reduced Coxeter presentation of it. In section \ref{s-n}, we show a generalization of the standard $OGS$ canonical form to the $A$-type Coxeter groups, which can be considered as a dual family of the $I$-type Coxeter groups, where instead of limiting the number of vertices to two in the Coxeter graph, we limit the lace of the edges to be simply laced. The $(n-1)$-generated $A$-type Coxeter group is the symmetric group $S_n$, which can be considered as the permutation group on $n$ elements. Therefore, a lot of work has been accomplished concerning the connections between permutation invariants and the Coxeter length of the elements by Brenti, Bjorner \cite{BB}, B\'{o}na \cite{BO}, Foata, Schutzenberger \cite{FS}, Grasia, Gessel \cite{GG}, Reiner \cite{R}, Stanley \cite{Sta}, Steingrimsson \cite{Ste}, Bagno, Garber, Mansour, Shwartz \cite{BGMS}, and many others.
  In the same aspect, Adin and Roichman \cite{AR} introduced a presentation of an $OGS$ canonical form  for the symmetric group $S_n$, for the hyperoctahedral group $B_n$, and for the wreath product $\mathbb{Z}_{m}\wr S_{n}$. In the case of $S_n$, the $OGS$ which they used, coincides with the standard $OGS$ by our definition. Adin and Roichman proved that for every element of $S_n$ presented in the standard $OGS$ canonical form, the sum of the exponents of the $OGS$ equals the major-index of the permutation. Moreover, by using an $OGS$ canonical form, Adin and Roichman generalized the theorem of MacMahon \cite{Mac} to the $B$-type Coxeter group, and to the wreath product $\mathbb{Z}_{m}\wr S_{n}$. A few years later, that $OGS$ canonical form was generalized for complex reflection groups by Shwartz, Adin and Roichman \cite{SAR}. Although an $OGS$ canonical form for the symmetric groups $S_n$ has been already introduced, and a lot of work has been done concerning permutation statistics and Coxeter length of elements in the symmetric groups, nothing has been carried out yet concerning very important and very interesting aspects of the $OGS$ canonical forms, like exchange laws, an algorithm for a canonical form of the inverse element of a given element of the symmetric group, or an explicit formula for the Coxeter length of a given element of $S_n$. In this paper, we find the related exchange laws for the standard and for the dual-standard $OGS$ canonical forms of the symmetric group $S_n$, with very interesting and surprising properties. By using the standard $OGS$, we define standard $OGS$ elementary elements, which coincide with the elements of $S_{n}$ with a single descent. Then, we define standard $OGS$ elementary factorization onto standard $OGS$ elementary factors,  such that the number of the elementary factors of an element $\pi\in S_{n}$ equals to the size of the descent set of $\pi$. We apply the standard $OGS$ elementary factorization of $\pi$ to introduce an algorithm for the standard $OGS$ canonical form and the descent set of the inverse element $\pi^{-1}$.  We also give a new explicit formula for the Coxeter length of a permutation in the symmetric group $S_n$ by using the standard $OGS$ canonical form, and the standard $OGS$ elementary factorization.

\section{$OGS$ Canonical forms and exchange laws for $Dih(A)$}\label{dih}

In this section we show an $OGS$ canonical form, with very simple exchange laws,  for the dihedral groups, a very important family of non-abelian groups. Then, we show connections between  the mentioned $OGS$ canonical form and the presentation in Coxeter generators of the two-generated Coxeter groups, which are dihedral groups.

\begin{definition}\label{dihed}
Let $A$ be an abelian group. The map $\phi: A\rightarrow A$, such that $\phi(a)=a^{-1}$ is an automorphism of $A$.   Then, we define $Dih(A)$ to be the dihedral group of order $2|A|$, as an extension of the group $A$ by an involution $b$, where $b(a)=a^{-1}$, for every $a\in A$.
\end{definition}
Obviously, the relations of $Dih(A)$ are the relations of $A$ and conjugation  of the elements of $A$ to their inverse by an involution $b$, i.e., $b^{-1}\cdot a\cdot b=a^{-1}$ for every $a\in A$, which is equivalent to: $b\cdot a=a^{-1}\cdot b$, for every $a\in A$. Thus,
every element of $Dih(A)$ has a unique presentation in the following canonical form
$$g=a_{1}^{i_{1}}\cdot a_{2}^{i_{2}}\cdots a_{n}^{i_{n}}\cdot b^{j}$$
where, $i_{1}, i_{2}, \ldots, i_{n}$ are $n$ integers such that for  $1\leq k\leq n$, $0\leq i_{k}<|a_{k}|$, where $a_{k}$ has a finite order of $|a_{k}|$ in $A$, and $i_{k}\in \mathbb{Z}$, where $a_{k}$ has infinite order in $A$, and $0\leq j<2$. By Definition \ref{ogs}, the ordered sequence $\langle a_1, a_2, \ldots, a_n, b\rangle$ is an $OGS$ for $G=Dih(A)$.
The relation $b\cdot a=a^{-1}\cdot b$ for every $a\in A$ implies exchange laws of the form $b\cdot a_{k}^{i_{k}}=a_{k}^{-i_{k}}\cdot b$. Therefore,
$$(i_{1}, i_{2}, \ldots i_{n}, j)+(p_{1}, p_{2},\ldots p_{n}, q)=(i_{1}+(-1)^{j}\cdot p_{1}, i_{2}+(-1)^{j}\cdot p_{2}, \ldots, i_{n}+(-1)^{j}\cdot p_{n}, j+q),$$
where, the operation  "$+$" in $i_{k}+(-1)^{j}\cdot p_{k}$ is the group operation in the additive group of $\mathbb{Z}_{r}$, and $j+q$ is the group operation in the additive group of  $\mathbb{Z}_2$.
\begin{proposition}\label{dihedral-order}
 Every ordered sequence of a form  $$\langle a_{\pi(1)}, \ldots, a_{\pi(w)}, b, a_{\pi(w+1)}, \ldots a_{\pi(n)}\rangle,$$
 where, $w$ is an arbitrary integer such that $1\leq w\leq n$, and $\pi$ is a permutation of the elements in the set of the $n$ integers  $\{1, 2, \ldots n\}$, forms an $OGS$ for $G=Dih(A)$, where the exchange laws of the given $OGS$ canonical form is the following:
 \begin{itemize}
 \item $a_{\pi(k)}\cdot a_{\pi(j)}=a_{\pi(j)}\cdot a_{\pi(k)}$, for $1\leq j,k,\leq n$ (i.e., commutative exchange laws);
 \item $b\cdot a_{\pi(k)}^{i_{\pi(k)}}=a_{\pi(k)}^{-i_{\pi(k)}}\cdot b$, for $1\leq \pi(k)\leq w$;
 \item $a_{\pi(k)}^{i_{\pi(k)}}\cdot b = b\cdot a_{\pi(k)}^{-i_{\pi(k)}}$, for $w+1\leq \pi(k)\leq n$.
 \end{itemize}
\end{proposition}

\begin{proof}
The proof comes directly by the definition of $G=Dih(A)$ as an extension of the abelian group $A$ by an involution $b$, according to the automorphism $b(a)=a^{-1}$, for every $a\in A$.
\end{proof}

Now, we show the $OGS$ canonical form of the inverse element $g^{-1}$ for a given element $g\in Dih(A)$.

\begin{proposition}\label{inverse-dih}
Let $G=Dih(A)$ for an abelian group $A$, where $A$ has a basis \\ $\langle a_1, a_2, \ldots, a_n\rangle$ by the fundamental theorem of abelian groups. Consider the $OGS$ canonical form with $OGS(G)=\langle a_{\pi(1)}, \ldots, a_{\pi(w)}, b, a_{\pi(w+1)}, \ldots a_{\pi(n)}\rangle$, where $w$ is an arbitrary integer such that $1\leq w\leq n$, and $\pi$ is a permutation of the elements in the set of the $n$ integers  $[n]=\{1, 2, \ldots, n\}$.
 \begin{itemize}
 \item If $g\in A$,  i.e., $$g=a_{\pi(1)}^{i_{\pi(1)}}\cdot a_{\pi(2)}^{i_{\pi(2)}}\cdots a_{\pi(w)}^{i_{\pi(w)}}\cdot a_{\pi(w+1)}^{i_{\pi(r+1)}}\cdots a_{\pi(n)}^{i_{\pi(n)}}$$ then
     $$g^{-1}=a_{\pi(1)}^{|a_{\pi(1)}|-i_{\pi(1)}}\cdot a_{\pi(2)}^{|a_{\pi(2)}|-i_{\pi(2)}}\cdots a_{\pi(w)}^{|a_{\pi(w)}|-i_{\pi(w)}}\cdot a_{\pi(w+1)}^{|a_{\pi(w+1)}|-i_{\pi(r+1)}}\cdots a_{\pi(n)}^{|a_{\pi(n)}|-i_{\pi(n)}};$$
 \item If $g\notin A$, i.e., $$g=a_{\pi(1)}^{i_{\pi(1)}}\cdot a_{\pi(2)}^{i_{\pi(2)}}\cdots a_{\pi(w)}^{i_{\pi(w)}}\cdot b\cdot a_{\pi(w+1)}^{i_{\pi(r+1)}}\cdots a_{\pi(n)}^{i_{\pi(n)}},$$  then $g^{-1}=g$ (i.e., $g$ is an involution).
\end{itemize}
\end{proposition}

\begin{proof}
If $g\in A$, then we conclude $g^{-1}$ by the fundamental theorem of abelian groups. Therefore, assume $g\notin A$, which means the $OGS$ canonical form of $g$ involves $b$. Notice, the relation $b\cdot a=a^{-1}\cdot b$, for every $a\in A$, applies:
\begin{align*}
\left(a'\cdot b\cdot a''\right)^{2} &= a'\cdot (b\cdot a''\cdot a')\cdot b\cdot a'' \\ &= a'\cdot \left(a''\right)^{-1}\cdot \left(a'\right)^{-1}\cdot b^{2}\cdot a'' \\ &=a'\cdot \left(a''\right)^{-1}\cdot \left(a'\right)^{-1}\cdot a'' \\ &= a'\cdot \left(a'\right)^{-1}\cdot \left(a''\right)^{-1}\cdot a'' \\ &= 1,
\end{align*}
for every $a', a''\in A$.
Thus, $$\left(a'\cdot b\cdot a''\right)^{-1}=a'\cdot b\cdot a''$$
for every $a', a''\in A$.
\end{proof}

The following example shows how we can multiply two arbitrary elements in $Dih(A)$, which are presented in $OGS$ canonical form.

\begin{example}
Let $A$ be $\mathbb{Z}_{9}\bigoplus\mathbb{Z}_{3}$, where the elements $a_{1}$ and $a_{2}$ generates $A$, such that $|a_{1}|=9$, $|a_{2}|=3$, and every element in $A$ has a unique presentation in the canonical form $a_{1}^{i_{1}}\cdot a_{2}^{i_{2}}$, where $0\leq i_{1}<9$, and $0\leq i_{2}<3$. Now, consider the group $Dih(\mathbb{Z}_{9}\bigoplus\mathbb{Z}_{3})$, which is the extension of $\mathbb{Z}_{9}\bigoplus\mathbb{Z}_{3}$ by an involution $b$ such that $b\cdot a=a^{-1}\cdot b$, for every $a\in A$. Then, every element of $Dih(\mathbb{Z}_{9}\bigoplus\mathbb{Z}_{3})$ ha a unique presentation in a canonical form $a_{1}^{i_{1}}\cdot a_{2}^{i_{2}}\cdot b^{j}$, where $0\leq i_{1}<9$, $0\leq i_{2}<3$, and $0\leq j<2$, with the exchange laws:
\begin{itemize}
\item $a_{2}^{i_{2}}\cdot a_{1}^{i_{1}}=a_{1}^{i_{1}}\cdot a_{2}^{i_{2}}$;
\item $b\cdot a_{1}^{i_{1}}=a_{1}^{9-i_{1}}\cdot b$;
\item $b\cdot a_{2}^{i_{2}}=a_{2}^{3-i_{2}}\cdot b$.
\end{itemize}

Thus, for example let $x$ and $y$ be the following elements: $x=a_{1}^{4}\cdot a_{2}^{2}\cdot b$, and $y=a_{1}^{7}\cdot a_{2}\cdot b$, then by the exchange laws the following holds:
$$x\cdot y=a_{1}^{4}\cdot a_{2}^{2}\cdot b\cdot a_{1}^{7}\cdot a_{2}\cdot b = a_{1}^{4}\cdot a_{2}^{2}\cdot a_{1}^{9-7}\cdot a_{2}^{3-2}\cdot b\cdot b = a_{1}^{6}\cdot a_{2}^{3}\cdot b^{2} = a_{1}^{4}.$$

\end{example}

\subsection{The Coxeter group $I_{2}(m)$}

There is a special interest in the family dihedral groups $Dih(A)$, where $A$ is a cyclic group. Let $A$ be a cyclic group of order $m$ ($m$ might be $\infty$), then $Dih(A)$ is a two-generated Coxeter group $I_{2}(m)$, of order $2m$, where $m$ is finite, or order $\infty$, in case $m=\infty$. We recall the Coxeter presentation of $I_{2}(m)$, and some basic properties of it, as described in \cite{BB}:
\begin{itemize}
\item $I_{2}(m)=\langle s_1, s_2 | s_1^{2}=s_2^{2}=1, (s_1\cdot s_2)^{m}=1\rangle$ in case of finite $m$;
\item $I_{2}(\infty)=\langle s_1, s_2 | s_1^{2}=s_2^{2}=1\rangle$, i.e., $I_{2}(\infty)=\mathbb{Z}_{2}*\mathbb{Z}_{2}$.
\end{itemize}
Now, define $b$ to be $s_1$, and define $a$ to be $s_1\cdot s_2$. Then, $\langle b, a\rangle$ is an $OGS$ for $I_{2}(m)$ with the exchange law
$a\cdot b=b\cdot a^{m-1}$, in case of finite $m$, or $a\cdot b=b\cdot a^{-1}$ in case of $m=\infty$.
Now, Consider the presentation of the elements of $I_{2}(m)$ in Coxeter generators.
\begin{proposition}\label{cox-ogs}
Let $G=I_{2}(m)$, with the Coxeter generators $s_1, s_2$, and let $b=s_1$, $a=s_1\cdot s_2$, then the following holds:
\begin{itemize}
\item $b=s_1$;
\item $b\cdot a=s_2$;
\item $b\cdot a^{i}=s_2\cdot s_1\cdot s_2\cdots s_1\cdot s_2=(s_2\cdot s_1)^{i-1}\cdot s_2$, for every $1<i\leq\frac{m+1}{2}$ in case of finite $m$, and for every $i>1$ in case of infinite $m$;
\item $b\cdot a^{i}=s_1\cdot s_2\cdot s_1\cdots s_2\cdot s_1=(s_1\cdot s_2)^{m-i}\cdot s_1$, for every $\frac{m+1}{2}\leq i<m$ in case of finite $m$, and for every $i<0$ in case of infinite $m$;
\item $a^{i}=s_1\cdot s_2\cdot s_1\cdots s_1\cdot s_2=(s_1\cdot s_2)^{i}$, for every $0<i\leq \frac{m}{2}$ in case of finite $m$, and for every $i>0$ in case of infinite $m$;
\item $a^{i}=s_2\cdot s_1\cdot s_2\cdots s_2\cdot s_1=(s_2\cdot s_1)^{m-i}$, for every $\frac{m}{2}\leq i<m$ in case of finite $m$, and for every $i<0$ in case of infinite $m$.
\end{itemize}
\end{proposition}
\begin{proposition}\label{length-cox-canon}
Let $G$ be $I_{2}(m)$ for a finite $m$, then the Coxeter length is equidistributed with the length (sum of the exponents) in the $OGS$ canonical form by \\ $OGS(I_{2}(m))=\langle b, a\rangle $.
\end{proposition}

\begin{proof}
Let $G=I_{2}(m)$, then the following hold:
\begin{itemize}
\item By \cite{BB}, there are exactly two elements with Coxeter length $i$ for every $1\leq i\leq m-1$. Namely $(s_1\cdot s_2)^{\frac{i}{2}}$, and $(s_2\cdot s_1)^{\frac{i}{2}}$ for an even $i$. $(s_1\cdot s_2)^{\frac{i-1}{2}}\cdot s_1$ and  $(s_2\cdot s_1)^{\frac{i-1}{2}}\cdot s_2$  for an odd $i$;
\item By \cite{BB}, there is exactly one element with Coxeter length $0$ and Coxeter length $m$. The element with Coxeter length $m$ is $(s_1\cdot s_2)^{\frac{m}{2}}= (s_2\cdot s_1)^{\frac{m}{2}}$ in case of even $m$, and $(s_1\cdot s_2)^{\frac{m-1}{2}}\cdot s_1= (s_2\cdot s_1)^{\frac{m-1}{2}}\cdot s_2$ in case of odd $m$;
\item The presentation of the elements in $I_{2}(m)$  in the canonical form according to the sequence $\{b, a\}$ is $b^{j}\cdot a^{i}$, such that $0\leq j\leq 1$, $0\leq i\leq m-1$. Therefore, there are two elements of length $i$ for every $1\leq i\leq m-1$, namely $a^{i}$, and $b\cdot a^{i-1}$.  The identity is the only element of length $0$,  and $b\cdot a^{m-1}$ is the only one element of length $m$.
\end{itemize}
\end{proof}

In Propositions \ref{cox-ogs} and \ref{length-cox-canon}, we show interesting connections between two presentations of the two-generated Coxeter group, $I_{2}(m)$:
\begin{itemize}
\item The presentation in Coxeter generators;
\item The  $OGS$ canonical presentation for $OGS(I_{2}(m))=\langle b, a\rangle$, where $b=s_1$ and $a=s_1\cdot s_2$.
\end{itemize}
Therefore, we call $OGS(I_{2}(m))=\langle b, a\rangle$, the standard $OGS$ of $I_{2}(m)$.

\begin{remark}\label{symmetry-dihedral}
The  geometric meaning of the group $I_{2}(m)=Dih(\mathbb{Z}_{m})$   is the symmetry group of a regular $m$-sided polygon, where the $m$ elements of $\mathbb{Z}_{m}$ present the rotations of the polygon, and the $m$ elements of the form $b\cdot a^{i}$ (where, $0\leq i<m$) present the $m$ reflections of the polygon.  The case of $m=3$ is the non-abelian group which has the smallest order, and in this case $I_{2}(3)=Dih(\mathbb{Z}_{3})$ is the symmetric group on $3$ elements, which is denoted by $S_3$.
\end{remark}

\section{$OGS$ canonical forms and exchange laws for the symmetric group $S_n$}\label{s-n}

In the last section we have shown interesting connections between the presentation in the standard $OGS$ canonical form, and the presentation in Coxeter generators for every two-generated Coxeter group $I_{2}(m)$. These connections motivate us to look at a generalization for Coxeter groups with more than two Coxeter generators.  The family $I_{2}(m)$ is the natural generalization of $I_{2}(3)$, changing the lace of the edge connecting the two vertices on the Coxeter graph to arbitrary lace, but fixing the number of vertices to two. The natural analogous generalization of $I_{2}(3)$ is the $A$-type Coxeter group with arbitrary $n$ generators, which we get by changing the number of vertices on the Coxeter graph, but keeping it connected and simply laced by connecting every two adjacent vertices in the Coxeter graph by a simply laced edge. Therefore, in this section we consider the $A$-type Coxeter groups, where the Coxeter graph is a line with arbitrary $n$ vertices, but the lace of the connecting edges is fixed to be simply laced for every two adjacent vertices in the Coxeter graph. Notice, the smallest non-abelian $I$-type Coxeter group $I_{2}(3)$ is the same to the smallest non-abelian $A$-type Coxeter group $A_{2}$, where the Coxeter graph is just two vertices, which are  connected by a simply-laced edge. The $(n-1)$-generated $A$-type Coxeter group can considered as the symmetric group on $n$ elements, $S_n$. Adin and Roichman  \cite{AR} have introduced an $OGS$ canonical form for the family of Coxeter groups $S_n$ and $B_n$ s.t.; the length according to that form is equi-distributed with the Coxeter length (as has been shown for $I_{2}(m)$). They also give a combinatorial interpretation for the length of an element in $S_n$ or in $B_n$ according to that $OGS$ canonical form, which is called flag-major-index and motivated by MacMahon's theorem for the symmetric groups \cite{Mac}. In this section we find the exchange laws for that $OGS$ canonical form  of $S_n$, and for the dual $OGS$ canonical form. We also show interesting properties of the exchange laws, which allows for an efficient multiplication of elements in $S_n$. We give a new explicit formula for the Coxeter length of an element in $S_n$ using the $OGS$ canonical form. Finally, we find a formula for the $OGS$ canonical form of the inverse element of an arbitrary element of $S_n$.

We start with some basic definitions concerning the symmetric group $S_n$.

\begin{definition}\label{sn}
Let $S_n$ be the symmetric group on $n$ elements, then :
\begin{itemize}
\item The symmetric group $S_n$ is an $(n-1)$-generated simply-laced Coxeter group which has the presentation of: $$\langle s_1, s_2, \ldots, s_{n-1} | s_i^{2}=1, ~~ (s_i\cdot s_{i+1})^{3}=1, ~~(s_i\cdot s_j)^2=1 ~~for ~~|i-j|\geq 2\rangle;$$
\item The group $S_n$ can be considered as the permutation group on $n$ elements. A permutation $\pi\in S_n$ is denoted by $[\pi(1);\pi(2);\ldots;\pi(n)]$ (i.e., $\pi=[2;4;1;3]$ is a permutation in $S_{4}$ which satisfies $\pi(1)=2$, $\pi(2)=4$, $\pi(3)=1$, and $\pi(4)=3$);
\item Every permutation $\pi\in S_n$ can be presented in a cyclic notation, as a product of disjoint cycles of the form $(i_1, ~i_2, ~\ldots, ~i_m)$, which means $\pi(i_{k})=i_{k+1}$, for $1\leq k\leq m-1$, and $\pi(i_{m})=i_{1}$
    (i.e., The cyclic notation of $\pi=[3;4;1;5;2]$ in $S_5$, is $(1, ~3)(2, ~4, ~5)$);
\item The Coxeter generator $s_i$ can be considered the permutation which exchanges the element $i$ with the element $i+1$, i.e., the transposition $(i, i+1)$;
\item We consider multiplication of permutations in left to right order; i.e., for every $\pi_1, \pi_2\in S_n$, $\pi_1\cdot \pi_2 (i)=\pi_2(j)$, where $\pi_1(i)=j$ (in contrary to the notation in \cite{BB}, \cite{AR} where, Brenti, Bjorner, Adin, Roichman and other people have considered right to left multiplication of permutations);
\item For every permutation $\pi\in S_n$, the Coxeter length $\ell(\pi)$ is the number of inversions in $\pi$, i.e., the number of different pairs $i, j$, s. t. $i<j$ and $\pi(i)>\pi(j)$;
\item For every permutation $\pi\in S_n$, the set of the locations of the descents is defined to be $$des\left(\pi\right)=\{1\leq i\leq n-1 | \pi(i)>\pi(i+1)\},$$ and $$i\in des\left(\pi\right) ~~if ~and ~only ~if ~~\ell(s_i\cdot \pi)<\ell(\pi)$$ (i.e., $i$ is a descent of $\pi$ if and only if multiplying $\pi$ by $s_i$ in the left side shortens the Coxeter length of the element.);
\item For every permutation $\pi\in S_n$, the major-index is defined to be  $$maj\left(\pi\right)=\sum_{\pi(i)>\pi(i+1)}i$$ (i.e., major-index is the sum of the locations of the descents of $\pi$.).
\end{itemize}
\end{definition}

\subsection{The standard and the dual-standard $OGS$ canonical forms and the exchange laws}

\begin{theorem}\label{canonical-sn}
Let $S_n$ be the symmetric group on $n$ elements. For every $2\leq m\leq n$, define $t_{m}$ to be the product $\prod_{j=1}^{m-1}s_{j}$. The element $t_{m}$ is the permutation $[m;1;\ldots;m-1]$, which is the $m$-cycle $(m, ~m-1, ~\ldots, ~1)$ in the cyclic notation of the permutation. Then, the elements $t_{n}, t_{n-1}, \ldots, t_{2}$ generates $S_n$, and every element of $S_n$ has a unique presentation in each one of the following $OGS$ canonical forms:
\begin{enumerate}
\item $t_{2}^{i_{2}}\cdot t_{3}^{i_{3}}\cdots t_{n}^{i_{n}}$, where $0\leq i_{k}<k$ for $2\leq k\leq n$;
\item  $t_{n}^{i_{n}}\cdot t_{n-1}^{i_{n-1}}\cdots t_{2}^{i_{2}}$, where $0\leq i_{k}<k$ for $2\leq k\leq n$.
\end{enumerate}
\end{theorem}
\begin{proof}
The theorem has been stated, and has been used in \cite{AR}. We give a detailed proof of it.
The proof is by induction on $n$. For $n=2$, we have $S_{2}$ a cyclic group of order $2$, which is generated by $t_2$. Therefore, the theorem holds trivially. Now, assume by induction that the theorem holds for every $m$ such that $2\leq m\leq n-1$. Thus, for each $1\leq m\leq n-1$, every element in $S_{m}$ has a unique presentation in a form  $t_{m}^{i_{m}}\cdot t_{m-1}^{i_{m-1}}\cdots t_{2}^{i_{2}}$, where $0\leq i_{k}<k$ for $2\leq k\leq m$, and every element in $S_{m}$ has a unique presentation also in a form $t_{2}^{i_{2}}\cdot t_{3}^{i_{3}}\cdots t_{m}^{i_{m}}$, where $0\leq i_{k}<k$ for $2\leq k\leq m$. Now, let us consider the group $S_{n}$. There is a subgroup $\dot{S}_{n-1}$ of $S_{n}$ which contains all the elements of $S_{n}$ such that $n$ is a fixed point. It is easy to show that $\dot{S}_{n-1}$ is isomorphic to $S_{n-1}$. Then, by the induction hypothesis, every element in $\dot{S}_{n-1}$ has a unique presentation in a form $t_{n-1}^{i_{n-1}}\cdot t_{n-2}^{i_{n-2}}\cdots t_{2}^{i_{2}}$, where $0\leq i_{k}<k$ for $2\leq k\leq n-1$, and also in a form  $t_{2}^{i_{2}}\cdot t_{3}^{i_{3}}\cdots t_{n-1}^{i_{n-1}}$, where $0\leq i_{k}<k$ for $2\leq k\leq n-1$. The index $[S_{n}:\dot{S}_{n-1}]=n$, and it is easy to show that the $n$ elements $t_{n}^{i_{n}}$, where $0\leq i_{n}<n$ are both left and right cosset representatives of $\dot{S}_{n-1}$ in $S_{n}$. Thus, the theorem holds for $S_{n}$.
\end{proof}

\begin{remark}\label{order-important}

Contrary to the abelian groups and to the dihedral groups, in the case of $S_{n}$ for $n\geq 4$, the order in which the $n-1$  generators $t_{2}, t_{3}, \ldots, t_{n}$ can appear in the $OGS$  canonical form is important, and there is no every ordered  sequence $\langle t_{\pi(2)}, t_{\pi(3)}, \ldots t_{\pi(n)}\rangle$ forms an $OGS$ for $S_n$, where $\pi$ is an arbitrary  permutation of the elements in the set of the $n-1$ integers $\{2, 3, \ldots n\}$. For example, consider the goup $S_{4}$, which is generated by the elements $t_{2}, t_{3}, t_{4}$. Then, it is satisfied that  $t_{4}^{2}\cdot t_{2}=t_{3}\cdot t_{4}$. Thus, there is no unique presentations of the elements of $S_4$ in the form  $t_{3}^{i_{3}}\cdot t_{4}^{i_{4}}\cdot t_{2}^{i_{2}}$, where $0\leq i_{k}<k$. Therefore, the ordered sequence $\langle t_{3}, t_{4}, t_{2}\rangle$ does not form an $OGS$ for $S_n$..
\end{remark}

 As a conclusion  we consider the following two $OGS$ for $S_n$.
  \begin{itemize}
  \item The standard $OGS$ for $S_n$: $OGS(S_n)=\langle t_{2}, t_{3}, \ldots t_{n}\rangle$;
  \item The dual-standard $OGS$ for $S_n$: $OGS(S_n)=\langle t_{n}, t_{n-1}, \ldots t_{2}\rangle$.
  \end{itemize}

   The standard $OGS$ for $S_n$ has been considered in \cite{AR} for combinatorial interest too. By \cite{AR}, the sum of the exponents of the generators according to the sequence, $\sum_{j=2}^{n}i_{j}$,  (i.e., the length of the element $\pi$ in the canonical form according to the standard $OGS$) is the major-index of the permutation $\pi$.

Both the standard and the dual-standard $OGS$ give exchange laws with very interesting properties which we describe now.

\begin{proposition}\label{exchange}
The following holds:
\begin{enumerate}
\item For transforming  the element $t_{q}^{i_{q}}\cdot t_{p}^{i_{p}}$  ($p<q$) onto the $OGS$ canonical form\\ $t_{2}^{i_{2}}\cdot t_{3}^{i_{3}}\cdots t_{n}^{i_{n}}$, i.e., according to the  standard $OGS$, one needs to use the following exchange laws:

 \[ t_{q}^{i_{q}}\cdot t_{p}^{i_{p}}=\begin{cases}
t_{i_{q}+i_{p}}^{i_q}\cdot t_{p+i_{q}}^{i_{p}}\cdot t_{q}^{i_{q}}  & q-i_{q}\geq p \\
\\
t_{i_{q}}^{p+i_{q}-q}\cdot t_{i_{q}+i_{p}}^{q-p}\cdot t_{q}^{i_{q}+i_{p}} & i_{p}\leq q-i_{q}\leq p \\
\\
t_{p+i_{q}-q}^{i_{q}+i_{p}-q}\cdot t_{i_{q}}^{p-i_{p}}\cdot t_{q}^{i_{q}+i_{p}-p}  & q-i_{q}\leq i_{p}.
\end{cases}
\]
\item Similarly, for transforming  the element $t_{p}^{{i_p}'}\cdot t_{q}^{{i_q}'}$  ($p<q$) onto the $OGS$ canonical form $t_{n}^{i_{n}}\cdot t_{n-1}^{i_{n-1}}\cdots t_{2}^{i_{2}}$, i.e., according to dual-standard, one needs to use the following exchange laws:

 \[ t_{p}^{{i_p}'}\cdot t_{q}^{{i_q}'}=\begin{cases}
 t_{q}^{{i_q}'}\cdot t_{p+q-{i_q}'}^{q-{i_q}'+{i_p}'}\cdot t_{q-{i_q}'+p-{i_p}'}^{p-{i_p}'}  & {i_q}'\geq p \\
\\
t_{q}^{{i_q}'+{i_p}'-p}\cdot t_{q-{i_q}'+p-{i_p}'}^{2p-{i_p}'-{i_q}'}\cdot t_{q-{i_q}'}^{q-p}  & p-{i_p}'\leq {i_q}'\leq p \\
\\
t_{q}^{{i_q}'+{i_p}'}\cdot t_{q-{i_q}'}^{q-{i_q}'-{i_p}'}\cdot t_{p-{i_q}'}^{{i_p}'} & {i_q}'\leq p-{i_p}'.
\end{cases}
\]
\end{enumerate}
\end{proposition}
\begin{proof}
 First, consider the standard $OGS$ of $S_n$. Then, we look at the  exchange laws for $t_{q}^{i_{q}}\cdot t_{p}^{i_{p}}$, where $q>p$. Since all \\ $j\in [n]:=\{1, 2, \ldots, n\}$ such that $j>q$ are fixed by $t_{q}^{i_{q}}\cdot t_{p}^{i_{p}}$, we may consider the element $t_{q}^{i_{q}}\cdot t_{p}^{i_{p}}$ in $S_{q}$ instead of considering them in $S_{n}$. Therefore, the operation $"+"$ is  considered addition modulo $q$.
 Now, look at the elements $t_{q}^{i_{q}}$ and $t_{p}^{i_{p}}$ as permutations in $S_{q}$, then the following is satisfied:
 \begin{itemize}
 \item $t_{q}^{i_{q}}(j)=j-i_{q}$ ~~for every $1\leq j\leq q$;
 \item $t_{p}^{i_{p}}(j)=j-i_{p}+p=j-(i_{p}-p+q)$ ~~for every $1\leq j\leq i_{p}$;
 \item $t_{p}^{i_{p}}(j)=j-i_{p}$ ~~for every $i_{p}+1\leq j\leq p$;
 \item $t_{p}^{i_{p}}(j)=j$ ~~for every $p+1\leq j\leq q$.
 \end{itemize}
 Let $x$ be the conjugate of $t_{p}^{i_{p}}$ by the element $t_{q}^{-i_{q}}$, i.e, $x:=t_{q}^{i_{q}}\cdot t_{p}^{i_{p}}\cdot t_{q}^{-i_{q}}$. Then, by conjugating laws of permutations, the  permutation $x$ satisfies the following by considering  the inequalities modulo $q$ (i.e., if $q=10$, then $8\leq j\leq 2$ means the set of integers $\{8, 9, 10, 1, 2\})$:
\begin{itemize}
\item $x(j)=j-(i_{p}-p+q)$ ~for ~$i_{q}+1\leq j\leq i_{q}+i_{p}$;
\item $x(j)=j-i_{p}$ ~for ~$i_{q}+i_{p}+1\leq j\leq i_{q}+p$;
\item $x(j)=j$ ~for ~$i_{q}+p+1\leq j\leq i_{q}$.
\end{itemize}

Thus, by using $x(j)$ we get the following canonical form for $x$:
\begin{itemize}
\item $x=t_{i_{q}+i_{p}}^{i_q}\cdot t_{p+i_{q}}^{i_{p}}$ in case ~$q\geq i_{q}+p$  ~(i.e., $q-i_{q}\geq p$);
\item $x=t_{i_{q}}^{i_{q}-(q-p)}\cdot t_{i_{q}+i_{p}}^{q-p}\cdot t_{q}^{i_{p}}$ in case ~$i_{q}+i_{p}\leq q\leq i_{q}+p$ ~(i.e., $i_{p}\leq q-i_{q}\leq p$);
\item $x=t_{i_{q}-(q-p)}^{i_{q}+i_{p}-q}\cdot t_{i_{q}}^{p-i_{p}}\cdot t_{q}^{i_{p}-p}$ in case ~$q\leq i_{q}+i_{p}$ ~(i.e., $q-i_{q}\leq i_{p}$).
\end{itemize}
Now, right multiplications of each one of the three equalities by $t_{q}^{i_{q}}$ gives the desired exchange laws for $t_{q}^{i_{q}}\cdot t_{p}^{i_{p}}$.

Now, consider the dual-standard $OGS$ of $S_n$.
Since $(t_{q}^{i_{q}}\cdot t_{p}^{i_{p}})^{-1}=t_{p}^{p-i_{p}}\cdot t_{q}^{q-i_{q}}$, we get the exchange laws for $t_{p}^{{i_p}'}\cdot t_{q}^{{i_q}'}$ easily by taking the inverse of the the exchange laws for $t_{q}^{i_{q}}\cdot t_{p}^{i_{p}}$, and substituting ${i_p}'=p-i_{p}$, ${i_q}'=q-i_{q}$.
\end{proof}

From the described exchange laws for $S_n$ we conclude the following properties:

\begin{proposition}\label{exchange-2}

 The standard $OGS$ canonical form of $t_{q}^{i_{q}}\cdot t_{p}^{i_{p}}$, and the dual-standard canonical form of $t_{p}^{{i_p}'}\cdot t_{q}^{{i_q}'}$ (where, $p<q$) are products of non-zero powers of at most three different canonical generators, where the following holds:
   \begin{enumerate}
    \item  The standard $OGS$ canonical form of $t_{q}^{i_{q}}\cdot t_{p}^{i_{p}}$ ($p<q$, $i_{p}>0, i_{q}>0$) is a product of non-zero powers of two different canonical generators if and only if $q-i_{q}=p$ or $q-i_{q}=i_{p}$, and then the following hold:
 \[ t_{q}^{i_{q}}\cdot t_{p}^{i_{p}}=\begin{cases}
  t_{i_{q}+i_{p}}^{i_q}\cdot t_{q}^{i_{q}+i_{p}} & q-i_{q} = p \\
\\
  t_{i_{q}}^{p-i_{p}}\cdot t_{q}^{q-p} & q-i_{q} = i_{p}.
\end{cases}
\]
     \item The dual-standard $OGS$ canonical form of $t_{p}^{{i_p}'}\cdot t_{q}^{{i_q}'}$ ($p<q$, ${i_p}'>0, {i_q}'>0$) is a product of non-zero powers of two different canonical generators if and only if ${i_q}'=p$, or ${i_q}'=p-{i_p}'$, and then the following holds:
 \[ t_{p}^{{i_p}'}\cdot t_{q}^{{i_q}'}=\begin{cases}
t_{q}^{{i_p}'}\cdot t_{q-{i_p}'}^{p-{i_p}'}   & {i_q}' = p \\
\\
t_{q}^{p}\cdot t_{q-{i_q}'}^{q-p}  & {i_q}' = p-{i_p}'.
\end{cases}
\]
\end{enumerate}
\end{proposition}
\begin{proof}
The results of  Proposition \ref{exchange} imply that the standard $OGS$ canonical form of   $t_{q}^{i_{q}}\cdot t_{p}^{i_{p}}$, and the dual-standard $OGS$ canonical form of  $t_{p}^{{i_p}'}\cdot t_{q}^{{i_q}'}$ (where, $p<q$) are product of non-zero powers of at most three different canonical generators. In the cases $q-i_{q}=p$ and in case $q-i_{q}=i_{p}$ the element $t_{q}^{i_{q}}\cdot t_{p}^{i_{p}}$ (where, $p<q$) is a product of a non-zero powers of two different canonical generators, using the exchange laws in case 1. of Proposition \ref{exchange}. Similarly, in cases ${i_q}'=p$ or ${i_q}'=p-{i_p}'$, the element $t_{p}^{{i_p}'}\cdot t_{q}^{{i_q}'}$ (where, $p<q$) is a product of a non-zero powers of two different canonical generators, using the exchange laws in case 2. of  Proposition \ref{exchange}.
\end{proof}

\begin{proposition}
The exchange laws for $S_{n}$ which are described in Propositions \ref{exchange}, \ref{exchange-2} satisfy the following properties:
\begin{enumerate}
\item The standard $OGS$ canonical form of the element $t_{q}^{i_{q}}\cdot t_{p}^{i_{p}}$ for $q>p$ is either $t_{a}^{i_{a}}\cdot t_{b}^{i_{b}}\cdot t_{c}^{i_{c}}$, or $t_{a}^{i_{a}}\cdot t_{b}^{i_{b}}$ for $a<b<c$, where $c=q$ and  all of the numbers: $a$, $b$, $i_{a}$, $i_{b}$, and $i_{c}$, are linear combinations of at most three different elements from $\{p, q, i_{p}, i_{q}\}$ with co-efficients $1$ or $-1$;
\item The dual-standard canonical form of the element $t_{p}^{{i_p}'}\cdot t_{q}^{{i_q}'}$ for $q>p$ is either $t_{a}^{{i_{a}}'}\cdot t_{b}^{{i_{b}}'}\cdot t_{c}^{{i_{c}}'}$, or $t_{a}^{{i_{a}}'}\cdot t_{b}^{{i_{b}}'}$ for $a>b>c$, where $a=q$ and  all of the numbers: $b$, $c$, ${i_{a}}'$, ${i_{b}}'$, and ${i_{c}}'$, are linear combination of at most four elements from $\{p, q, {i_p}', {i_q}'\}$ with co-efficients $1$ or $-1$ or $2$ (The co-efficient $2$ appears just in one of the three cases for ${i_{b}}'$, as a co-efficient of $p$).
\end{enumerate}
\end{proposition}
\begin{proof}
The results come directly from Propositions \ref{exchange} and \ref{exchange-2}.
\end{proof}

\begin{example}
Consider the following elements $x, y\in S_{5}$. Let $x=t_{2}\cdot t_{3}^{2}\cdot t_{4}\cdot t_{5}^{4}$, which is the permutation $[3;2;5;4;1]$ which is the $3$-cycle $(1,3,5)$. Let $y=t_{2}\cdot t_{3}\cdot t_{4}^{2}\cdot t_{5}^{2}$, which is the permutation $[1;4;2;5;3]$, which is the $4$-cycle $(2, 4, 5, 3)$. Now, we find the standard $OGS$ canonical form of the product $\pi=x\cdot y$ by using the exchange laws described in Propositions \ref{exchange}, and \ref{exchange-2}:
\begin{align*}
x\cdot y&=t_{2}\cdot t_{3}^{2}\cdot t_{4}\cdot (t_{5}^{4}\cdot t_{2})\cdot t_{3}\cdot t_{4}^{2}\cdot t_{5}^{2} \\
&=t_{2}\cdot t_{3}^{2}\cdot t_{4}^{2}\cdot (t_{5}^{3}\cdot t_{3})\cdot t_{4}^{2}\cdot t_{5}^{2}\\
&=t_{2}\cdot t_{3}^{2}\cdot t_{4}^{2}\cdot t_{3}\cdot t_{4}^{2}\cdot (t_{5}^{4}\cdot t_{4}^{2})\cdot t_{5}^{2} \\
&=t_{2}\cdot t_{3}^{2}\cdot t_{4}^{2}\cdot t_{3}\cdot (t_{4}^{2}\cdot t_{3})\cdot t_{4}^{2}\cdot t_{5}^{4} \\
&=t_{2}\cdot t_{3}^{2}\cdot t_{4}^{2}\cdot (t_{3}\cdot t_{2})\cdot t_{3}\cdot t_{4}\cdot t_{5}^{4} \\
&=t_{2}\cdot t_{3}^{2}\cdot t_{4}^{2}\cdot t_{2}\cdot (t_{3}^{2}\cdot t_{3})\cdot t_{4}\cdot t_{5}^{4} \\
&=t_{2}\cdot t_{3}^{2}\cdot (t_{4}^{2}\cdot t_{2})\cdot t_{4}\cdot t_{5}^{4} \\
&=t_{2}\cdot (t_{3}^{2}\cdot t_{3}^{2})\cdot (t_{4}^{3}\cdot t_{4})\cdot t_{5} \\
&=t_{2}\cdot t_{3}\cdot t_{5}^{4}.
\end{align*}
$\pi=x\cdot y$ is the permutation $[2;4;3;5;1]$, which is the $4$-cycle $(1, 2, 4, 5)$.
\end{example}

The next proposition shows the standard $OGS$ canonical form of some conjugates of $t_{k}$ and $t_{k}^{-1}$.

\begin{proposition}\label{exchange-conj}
Let $G=S_{n}$ and consider the standard $OGS$ canonical form.
Then, for every $2\leq k\leq n-1$ the following holds:

$$\left(\prod_{j=0}^{n-k-1}t_{n-j}\right)\cdot t_{k}\cdot \left(\prod_{j=0}^{n-k-1}t_{n-j}\right)^{-1}=t_{n-k+1}^{-1}\cdot t_{n}.$$

$$\left(\prod_{j=0}^{n-k-1}t_{n-j}\right)\cdot t_{k}^{-1}\cdot \left(\prod_{j=0}^{n-k-1}t_{n-j}\right)^{-1}=t_{n-1}^{n-k}\cdot t_{n}^{k-1}.$$
\end{proposition}

\begin{proof}
The proof is in induction on $n-k$. If $n-k=1$ (i.e., $k=n-1$, then by Proposition \ref{exchange-2},
$$t_{n}\cdot t_{n-1}\cdot t_{n}^{-1}=t_{2}\cdot t_{n}^{2}\cdot t_{n}^{-1}=t_{2}\cdot t_{n}.$$
Assume in induction that the proposition holds for every $k$ such that $n-k<k'$ for some $k'\geq 1$ and we prove it for $n-k=k'$. Then, by the induction hypothesis, the following is satisfied:
$$\left(\prod_{j=0}^{n-k'-1}t_{n-j}\right)\cdot t_{k'}\cdot \left(\prod_{j=0}^{n-k-1}t_{n-k}\right)^{-1}=t_{n}\cdot t_{n-k'}^{-1}\cdot t_{n-1}\cdot t_{n}^{-1}.$$
Then, by Propositions \ref{exchange} and \ref{exchange-2} the following holds:
\begin{align*}
t_{n}\cdot t_{n-k'}^{n-k'-1}\cdot t_{n-1}\cdot t_{n}^{-1}&=t_{n-k'}\cdot t_{n-k'+1}^{n-k'-1}\cdot t_{n}\cdot t_{n-1}\cdot t_{n}^{-1} \\ \\&=
t_{n-k'}\cdot t_{n-k'+1}^{n-k'-1}\cdot t_{2}\cdot t_{n}^{2}\cdot t_{n}^{-1} \\ \\&= t_{n-k'}\cdot t_{n-k'}^{n-k'-1}\cdot t_{n-k'+1}^{n-k'}\cdot t_{n} \\ \\&=
t_{n-k'+1}^{-1}\cdot t_{n}.
\end{align*}
Thus, the first part of the proposition holds for every $k$.
Now,
\begin{align*}
\left(\prod_{j=0}^{n-k-1}t_{n-j}\right)\cdot t_{k}^{-1}\cdot \left(\prod_{j=0}^{n-k-1}t_{n-j}\right)^{-1}&=\left(\left(\prod_{j=0}^{n-k-1}t_{n-j}\right)\cdot t_{k}\cdot \left(\prod_{j=0}^{n-k-1}t_{n-j}\right)^{-1}\right)^{-1} \\ &=\left(t_{n-k+1}^{-1}\cdot t_{n}\right)^{-1}=t_{n}^{-1}\cdot t_{n-k+1}.
\end{align*}
By Proposition \ref{exchange-2},
$$t_{n}^{-1}\cdot t_{n-k+1}=t_{n}^{n-1}\cdot t_{n-k+1}=t_{n-1}^{n-k}\cdot t_{n}^{k-1}.$$
\end{proof}

\subsection{Normal form, and its connection to the standard $OGS$ canonical form}

From now on, we consider just the standard $OGS$ canonical form.
\vskip 0.5cm
In this subsection, we recall the definition of  a  normal form of elements of $S_n$ in Coxeter generators, which is described in detail in  Brenti and Bjorner's book,  "Combinatorics of Coxeter groups" \cite{BB}. Then, we find a connection of the normal form to the standard $OGS$ canonical form.

\begin{definition}\label{brenti-normal}
By \cite{BB}, every element $\pi$ of $S_n$ can be presented uniquely in the following normal reduced form, which we denote by $norm(\pi)$:
$$norm(\pi)=\prod_{u=1}^{n-1}\prod_{r=0}^{y_{u}-1}s_{u-r}.$$
such that $y_u$ is a non-negative integer where, $0\leq y_u\leq u$ for every $1\leq u\leq n-1$.

We denote by $\ell(\pi)$, the Coxeter length of an element $\pi\in S_n$, which is the number of Coxeter generators $s_{j}$ which are used in the reduced presentation of $\pi$. By our notation of $norm(\pi)$, $$\ell(\pi)=\sum_{u=1}^{n-1}y_{u}.$$
\end{definition}

\begin{example}
Let  $m=8$, $y_2=2$, $y_4=3$, $y_5=1$, $y_8=4$, and $y_1=y_3=y_6=y_7=0$, then
$$norm(\pi)=(s_2\cdot s_1)\cdot (s_4\cdot s_3\cdot s_2)\cdot s_5\cdot (s_8\cdot s_7\cdot s_6\cdot s_5).$$
$$\ell(\pi)=2+3+1+4=10.$$
\end{example}
Notice, $t_{k}=\prod_{u=1}^{k-1}s_{u}$ is already presented in the mentioned normal form according to \cite{BB}, with $y_{u}=1$ for every $1\leq u\leq k-1$.

\begin{proposition}\label{t-power}
Consider the symmetric group $S_n$. then for every $2\leq k\leq n$, and \\ $1\leq i_{k}\leq k-1$ the following holds:

$$norm(t_{k}^{i_k})=\prod_{u=i_{k}}^{k-1}\prod_{r=0}^{i_{k}-1}s_{u-r}.$$

$$\ell(t_{k}^{i_{k}})=k\cdot i_{k}-i_{k}^{2}.$$

\end{proposition}

\begin{proof}
 First, notice, by the definition of $t_{i_{k+1}}$, we conclude:  $\prod_{r=0}^{i_{k}-1}s_{i_{k}-r}=t_{i_{k}+1}^{-1}=t_{i_{k}+1}^{i_{k}}$.
 Notice also, $\prod_{r=0}^{i_{k}-1}s_{u-r}$ for $i_{k}<u\leq k-1$, is a conjugate of $\prod_{r=0}^{i_{k}-1}s_{i_{k}-r}$ by $\prod_{u'=i_{k}+2}^{u+1}t_{u'}^{-1}$. Hence, $$\prod_{u=i_{k}}^{k-1}\prod_{r=0}^{i_{k}-1}s_{u-r}=t_{i_{k}+1}^{-1}\cdot \prod_{u=i_{k}+2}^{k}\left(\prod_{r=0}^{u-i_{k}-2}t_{u-r}\cdot t_{i_{k}+1}^{-1}\cdot \left(\prod_{r=0}^{u-i_{k}-2}t_{u-r}\right)^{-1}\right).$$
 Now, assume in induction that $$t_{i_{k}+1}^{-1}\cdot \prod_{u=i_{k}+2}^{k-1}\left(\prod_{r=0}^{u-i_{k}-2}t_{u-r}\cdot t_{i_{k}+1}^{-1}\cdot \left(\prod_{r=0}^{u-i_{k}-2}t_{u-r}\right)^{-1}\right)=t_{k-1}^{i_{k}}.$$
 Then,
 \begin{align*}
 &t_{i_{k}+1}^{-1}\cdot \prod_{u=i_{k}+2}^{k}\left(\prod_{r=0}^{u-i_{k}-2}t_{u-r}\cdot t_{i_{k}+1}^{-1}\cdot \left(\prod_{r=0}^{u-i_{k}-2}t_{u-r}\right)^{-1}\right) \\&=t_{k-1}^{i_{k}}\cdot\left(\prod_{r=0}^{k-i_{k}-2}t_{k-r}\cdot t_{i_{k}+1}^{-1}\cdot \left(\prod_{r=0}^{k-i_{k}-2}t_{k-r}\right)^{-1}\right).
 \end{align*}
 Then, by Proposition \ref{exchange-conj},
 $$\left(\prod_{r=0}^{k-i_{k}-2}t_{k-r}\cdot t_{i_{k}+1}^{-1}\cdot \left(\prod_{r=0}^{k-i_{k}-2}t_{k-r}\right)^{-1}\right)=t_{k-1}^{k-(i_{k}+1)}\cdot t_{k}^{i_{k}}.$$
Hence,
$$\prod_{u=i_{k}}^{k-1}\prod_{r=0}^{i_{k}-1}s_{u-r}=t_{k-1}^{i_{k}}\cdot t_{k-1}^{k-(i_{k}+1)}\cdot t_{k}^{i_{k}}=t_{k}^{i_{k}}.$$
 Then, by using $norm(t_{k}^{i_k})$, we get $\ell(t_{k}^{i_k})=(k-i_{k})\cdot i_{k}=k\cdot i_{k}-i_{k}^{2}$.
\end{proof}

\begin{example}
$$t_7^{3}=(s_3\cdot s_2\cdot s_1)\cdot (s_4\cdot s_3\cdot s_2)\cdot (s_5\cdot s_4\cdot s_3)\cdot (s_6\cdot s_5\cdot s_4).$$

$$\ell(t_{7}^{3})=7\cdot 3-3^2=12.$$

\end{example}

\begin{proposition}\label{s-r-k-v}
Assume $norm(\pi)=\prod_{r=0}^{v-1}s_{k-r}$, where $k, v$ are positive integers, such that $v\leq k$,  then the standard $OGS$ canonical form of $\pi$ is the following:
\begin{itemize}
\item  $\pi=t_{k}^{k-v}\cdot t_{k+1}^{v}$ in case $v<k$ \\ (i.e., ~$norm(\pi)=s_k\cdot s_{k-1}\cdots s_{k-v+1}$, where $k-v+1\geq 2$).
\item $\pi=t_{k+1}^{k}$ in case $v=k$ \\ (i.e., ~$norm(\pi)=s_k\cdot s_{k-1}\cdots s_1$).
\end{itemize}
\end{proposition}

\begin{proof}
Assume $\pi=t_{k+1}^{k}$. By Proposition \ref{t-power}, $t_{k+1}^{k}=\prod_{r=0}^{k-1}s_{k-r}$. \\ Now, assume $\pi=t_{k}^{k-v}\cdot t_{k+1}^{v}$. Then, $\pi=(t_{k}^{v})^{-1}\cdot t_{k+1}^{v}$. By using Proposition \ref{t-power},
$$norm(t_{k}^{v})=(s_{v}\cdot s_v\cdots s_1)\cdot (s_{v+1}\cdot s_{v+1}\cdots s_2)\cdots (s_{k-1}\cdot s_{k-2}\cdots s_{k-v})$$
and $$norm(t_{k+1}^{v})=(s_{v}\cdot s_v\cdots s_1)\cdot (s_{v+1}\cdot s_{v+1}\cdots s_2)\cdots (s_{k-1}\cdot s_{k-2}\cdots s_{k-v})\cdot (s_k\cdot s_{k-1}\cdots s_{k-v+1}).$$
Thus, $$norm(t_{k+1}^{v})=norm(t_{k}^{v})\cdot (s_k\cdot s_{k-1}\cdots s_{k-v+1}).$$
Hence, $$norm(\pi)=norm(t_{k}^{-(v)}\cdot t_{k+1}^{v})=\prod_{r=0}^{v-1}s_{k-r}.$$
\end{proof}

\begin{example}
Assume, $norm(\pi)=s_5\cdot s_4\cdot s_3$, then $\pi=t_{5}^{2}\cdot t_{6}^{3}$ in the standard $OGS$ canonical form.
\end{example}

\begin{proposition}\label{s-r-k-v-2}
Let $\pi=t_{k_1}^{k_1-v}\cdot t_{k_2}^{v}$ be an element of $S_n$, which is presented in the standard $OGS$ canonical form, where $v$ is a positive integer such that $1\leq v\leq k_1-1$. Then,
 $$norm(\pi)=norm(t_{k_1}^{k_1-v}\cdot t_{k_2}^{v})=\prod_{u=k_1}^{k_2-1}\prod_{r=0}^{v-1}s_{u-r}.$$

$$\ell(\pi)=(k_{2}-k_{1})\cdot v.$$
\end{proposition}

\begin{proof}
Let $\pi=t_{k_1}^{k_1-v}\cdot t_{k_2}^{v}$ be an element of $S_n$, which is presented in the standard $OGS$ canonical form. Then, $\pi=\prod_{u=k_1}^{k_2-1}t_{u}^{-v}\cdot t_{u+1}^{v}$. By proposition \ref{s-r-k-v}, \\ $norm(t_{u}^{-v}\cdot t_{u+1}^{v})=\prod_{r=0}^{v-1}s_{u-r}$. Therefore, $norm(\prod_{u=k_1}^{k_2-1}t_{u}^{-v}\cdot t_{u+1}^{v})=\prod_{u=k_1}^{k_2-1}\prod_{r=0}^{v-1}s_{u-r}$.
Therefore, obviously, $\ell(\pi)=(k_{2}-k_{1})\cdot v.$
\end{proof}
\\

The next Proposition describes the connection between the normal form of $S_n$ in Coxeter generators, and the standard $OGS$ canonical form of $S_n$.

\begin{proposition}\label{normal-canonical}
Let $\pi\in S_n$, such that $norm(\pi)=\prod_{u=1}^{n-1}\prod_{r=0}^{y_{u}-1}s_{u-r}$, where \\ $0\leq y_{u}\leq u$ is a non negative integer ($y_{u}=0$ means $norm(\pi)$ does not contain any segment of decreasing indices starting with $s_{u}$). Then, the standard $OGS$ canonical form of $\pi$ is:
$$\prod_{j=2}^{n}t_{j}^{i_{j}},$$
such that:
\begin{itemize}
\item If $y_{j}\leq y_{j-1}$, then $i_{j}=y_{j-1}-y_{j}$;
\item If $y_{j}>y_{j-1}$, then $i_{j}=j+y_{j-1}-y_{j}$;
\item $i_{n}=y_{n-1}$ (We may assume $y_{n}=0$, and using $i_{n}=y_{n-1}-y_{n}$).
\end{itemize}
\end{proposition}

\begin{proof}
Let $\pi\in S_n$, such that $norm(\pi)=\prod_{u=1}^{n-1}\prod_{r=0}^{y_{u}-1}s_{u-r}$, where $0\leq y_{u}\leq u$ is a non negative integer. Then, by Proposition \ref{s-r-k-v}, the following holds:
\begin{itemize}
\item If $0<y_{u}<u$, then $\prod_{r=0}^{y_{u}-1}s_{u-r}=t_{u}^{u-y_{u}}\cdot t_{u+1}^{y_{u}}$;
\item If $y_{u}=u$, then $\prod_{r=0}^{y_{u}-1}s_{u-r}=\prod_{r=0}^{u-1}s_{u-r}=t_{u+1}^{u}$;
\end{itemize}
Thus, by substituting instead of $\prod_{r=0}^{y_{u}-1}s_{u-r}$, the suitable presentation in canonical form, and by using $t_{u}^{u}=1$ for $2\leq u\leq n$, we get the desired result.
\end{proof}

\begin{example}
Let $\pi\in S_9$ with the following normal form in Coxeter generators:
$$norm(\pi)=s_1\cdot (s_3\cdot s_2)\cdot (s_4\cdot s_3\cdot s_2)\cdot (s_6\cdot s_5\cdot s_4\cdot s_3)\cdot (s_7\cdot s_6)\cdot (s_8\cdot s_7).$$
Then, $y_{1}=1$, $y_2=0$, $y_3=2$, $y_4=3$, $y_5=0$, $y_6=4$, $y_7=2$, $y_8=2$.
Thus, the standard $OGS$ canonical form of $\pi$ is the following:
\begin{align*}
\pi&=t_{2}^{1}\cdot t_{3}^{3-2}\cdot t_{4}^{4-1}\cdot t_{5}^{3}\cdot t_{6}^{6-4}\cdot t_{7}^{2}\cdot t_{8}^{0}\cdot t_{9}^{2} \\ &=t_{2}\cdot t_{3}\cdot t_{4}^{3}\cdot t_{5}^{3}\cdot t_{6}^{2}\cdot t_{7}^{2}\cdot t_{9}^{2}.
\end{align*}
\end{example}

\begin{conclusion}
Let $\omega_{n}$ be the longest element of $S_{n}$ according to the Coxeter presentation (i.e. $norm(\omega_{n})=\prod_{u=1}^{n-1}\prod_{y=0}^{u-1}s_{u-r}$). Then, the following hold:
\begin{itemize}
\item $\omega_{n}=\prod_{j=2}^{n}t_{j}^{j-1}$ in the presentation by the standard $OGS$ canonical form;
\item The subgroup of $S_{n}$ which is generated by $\omega_{n}$ and $\omega_{n-1}$ is the Coxeter group $I_{2}(n)$, where the following holds:
\begin{itemize}
\item $$b=\prod_{j=2}^{n-1}t_{j}^{j-1}, ~~ a=t_{n},$$
form a standard $OGS$ for $I_{2}(n)$, such that the presentation of $a^{i}$ and of $b\cdot a^{i}$ for $0\leq i<n$, in the standard $OGS$ canonical form of $S_{n}$ is as follows:
$$a^{i}=t_{n}^{i}, ~~ b\cdot a^{i}=\omega_{n-1}=\prod_{j=2}^{n-1}t_{j}^{j-1}t_{n}^{i};$$
\item $$b'=\prod_{j=2}^{n}t_{j}^{j-1}, ~~ a'=t_{n}^{-1},$$
form a standard $OGS$ for $I_{2}(n)$, such that the presentation of ${a'}^{i}$ and of $b'\cdot {a'}^{i}$ for $0\leq i<n$, in the standard $OGS$ canonical form of $S_{n}$ is as follows:
$${a'}^{i}=t_{n}^{n-i}, ~~ b'\cdot {a'}^{i}=\omega_{n}=\prod_{j=2}^{n-1}t_{j}^{j-1}t_{n}^{n-1-i}.$$
\end{itemize}
\end{itemize}
\end{conclusion}

\begin{proof}
Since $norm(\omega_{n})=\prod_{u=1}^{n-1}\prod_{y=0}^{u-1}s_{u-r}$, the standard $OGS$ canonical form of $\omega_{n}$ comes directly from Proposition \ref{normal-canonical}. By considering the permutation presentation of $\omega_{n}$ and $\omega_{n-1}$, we get by Remark \ref{symmetry-dihedral}, $\omega_{n-1}$ and $\omega_{n}$ generates the Coxeter group $I_{2}(n)$. Then, by using Propositions \ref{cox-ogs}, \ref{dihedral-order} we get the desired standard $OGS$ canonical forms for $I_{2}(n)$.
\end{proof}

\subsection{Standard $OGS$ elementary factorization, and the Coxeter length}

In this subsection, we define standard $OGS$ elementary elements, and the standard $OGS$ elementary factorization of elements of $S_n$ onto a product of standard $OGS$ elementary factors, which we need to describe the Coxeter length of an arbitrary $\pi\in S_n$, which is presented in the standard $OGS$ canonical form.
The standard $OGS$ elementary factorization is used for the algorithm of the standard $OGS$ canonical form of the inverse element $\pi^{-1}$ for an arbitrary $\pi\in S_n$ as well, which we will show in the next subsection.

\begin{definition}\label{elementary}
Let $\pi\in S_n$, where $\pi=\prod_{j=1}^{m}t_{k_{j}}^{i_{k_{j}}}$ is presented in the standard $OGS$ canonical form, with $i_{k_{j}}>0$ for every $1\leq j\leq m$. Then, $\pi$ is called standard $OGS$ elementary element of $S_n$, if
$$\sum_{j=1}^{m}i_{k_{j}}\leq k_{1}.$$
\end{definition}

\begin{definition}\label{rho}
Let $\pi=\prod_{j=1}^{m}t_{k_{j}}^{i_{k_{j}}}$  a standard  $OGS$ elementary  element of $S_n$ which is presented in the standard $OGS$ canonical form, with $i_{k_{j}}>0$ for every $1\leq k\leq m$. Then, for every $1\leq j\leq m$, $\rho_{j}$, $\varrho_{j}$, and $\vartheta_{j}$ are defined to be as follows:
\begin{itemize}
\item $\rho_{j}=\sum_{x=j}^{m}i_{k_{x}}$;
\item $\varrho_{j}=\sum_{x=1}^{j}i_{k_{x}}$;
\item $\vartheta_{j}=k_{j}-\rho_{1}=k_{j}-\varrho_{m}$.
\end{itemize}
\end{definition}

\begin{remark}\label{maj}
Let $\pi=\prod_{j=1}^{m}t_{k_{j}}^{i_{k_{j}}}$  a standard  $OGS$ elementary  element of $S_n$ which is presented in the standard $OGS$ canonical form, with $i_{k_{j}}>0$ for every $1\leq j\leq m$, let $\rho_{j}$ and $\varrho_{j}$ be integers as defined in Definition \ref{rho}. Then, by \cite{AR}, the parameters $\rho$  and $\varrho$ are connected to the major-index of  $\pi$, such that:
$$\rho_{1}=\varrho_{m}=maj\left(\pi\right).$$
Thus, we often write $maj\left(\pi\right)$ instead of $\rho_{1}$ or instead of $\varrho_{m}$, which appears several times in the paper, especially in the algorithm of the presentation of the inverse element in the standard $OGS$ canonical form.
In particular, an element $\pi=\prod_{j=1}^{m}t_{k_{j}}^{i_{k_{j}}}$, which is presented in the standard $OGS$ canonical form, with $i_{k_{j}}>0$ for $1\leq j\leq m$, is a standard $OGS$ elementary element if and only if
$$maj\left(\pi\right)\leq k_{1}.$$

\end{remark}

\begin{proposition}\label{order-rho}
Let $\pi=\prod_{j=1}^{m}t_{k_{j}}^{i_{k_{j}}}$  a standard  $OGS$ elementary  element of $S_n$ which is presented in the standard $OGS$ canonical form, with $i_{k_{j}}>0$ for every $1\leq j\leq m$, let $\rho_{j}$, $\varrho_{j}$ and $\vartheta_{j}$ be integers as defined in Definition \ref{rho}. Then, $1\leq j_{1}<j_{2}\leq m$ if and only if the following holds:
\begin{itemize}
\item $\rho_{j_{1}}>\rho_{j_{2}}$;
\item $\varrho_{j_{1}}<\varrho_{j_{2}}$;
\item $\vartheta_{j_{1}}<\vartheta_{j_{2}}$.
\end{itemize}
Moreover, for every $1\leq j\leq m$, we have:
\begin{itemize}
\item $\rho_{j}>0$;
\item $\varrho_{j}>0$;
\item $\vartheta\geq 0$;
\item $\vartheta_{j}=0$ ~if and only if ~$j=1$ ~and ~$k_{1}=maj\left(\pi\right)$;
\item If $\vartheta_{1}=0$, then $m>1$.
\end{itemize}
\end{proposition}

\begin{proof}
The proof comes directly from the definition of $\rho_{j}$, $\varrho_{j}$, and $\vartheta_{j}$.
\end{proof}

\begin{theorem}\label{basic-canonical}
Let $\pi=\prod_{j=1}^{m}t_{k_{j}}^{i_{k_{j}}}$ be a standard $OGS$ elementary element of $S_n$, presented in the standard $OGS$ canonical form, with $i_{k_{j}}>0$ for every $1\leq j\leq m$. Then, the following are satisfied:
 \begin{itemize}
\item  $$\pi=\begin{cases} t_{k_{1}}^{\rho_{1}}\cdot (t_{k_{1}}^{k_{1}-\rho_{2}}\cdot t_{k_{2}}^{\rho_{2}})\cdot (t_{k_{2}}^{k_{2}-\rho_{3}}\cdot t_{k_{3}}^{\rho_{3}})\cdots (t_{k_{m-1}}^{k_{m-1}-\rho_{m}}\cdot t_{k_{m}}^{\rho_{m}}) &  ~~ k_{1}>\rho_{1} \\  \\ (t_{k_{1}}^{k_{1}-\rho_{2}}\cdot t_{k_{2}}^{\rho_{2}})\cdot (t_{k_{2}}^{k_{2}-\rho_{3}}\cdot t_{k_{3}}^{\rho_{3}})\cdots (t_{k_{m-1}}^{k_{m-1}-\rho_{m}}\cdot t_{k_{m}}^{\rho_{m}}) &  ~~ k_{1}=\rho_{1}\end{cases}.$$
\item  $$norm(\pi)=\prod_{u=\rho_1}^{k_1-1}\prod_{r=0}^{\rho_1-1}s_{u-r}\cdot \prod_{u=k_1}^{k_2-1}\prod_{r=0}^{\rho_2-1}s_{u-r}\cdot \prod_{u=k_2}^{k_3-1}\prod_{r=0}^{\rho_3-1}s_{u-r}\cdots \prod_{u=k_{m-1}}^{k_m-1}\prod_{r=0}^{\rho_m-1}s_{u-r},$$
for $1\leq x\leq m$;
\item $$\ell(\pi)=\sum_{j=1}^{m}k_{j}\cdot i_{k_{j}}-(i_{k_{1}}+i_{k_{2}}+\cdots +i_{k_{m}})^{2}=\sum_{j=1}^{m}k_{j}\cdot i_{k_{j}}-\left(maj\left(\pi\right)\right)^{2};$$
\item Every sub-word of $\pi$ is a standard $OGS$ elementary element too. In particular, for every two sub-words $\pi_{1}$ and $\pi_{2}$ of $\pi$, such that $\pi=\pi_{1}\cdot \pi_{2}$, it is satisfied: $$\ell(\pi)=\ell(\pi_{1}\cdot \pi_{2})<\ell(\pi_{1})+\ell(\pi_{2});$$
\item $$\ell(s_r\cdot \pi)=\begin{cases} \ell(\pi)-1 & r=\sum_{j=1}^{m}i_{k_{j}} \\
\ell(\pi)+1 & r\neq \sum_{j=1}^{m}i_{k_{j}} \end{cases}.$$
i.e., $des\left(\pi\right)$ contains just one element, which means $des\left(\pi\right)=\{maj\left(\pi\right)\}$.
\end{itemize}
\end{theorem}

\begin{proof}
Let $\pi=\prod_{j=1}^{m}t_{k_{j}}^{i_{k_{j}}}\in S_n$, s.t. $k_{1}<k_{2}<\ldots, <k_{m}$, and $i_{k_{j}}>0$ for $1\leq j\leq m$, be a standard $OGS$ elementary element. For $1\leq x\leq m$, let  $\rho_x=\sum_{j=x}^{m}i_{k_{j}}$. Then, we get $i_{k_{x}}=\rho_{x}-\rho_{x+1}$ for every $1\leq x\leq m-1$ and $\rho_{m}=i_{k_{m}}$. Thus, $$\pi=\begin{cases} t_{k_{1}}^{\rho_{1}}\cdot (t_{k_{1}}^{k_{1}-\rho_{2}}\cdot t_{k_{2}}^{\rho_{2}})\cdot (t_{k_{2}}^{k_{2}-\rho_{3}}\cdot t_{k_{3}}^{\rho_{3}})\cdots (t_{k_{m-1}}^{k_{m-1}-\rho_{m}}\cdot t_{k_{m}}^{\rho_{m}}) &  ~~ k_{1}>\rho_{1} \\  \\ (t_{k_{1}}^{k_{1}-\rho_{2}}\cdot t_{k_{2}}^{\rho_{2}})\cdot (t_{k_{2}}^{k_{2}-\rho_{3}}\cdot t_{k_{3}}^{\rho_{3}})\cdots (t_{k_{m-1}}^{k_{m-1}-\rho_{m}}\cdot t_{k_{m}}^{\rho_{m}}) &  ~~ k_{1}=\rho_{1}\end{cases}.$$
 Now we turn to the proof of the second part of the proposition. By applying Proposition \ref{t-power} for the normal form of the sub-word $t_{k_{1}}^{\rho_{1}}$ and applying Proposition \ref{s-r-k-v-2}, for the normal form of every sub-word $t_{k_{x}}^{-\rho_{x+1}}\cdot t_{k_{x+1}}^{\rho_{x+1}}$, we get the desired normal form for $\pi$ according to \cite{BB}.
Now, we turn to the proof of the third part of the proposition.
By the formula of $norm(\pi)$:
\begin{align*}
\ell(\pi)&=\left(k_{1}-\rho_{1}\right)\cdot \rho_{1}+\sum_{j=2}^{m}\left(k_{j}-k_{j-1}\right)\cdot \rho_{j} \\ &= \sum_{j=1}^{m-1}k_{j}\cdot\left(\rho_{j}-\rho_{j+1}\right)+k_{m}\cdot \rho_{m}-\rho_{1}^{2} \\ &= \sum_{j=1}^{m}k_{j}\cdot i_{k_{j}}-\left(maj\left(\pi\right)\right)^{2}.
\end{align*}
Now, we turn to the proof of the forth part of the proposition. Assume $\pi_{1}$ and $\pi_{2}$ are two sub-words of $\pi$, such that $\pi=\pi_{1}\cdot \pi_{2}$. Then, the standard $OGS$ presentation of $\pi_{1}$ and $\pi_{2}$ as follows:
$$\pi_{1}=\prod_{j=1}^{w-1}t_{k_{j}}^{i_{k_{j}}}\cdot t_{k_{w}}^{i'_{k_{w}}} ~~~~~~ \pi_{2}=t_{k_{w}}^{i''_{k_{w}}}\cdot\prod_{j=w+1}^{m}t_{k_{j}}^{i_{k_{j}}},$$
where, $1\leq w\leq m$, and $i'_{k_{w}}+i''_{k_{w}}=i_{k_{w}}$.
Obviously, $$maj\left(\pi_{1}\right)=\sum_{j=1}^{w-1}i_{k_{j}}+i'_{k_{w}}\leq maj\left(\pi\right)\leq k_{1},$$ $$maj\left(\pi_{2}\right)=\sum_{j=w+1}^{m}i_{k_{j}}+i''_{k_{w}}\leq maj\left(\pi\right)\leq k_{1}\leq k_{w}.$$ Thus, $\pi_{1}$ and $\pi_{2}$ are standard $OGS$ elementary elements too. Since $\pi=\pi_{1}\cdot \pi_{2}$, obviously, $maj\left(\pi\right)=maj\left(\pi_{1}\right)+maj\left(\pi_{2}\right)$. Thus,
$$\ell(\pi_{1})+\ell(\pi_{2})=\sum_{j=1}^{m}k_{j}\cdot i_{k_{j}}-\left(maj\left(\pi_{1}\right)\right)^{2}-\left(maj\left(\pi_{2}\right)\right)^{2}>\sum_{j=1}^{m}k_{j}\cdot i_{k_{j}}-\left(maj\left(\pi\right)\right)^{2}=\ell(\pi).$$
Now, we turn to the proof of the last part of the proposition.
Recall, $t_{j}=s_1\cdot s_2\cdots s_{j-1}$, therefore $$t_{j}(p)=\begin{cases} j & p=1 \\ p-1 & 2\leq p\leq j \\ p  & p\geq j+1\end{cases}.$$
Hence, by using $\rho_{1}\leq k_{1}$, the following holds:
\begin{itemize}
\item $\pi(p)=k_{1}-\rho_{1}+p$, for $1\leq p\leq i_{k_{1}}$;
\item $\pi(p)=k_{q}-\rho_{1}+p$, for $2\leq q\leq m$ and $\varrho_{q-1}+1\leq p\leq \varrho_{q}$;
\item In particular, $\pi(\rho_{1})=\pi(\varrho_{m})=k_{m}$;
\item $\pi(p)=p-\rho_{1}$, for $\rho_{1}+1\leq p\leq k_{1}$;
\item $\pi(p)=p-\rho_{q}$, for $2\leq q\leq m$ and $k_{q-1}\leq p\leq k_{q}$;
\item $\pi(p)=p$, for $k_{m}+1\leq p\leq n$.
\end{itemize}
 We have  $s_j$ shorten the length of $\pi$ if and only if $j\in des\left(\pi\right)$. By the observation of $\pi(p)$ for $r+1\leq p\leq n$ and by Proposition \ref{order-rho}, ~$\pi(\rho_{1}+1)<\pi(\rho_{1}+2)<\ldots <\pi(n)$. Since $k_{1}<k_{2}<\ldots, <k_{m}$, it follows $\pi(1)<\pi(2)<\ldots <\pi(\rho_1)$. Therefore, $des\left(\pi\right)$ contains $\rho_{1}$ only, and thus, $\ell(s_{\rho_{1}}\cdot \pi)=\ell(\pi)-1$, and $\ell(s_{j}\cdot \pi)=\ell(\pi)+1$, for $j\neq \rho_{1}$.

\end{proof}

\begin{proposition}\label{exchange-elementary1}
Let $\pi=\prod_{j=1}^{m}t_{k_{j}}^{i_{k_{j}}}$ be a standard $OGS$ elementary element of $S_n$, presented in the standard $OGS$ canonical form, with $i_{k_{j}}>0$ for every $1\leq j\leq m$. Let $t_{x}^{i_{x}}\in S_{n}$ such that $x\leq k_{1}-maj\left(\pi\right)$. Then, the presentation of $\pi\cdot t_{x}^{i_{x}}$ in the standard $OGS$ canonical form is as follows:
$$\pi\cdot t_{x}^{i_{x}}=t_{\rho_{1}+i_{x}}^{\rho_{1}}\cdot t_{\rho_{1}+x}^{i_{x}}\cdot \pi.$$
\end{proposition}

\begin{proof}
The proof is in induction on $m$. For the case $m=1$, the proposition holds by using Proposition \ref{exchange}, where considering $t_{q}^{i_{q}}\cdot t_{p}^{i_{p}}$ in case $q-i_{q}\geq p$. Assume in induction that the proposition holds for $m'<m$, and we prove it for $m'=m$.
Thus, by the induction hypothesis:
$$t_{k_{1}}^{i_{k_{1}}}\cdot \prod_{j=2}^{m}t_{k_{j}}^{i_{k_{j}}}\cdot t_{x}^{i_{x}}=t_{k_{1}}^{i_{k_{1}}}\cdot t_{\rho_{2}+i_{x}}^{\rho_{2}}\cdot t_{\rho_{2}+x}^{i_{x}}\cdot \prod_{j=2}^{m}t_{k_{j}}^{i_{k_{j}}}.$$
By our assumption, $k_{1}-\rho_{1}\geq x$ and $x>i_{x}$. Therefore, $k_{1}-i_{k_{1}}>\rho_{2}+i_{x}$. Hence, by using Proposition \ref{exchange}, for the case $q-i_{q}\geq p$, where $q=k_{1}$, ~$i_{q}=i_{k_{1}}$, ~$p=\rho_{2}+i_{x}$, ~and ~$i_{p}=\rho_{2}$, we have:
$$t_{k_{1}}^{i_{k_{1}}}\cdot t_{\rho_{2}+i_{x}}^{\rho_{2}}=t_{\rho_{1}}^{i_{k_{1}}}\cdot t_{\rho_{1}+i_{x}}^{\rho_{2}}\cdot t_{k_{1}}^{i_{k_{1}}}.$$
The assumption $k_{1}-\rho_{1}\geq x$ implies also $k_{1}-i_{k_{1}}\geq\rho_{2}+x$. Hence, by using again Proposition \ref{exchange}, for the case $q-i_{q}\geq p$, where $q=k_{1}$, ~$i_{q}=i_{k_{1}}$, ~$p=\rho_{2}+x$, ~and ~$i_{p}=\rho_{2}$, we have:
$$t_{k_{1}}^{i_{k_{1}}}\cdot t_{\rho_{2}+x}^{i_{x}}=t_{i_{k_{1}}+i_{x}}^{i_{k_{1}}}\cdot t_{\rho_{1}+x}^{i_{x}}\cdot t_{k_{1}}^{i_{k_{1}}}.$$
Thus,
$$\prod_{j=1}^{m}t_{k_{j}}^{i_{k_{j}}}\cdot t_{x}^{i_{x}}=t_{\rho_{1}}^{i_{k_{1}}}\cdot t_{\rho_{1}+i_{x}}^{\rho_{2}}\cdot t_{i_{k_{1}}+i_{x}}^{i_{k_{1}}}\cdot t_{\rho_{1}+x}^{i_{x}}\cdot \prod_{j=1}^{m}t_{k_{j}}^{i_{k_{j}}}.$$
Now, look at the sub-word $$t_{\rho_{1}+i_{x}}^{\rho_{2}}\cdot t_{i_{k_{1}}+i_{x}}^{i_{k_{1}}}.$$
By using Proposition \ref{exchange-2}, for the case $q-i_{q}=p$, where $q=\rho_{1}+i_{x}$, ~$i_{q}=\rho_{2}$, ~$p=i_{k_{1}}+i_{x}$, ~and ~$i_{p}=i_{k_{1}}$, we have:
$$t_{\rho_{1}+i_{x}}^{\rho_{2}}\cdot t_{i_{k_{1}}+i_{x}}^{i_{k_{1}}}=t_{\rho_{1}}^{\rho_{2}}\cdot t_{\rho_{1}+i_{x}}^{\rho_{1}}.$$
Thus,
$$\prod_{j=1}^{m}t_{k_{j}}^{i_{k_{j}}}\cdot t_{x}^{i_{x}}=t_{\rho_{1}+i_{x}}^{\rho_{1}}\cdot t_{\rho_{1}+x}^{i_{x}}\cdot \prod_{j=1}^{m}t_{k_{j}}^{i_{k_{j}}}.$$
Therefore, the proposition holds for every $m'$.
\end{proof}

\begin{proposition}\label{basic-canonical-2}
Let $\pi\in S_n$, where $$norm(\pi)=\prod_{u=p}^{q}\prod_{r=0}^{y_{u}-1}s_{u-r},$$
such that $y_{u+1}\leq y_{u}$, for $p\leq u\leq q-1$ (For a convenience, we omit from the formula in Definition \ref{brenti-normal} $u<p$, and $u>q$, where $y_u=0$), then $\pi$ is a standard $OGS$ elementary element, with the following standard $OGS$ canonical form: $$\prod_{j=1}^{m}t_{k_{j}}^{i_{k_{j}}}$$
with $i_{k_{j}}>0$ for every $1\leq j\leq m$, satisfies the following conditions:
\begin{itemize}
\item If $y_{p}=p$, then:
\begin{itemize}
\item $p=maj\left(\pi\right)<k_{1}$;
\item $k_{j}=u$, for $1\leq j\leq m-1$, if and only if $y_{u}<y_{u-1}$, and  $i_{k_{j}}=y_{u-1}-y_{u}$;
\item $k_{m}=q+1$, and $i_{k_{m}}=y_{q}$.
\end{itemize}
\item If $y_{p}<p$, then:
\begin{itemize}
\item $k_{1}=maj\left(\pi\right)=p$, and $i_{k_{1}}=p-y_{p}$;
\item $k_{j}=u$, for $2\leq j\leq m-1$, if and only if $y_{u}<y_{u-1}$, and  $i_{k_{j}}=y_{u-1}-y_{u}$;
\item $k_{m}=q+1$, and $i_{k_{m}}=y_{q}$.
\end{itemize}
\end{itemize}
\end{proposition}

\begin{proof}
Consider $norm(\pi)=\prod_{u=p}^{q}\prod_{r=0}^{y_{u}-1}s_{u-r}$. By our assumption, $y_{u}\leq y_{u-1}$ for every $u$ apart from $u=p$. Thus, by Proposition \ref{normal-canonical}, the standard $OGS$ canonical form of $\pi$ is  $\pi=\prod_{j=1}^{m}t_{k_{j}}^{i_{k_{j}}}$, where $i_{k_{j}}>0$ for every $1\leq j\leq m$, with the following conditions:
\begin{itemize}
\item Since $y_{u}=0$ for $u<p$, ~we have ~$k_{1}\geq p$. Therefore, in case $y_{p}<p$, ~$k_{1}=p$ and $i_{k_{1}}=y_{p}$. In case $y_{p}=p$, ~$k_{1}$ is the smallest $u>p$, such that $y_{u}<y_{u-1}=p$, and $i_{k_{1}}=p-y_{u}$;
\item For ~$2\leq j<m$, ~$k_{j}=u$, such that $y_{u}<y_{u-1}$ and $i_{k_{j}}=y_{u-1}-y_{u}$;
\item Since $y_{u}=0$ for $u>q$, ~we have ~$k_{m}=y_{q+1}$ and $i_{k_{m}}=y_{q}$;
\item $maj\left(\pi\right)=\sum_{j=1}^{m}i_{k_{j}}=p-y_{p}+\sum_{u=p}^{q-1}(y_{u}-y_{u+1})+y_{q}=p$.
\end{itemize}
\end{proof}

Now, we define the main definition of the paper, \textbf{Standard $OGS$ elementary factorization}, which allows us to present every $\pi\in S_{n}$ as a product
of standard $OGS$ elementary elements, by using the standard $OGS$ of $\pi$.
\\

\begin{definition}\label{canonical-factorization-def}
Let $\pi\in S_n$. Let $z(\pi)$ be the minimal number, such that $\pi$ can be presented as a product of standard $OGS$ elementary elements, with the following conditions:
\begin{itemize}
\item $$\pi=\prod_{v=1}^{z(\pi)}\pi^{(v)}, ~~~~ where ~~~~\pi^{(v)}=\prod_{j=1}^{m^{(v)}}t_{h^{(v)}_{j}}^{\imath_{j}^{(v)}},$$
 by the presentation in the standard $OGS$ canonical form  for every $1\leq v\leq z(\pi)$ and  $1\leq j\leq m^{(v)}$ such that:
 \begin{itemize}
\item $\imath_{j}^{(v)}>0;$ \\
\item $\sum_{j=1}^{m^{(1)}}\imath_{j}^{(1)}\leq h^{(1)}_{1}$ i.e., $maj\left(\pi^{(1)}\right)\leq h^{(1)}_{1}$; \\
\item $h^{(v-1)}_{m^{(v-1)}}\leq\sum_{j=1}^{m^{(v)}}\imath_{j}^{(v)}\leq h^{(v)}_{1}$ for $2\leq v\leq z$ \\ \\
i.e., $h^{(v-1)}_{m^{(v-1)}}\leq maj\left(\pi^{(v)}\right)\leq h^{(v)}_{1} ~~ for ~~ 2\leq v\leq z$.
\end{itemize}
\end{itemize}
Then, the mentioned presentation is called \textbf{Standard $OGS$ elementary factorization} of $\pi$. Since
the factors $\pi^{(v)}$ are standard $OGS$ elementary elements, they are called standard $OGS$ elementary factors of $\pi$.
\end{definition}

The next theorem shows some very important properties of the standard $OGS$ elementary factorization, which is connected to the descents of $\pi$, and we give an explicit formula for the Coxeter length of an arbitrary $\pi\in S_{n}$ by using the standard $OGS$.
\\

\begin{theorem}\label{canonical-factorization}
Let $\pi=\prod_{j=1}^{m}t_{k_{j}}^{i_{k_{j}}}$ be an element of $S_n$ presented in the standard $OGS$ canonical form, with $i_{k_{j}}>0$ for every $1\leq j\leq m$. Consider the standard $OGS$ elementary factorization of $\pi$ with all the notations used in Definition \ref{canonical-factorization-def}.
Then, the following properties hold:
\begin{itemize}
\item The standard $OGS$ elementary factorization of $\pi$ is unique, i.e., the parameters  $z(\pi)$, $m^{(v)}$ for $1\leq v\leq z(\pi)$, $h^{(v)}_{j}$, and $\imath_{j}^{(v)}$ for $1\leq j\leq m^{(v)}$, are uniquely determined by the standard $OGS$ canonical form of $\pi$, such that:
    \begin{itemize}
    \item For every $h^{(v)}_{j}$ there exists exactly one $k_{j'}$ (where, $1\leq j'\leq m$), such that $h^{(v)}_{j}=k_{j'}$;
    \item If $h^{(v)}_{j}=k_{j'}$, for some $1\leq v\leq z(\pi)$, ~$1<j<m^{(v)}$, and $1\leq j'\leq m$, then $\imath_{j}^{(v)}=i_{k_{j'}}$;
    \item If $h^{(v_{1})}_{j_{1}}=h^{(v_{2})}_{j_{2}}$, where $1\leq v_{1}<v_{2}\leq z(\pi)$, ~$1\leq j_{1}\leq m^{(v_{1})}$, and  \\ $1\leq j_{2}\leq m^{(v_{2})}$, then necessarily $v_{1}=v_{2}-1$, ~$j_{1}=m^{(v_{1})}$, ~$j_{2}=1$, and $$h^{(v_{2}-1)}_{m^{(v_{2}-1)}}=h^{(v_{2})}_{1}=maj\left(\pi_{(v_{2})}\right)=k_{j'},$$
        for some $j'$, such that $\imath_{m^{(v_{2}-1)}}^{(v_{2}-1)}+\imath_{1}^{(v_{2})}=i_{k_{j'}}$;
    \end{itemize}

\item $$norm(\pi)=\prod_{v=1}^{z(\pi)}norm(\pi^{(v)});$$
\item $$\ell(s_r\cdot \pi) = \begin{cases} \ell(\pi)-1 & r=\sum_{j=1}^{m^{(v)}}\imath_{j}^{(v)} ~~for ~~ 1\leq v\leq z(\pi) \\
\ell(\pi)+1 & otherwise \end{cases}.$$
i.e., $$des\left(\pi\right)=\bigcup_{v=1}^{z(\pi)}des\left(\pi^{(v)}\right)=\{maj\left(\pi^{(v)}\right)~|~1\leq v\leq z(\pi)\};$$
\item
\begin{align*}
\ell(\pi) &= \sum_{v=1}^{z(\pi)}\ell(\pi^{(v)}) \\ &= \sum_{v=1}^{z(\pi)}\sum_{j=1}^{m^{(v)}}h^{(v)}_{j}\cdot \imath_{j}^{(v)}-\sum_{v=1}^{z(\pi)}\left(maj\left(\pi^{(v)}\right)\right)^{2} \\ &= \sum_{x=1}^{m}k_{x}\cdot i_{k_{x}}-\sum_{v=1}^{z(\pi)}\left(maj\left(\pi^{(v)}\right)\right)^{2} \\ &= \sum_{x=1}^{m}k_{x}\cdot i_{k_{x}}-\sum_{v=1}^{z(\pi)}{{\left(c^{(v)}\right)}}^{2}, ~~ where ~~ c^{(v)}\in des\left(\pi\right).
\end{align*}
\end{itemize}
\end{theorem}

\begin{proof}
Let $\pi=\prod_{j=1}^{m}t_{k_{j}}^{i_{k_{j}}}$,
such that, $2\leq k_1<k_2<\ldots <k_m\leq n$, and $i_{k_{j}}>0$ for every $1\leq j\leq m$. We build the standard $OGS$ elementary factorization of $\pi$ in the following way. Let us start with the structure of $\pi^{(z(\pi))}$. Consider the smallest integer $r$, such that $\sum_{x=m-r+1}^{m}i_{k_{x}}\geq k_{m-r}$, and fit $m^{(z(\pi))}$ to be $r$,  and $h^{(z(\pi))}_{y}$ to be $k_{m-r+y}$ for every integer $1\leq y\leq r$. We set $\imath_{y}^{(z(\pi))}$ to be $i_{k_{m-r+y}}$ for $2\leq y\leq r$, and $\imath_{1}^{(z(\pi))}$ as follows:
Let $\imath_{1}^{(z(\pi))}$ be $i_{k_{m-r+1}}$ in case $\sum_{x=m-r+1}^{m}i_{k_{x}}\leq k_{m-r+1}$, and let $\imath_{1}^{(z(\pi))}$ be $k_{m-r+1}-\sum_{x=m-r+2}^{m}i_{k_{x}}$ in case $\sum_{x=m-r+1}^{m}i_{k_{x}}\geq k_{m-r+1}$. Now, we have $\pi=\pi'\cdot \pi^{(z(\pi))}$, where
\begin{itemize}
\item in case $\sum_{x=m-r+1}^{m}i_{k_{x}}\leq k_{m-r+1}$:
 \begingroup
 \large
 $$\pi'=t_{k_1}^{i_{k_1}}\cdot t_{k_2}^{i_{k_2}}\cdots t_{k_{m-r}}^{i_{k_{m-r}}}$$
 $$\pi^{(z(\pi))}=t_{k_{m-r+1}}^{i_{k_{m-r+1}}}\cdots t_{k_{m-1}}^{i_{k_{m-1}}}\cdot t_{k_m}^{i_{k_m}}.$$
\endgroup
Thus,
\begingroup
\large
$$k_{m-r}< ~maj\left(\pi^{(z(\pi))}\right)=\sum_{x=m-r+1}^{m}i_{k_{x}}~\leq k_{m-r+1};$$
\endgroup
\item in case $\sum_{x=m-r+1}^{m}i_{k_{x}}\geq k_{m-r+1}$:
\begingroup
\large
$$\pi'=t_{k_1}^{i_{k_1}}\cdot t_{k_2}^{i_{k_2}}\cdots t_{k_{m-r}}^{i_{k_{m-r}}}\cdot t_{k_{m-r+1}}^{i_{k_{m-r+1}}-k_{m-r+1}+\sum_{x=m-r+2}^{m}i_{k_{x}}}$$ $$\pi^{(z(\pi))}=t_{k_{m-r+1}}^{k_{m-r+1}-\sum_{x=m-r+2}^{m}i_{k_{x}}}\cdot t_{k_{m-r+2}}^{i_{k_{m-r+2}}}\cdots t_{k_{m-1}}^{i_{k_{m-1}}}\cdot t_{k_m}^{i_{k_m}}.$$
\endgroup
Thus,
\begingroup
\large
$$maj\left(\pi^{(z(\pi))}\right)=k_{m-r+1}.$$
\endgroup
     Now, we look at $\pi'$ and we construct $\pi^{(z(\pi)-1)}$ from the terminal segment of $\pi'$ in the same way as we constructed $\pi^{(z(\pi))}$ from the terminal segment of $\pi$, and we get $\pi'=\pi''\cdot \pi^{(z(\pi)-1)}$. Then,  $\pi=\pi''\cdot \pi^{(z(\pi)-1)}\cdot \pi^{(z(\pi))}$. We continue in the same way, by defining $\pi^{(x)}$ for every $1\leq x$. Finally, we get $\pi=\prod_{v=1}^{z(\pi)}\pi^{(v)}$.
\end{itemize}
Since $\pi^{(v)}$ is a standard $OGS$ elementary element for all $1\leq v\leq z(\pi)$, such that \\ $h^{(v)}_{1}\geq h^{(v-1)}_{ m^{(v-1)}}$ for every $2\leq v\leq z(\pi)$,  by using Theorem \ref{basic-canonical}, we have $$norm(\pi)=\prod_{v=1}^{z(\pi)}norm(\pi^{(v)}).$$
Now, we prove the next part of the theorem. The proof is in induction on $z(\pi)$. By the last part of Theorem \ref{basic-canonical}, $\ell(s_r\cdot \pi^{(1)})=\ell(\pi^{(1)})-1$ if and only if $r=\sum_{j=1}^{m^{(1)}}\imath_{j}^{(1)}$. Therefore, this part of the theorem holds for $v=1$. Now, assume in induction, for every $v\leq z(\pi)-1$:  $$\ell(s_r\cdot \pi^{(1)}\cdot \pi^{(2)}\cdots \pi^{(v)})=\begin{cases}\ell(\pi^{(1)}\cdot \pi^{(2)}\cdots \pi^{(v)})-1 & r=\sum_{j=1}^{m^{(w)}}\imath_{j}^{(w)} \\ \ell(\pi^{(1)}\cdot \pi^{(2)}\cdots \pi^{(v)})+1 & r\neq\sum_{j=1}^{m^{(w)}}\imath_{j}^{(w)}\end{cases},$$ for some $w\leq v$. Now, consider $v=z(\pi)$. Let $r\in maj\left(\pi^{(v')}\right)$ for some $1\leq v'\leq z(\pi)-1$. Then, $$\ell(s_r\cdot \pi)=\ell(s_{r}\cdot \prod_{v=1}^{z(\pi)-1}\pi^{(v)}\cdot \pi^{(z(\pi))})\leq \ell(s_r\cdot \prod_{v=1}^{z(\pi)-1}\pi^{(v)})+\ell(\pi^{(z(\pi))}).$$
Since $r\in maj\left(\pi^{(v')}\right)$ for some $1\leq v'\leq z(\pi)-1$, by our induction hypothesis, $$\ell(s_{r}\cdot  \prod_{v=1}^{z(\pi)-1}\pi^{(v)})=\ell(\prod_{v=1}^{z(\pi)-1}\pi^{(v)})-1.$$
Thus, $$\ell(s_{r}\cdot \pi)\leq \ell(\prod_{v=1}^{z(\pi)-1}\pi^{(v)})-1+\ell(\pi^{(z(\pi))})=\ell(\pi)-1.$$
By using $norm(\pi)=\sum_{v=1}^{z(\pi)}norm(\pi^{(v)})$, and by the property that a product of an element by a Coxeter generator $s_{r}$ either shortens or lengthens the length of the element by $1$, we conclude  $$\ell(s_{r}\cdot \pi)=\ell(\pi)-1,$$
for every $r\in maj\left(\pi^{(v')}\right)$ for some $1\leq v'\leq z(\pi)-1$.
 Notice also, that the sum of all $r$ such that $\ell(s_r\cdot \pi)=\ell(\pi)-1$ is the sum of the locations of all the descents of $\pi$, which is $maj\left(\pi\right)$. By \cite{AR}, $$maj\left(\pi\right)=\sum_{v=1}^{z(\pi)}\sum_{j=1}^{m^{(v)}}\imath_{j}^{(v)}$$
 and
 $$maj\left(\prod_{v=1}^{z(\pi)-1}\pi^{(v)}\right)=\sum_{v=1}^{z(\pi)-1}\sum_{j=1}^{m^{(v)}}\imath_{j}^{(v)}.$$
 Let $q$ be the number of descents of $\pi$, which are not descents of $\pi^{(v)}$ for any $v<z(\pi)$, and denote by $r_{x}$  (where, $1\leq x\leq q$) the descents such that $r_{x}\in des\left(\pi\right)$ and $r_{x}\notin des\left(\pi^{(v)}\right)$ for $v<z(\pi)$. Then, the following holds:
 \begin{align*}
 \sum_{x=1}^{q}r_{x}=maj\left(\pi\right)-maj\left(\prod_{v=1}^{z(\pi)-1}\pi^{(v)}\right)&=\sum_{v=1}^{z(\pi)}\sum_{j=1}^{m^{(v)}}\imath_{j}^{(v)}-\sum_{v=1}^{z(\pi)-1}
 \sum_{j=1}^{m^{(v)}}\imath_{j}^{(v)} \\ &=\sum_{j=1}^{m^{(z(\pi))}}\imath_{j}^{(z(\pi))}=maj\left(\pi^{(z(\pi))}\right),
 \end{align*}
 and by \cite{BB}
$$s_{r_{x}}\cdot \left(\prod_{v=1}^{z(\pi)-1}\pi^{(v)}\right)\cdot \pi^{(z(\pi))}=\left(\prod_{v=1}^{z(\pi)-1}\pi^{(v)}\right)\cdot\hat{\pi}_{z(\pi)}$$
 where, we get $\hat{\pi}_{z(\pi)}$ from $\pi^{(z(\pi))}$ by omitting one Coxeter generator from a reduced presentation of it (i.e., $s_{r_{x}}$ shortens by $1$ the length of the segment $\pi^{(z(\pi))}$ of $\pi$).

 Notice, by Theorem \ref{basic-canonical}, the first letter from left to right of $norm(\pi^{(z(\pi))})$  is \\
 \begingroup
 \large
  $s_{maj\left(\pi^{(z(\pi))}\right)}$.
  \endgroup
 We already proved $norm(\pi)=norm(\prod_{v=1}^{z(\pi)-1}\pi^{(v)})\cdot norm(\pi^{(z(\pi))})$. Therefore, the first letter from left to right of the segment $norm(\pi^{(z(\pi))})$  in $norm(\pi)$ is
 \begingroup
 \large
 $s_{maj\left(\pi^{(z(\pi))}\right)}$
 \endgroup
 too. Thus, any $r_{x}<maj\left(\pi^{(z(\pi))}\right)$ cannot shorten the length of the segment $norm(\pi^{(z(\pi))})$  in $norm(\pi)$. Thus, $q=1$, and $r_{1}=maj\left(\pi^{(z(\pi))}\right)$ is the only element in $des\left(\pi\right)$ which is not in $des\left(\pi^{(v)}\right)$, for $1\leq v\leq z(\pi)-1$.
  That proves
 $$des\left(\pi\right)=\bigcup_{v=1}^{z(\pi)}des\left(\pi^{(v)}\right)=\{maj\left(\pi^{(v)}\right)~|~1\leq v\leq z(\pi)\}.$$
Now, we prove the last part of the theorem, the explicit formula for length $\pi\in S_n$. Since $norm(\pi)=\prod_{v=1}^{z(\pi)}norm(\pi^{(v)})$ by a former part of the proposition, and  $norm(\pi)$ is a reduced Coxeter presentation of $\pi$, we get
$$\ell(\pi)=\sum_{v=1}^{z(\pi)}\ell(\pi^{(v)}).$$
Since $\pi^{(v)}$ is a standard $OGS$ elementary factor of $\pi$ for every $1\leq v\leq z(\pi)$, by Theorem \ref{basic-canonical}, $$\ell(\pi^{(v)})=\sum_{j=1}^{m^{(v)}}h^{(v)}_{j}\cdot \imath_{j}^{(v)}-\left(maj\left(\pi^{(v)}\right)\right)^{2}.$$
By the first part of the theorem, $$\sum_{v=1}^{z(\pi)}\sum_{j=1}^{m^{(v)}}h^{(v)}_{j}\cdot \imath_{j}^{(v)}=\sum_{x=1}^{m}k_{x}\cdot i_{k_{x}}.$$
By a former part of the theorem, $c^{(v)}\in des\left(\pi\right)$ if and only if $c^{(v)}=maj\left(\pi^{(v)}\right)$ for some $1\leq v\leq z(\pi)$. Therefore, we get
$$\ell(\pi)=\sum_{x=1}^{m}k_{x}\cdot i_{k_{x}}-\sum_{v=1}^{z(\pi)}{\left(c^{(v)}\right)}^{2}, ~~ where ~~ c^{(v)}\in des\left(\pi\right).$$
\end{proof}

\begin{example}
Let $\pi=t_{3}\cdot t_{4}^{2}\cdot t_{6}^{4}\cdot t_{7}^{3}\cdot t_{9}^{3}\cdot t_{10}^{2}$. Then, the standard $OGS$ elementary factors of $\pi$ are as follows:
$$\pi^{(1)}=t_{3}\cdot t_{4}^{2},$$
$$\pi^{(2)}=t_{6}^{4}\cdot t_{7},$$ and
$$\pi^{(3)}=t_{7}^{2}\cdot t_{9}^{3}\cdot t_{10}^{2}.$$
Now we compute $norm(\pi^{(1)})$, $norm(\pi^{(2)})$, and $norm(\pi^{(3)})$ by using Theorem \ref{basic-canonical}:
$$norm(\pi^{(1)})=s_3\cdot s_2,$$
$$norm(\pi^{(2)})=(s_5\cdot s_4\cdot s_3\cdot s_2\cdot s_1)\cdot s_6,$$
$$norm(\pi^{(3)})=(s_7\cdot s_6\cdot s_5\cdot s_4\cdot s_3)\cdot (s_8\cdot s_7\cdot s_6\cdot s_5\cdot s_4)\cdot (s_9\cdot s_8).$$
Therefore,
\begin{align*}
norm(\pi) &= norm(\pi^{(1)})\cdot norm(\pi^{(2)})\cdot norm(\pi^{(3)}) \\ &= [s_3\cdot s_2]\cdot [(s_5\cdot s_4\cdot s_3\cdot s_2\cdot s_1)\cdot s_6]\cdot \\ &\cdot [(s_7\cdot s_6\cdot s_5\cdot s_4\cdot s_3)\cdot (s_8\cdot s_7\cdot s_6\cdot s_5\cdot s_4)\cdot (s_9\cdot s_8)]
\end{align*}

$$\pi^{(1)}=[1; ~3; ~4; ~2],  ~~~~ \pi^{(2)}=[2; ~3; ~4; ~5; ~7; ~1; ~6], ~~~~ \pi^{(3)}=[1; ~2; ~5; ~6; ~7; ~9; ~10; ~3; ~4, ~8].$$
$$\pi=\pi^{(1)}\cdot \pi^{(2)}\cdot \pi^{(3)}=[2; ~6; ~7; ~5; ~10; ~1; ~9; ~3; ~4; ~8].$$

$$des\left(\pi^{(1)}\right)=\left\{maj\left(\pi^{(1)}\right)\right\}=\{3\}, ~~~~des\left(\pi^{(2)}\right)=\left\{maj\left(\pi^{(2)}\right)\right\}=\{5\},$$ $$des\left(\pi^{(3)}\right)=\left\{maj\left(\pi^{(3)}\right)\right\}=\{7\},$$
$$des\left(\pi\right)=des\left(\pi^{(1)}\right)\cup des\left(\pi^{(2)}\right)\cup des\left(\pi^{(3)}\right)=\left\{3, ~5, ~7\right\}.$$

\begin{align*}
\ell(\pi) &= \ell(\pi^{(1)})+\ell(\pi^{(2)})+\ell(\pi^{(3)}) \\ &= \left(3\cdot 1+4\cdot 2-3^2\right)+\left(6\cdot 4+7\cdot 1-5^2\right)+\left(7\cdot 2+9\cdot 3+10\cdot 2-7^2\right) \\ &=\left(3\cdot 1+4\cdot 2+6\cdot 4+7\cdot 3+9\cdot 3+10\cdot 2\right)-\left(3^2+5^2+7^2\right) \\&=20.
\end{align*}
\end{example}

In Theorem \ref{canonical-factorization} we have shown a very strong connection between the standard $OGS$ elementary factorization, the Coxeter length function, and the descent set of an element $\pi\in S_n$. Therefore, in the next proposition, we show how we conclude the normal form of the elementary factors in the standard $OGS$ elementary factorization, from a normal form of a given $\pi\in S_n$.
\begin{proposition}\label{canonical-factorization-2}
Let $\pi\in S_n$, where $$norm(\pi)=\prod_{u=p}^{q}\prod_{r=0}^{y_{u}-1}s_{u-r}.$$
Denote by $u_{1}, u_{2}, \ldots ,u_{z}$ the values of $u$ such that
\begin{itemize}
\item $u_{1}=p$;
\item $y_{u_{v}}>y_{u_{v}-1}$, for $1\leq v\leq z$;
\item $u_{1}<u_{2}<\cdots <u_{z}$.
\end{itemize}
Moreover, let $u_{z+1}=q+1.$
Then, $\pi=\prod_{v=1}^{z(\pi)}\pi^{(v)}$ is a standard $OGS$ elementary factorization of $\pi$, such that $$norm(\pi^{(v)})=\prod_{u=u_{v}}^{u_{v+1}-1}\prod_{r=0}^{y_{u}-1}s_{u-r}.$$
\end{proposition}

\begin{proof}
Consider $$norm(\pi^{(v)})=\prod_{u=u_{v}}^{u_{v+1}-1}\prod_{r=0}^{y_{u}-1}s_{u-r},$$ for $1\leq v\leq z(\pi)$. Since $y_{u_{j+1}}\leq  y_{u_{j}}$, for every $u_{v}\leq j<u_{v+1}-1$, by Proposition \ref{basic-canonical-2}, we have that $\pi^{(v)}$ is a standard $OGS$ elementary element, with $maj\left(\pi^{(v)}\right)=u_{v}$. Assume,
$$norm(\pi^{(v-1)})=\prod_{u=u_{v-1}}^{u_{v}-1}\prod_{r=0}^{y_{u}-1}s_{u-r},$$ where, the presentation of $\pi^{(v-1)}$ in standard $OGS$ canonical form is
\begingroup
\large
$$\pi^{(v-1)}=\prod_{j=1}^{m^{(v-1)}}{t_{{h^{(v-1)}}_{j}}}^{\imath_{j}^{(v-1)}},$$
\endgroup
for some $m^{(v-1)}$, $1\leq j\leq m^{(v-1)}$, $u_{v-1}\leq h^{(v-1)}_{j}\leq u_{v}$, and $0<\imath_{j}^{(v-1)}<h^{(v-1)}_{j}$. Since by Proposition \ref{basic-canonical-2}, $$maj\left(\pi^{(v)}\right)=u_{v}\geq h^{(v-1)}_{j}$$ for every $1\leq j\leq m^{(v-1)}$, we have $$h^{(v-1)}_{m^{(v-1)}}\leq maj\left(\pi^{(v)}\right)\leq h^{(v)}_{1}.$$ Therefore, by Definition \ref{canonical-factorization-def}, ~$\pi=\prod_{v=1}^{z(\pi)}\pi^{(v)}$ is a standard $OGS$ elementary factorization of $\pi$.
\end{proof}
\\

The next three propositions, show some other interesting properties of the standard $OGS$ elementary factorization, concerning exchange properties, parabolic subgroups, and elementary factors, which commute.

\begin{proposition}\label{exchange-elementary2}
Let $\pi\in S_n$, presented in standard $OGS$ elementary factorization, with all the notations used in Definition \ref{canonical-factorization-def}. Let $t_{q}^{i_{q}}\in S_{n}$ such that $q-i_{q}\geq h^{(z(\pi))}_{m^{(z(\pi))}}$. Then, the following holds:
\begingroup
\large
$$t_{q}^{i_{q}}\cdot \pi=\prod_{v=1}^{z(\pi)}\left(t_{\rho_{1}^{(v)}+i_{q}}^{i_{q}}\cdot \prod_{j=1}^{m^{(v)}}t_{h^{(v)}_{j}+i_{q}}^{\imath_{j}^{(v)}}\right)\cdot t_{q}^{i_{q}}.$$
\endgroup
\end{proposition}

\begin{proof}
The proof is in double induction on $z(\pi)$ and on $m^{(v)}$, for $1\leq v\leq z(\pi)$. First, consider the case $z(\pi)=1$ (i.e., $\pi$ is a standard $OGS$ elementary element). For the case of $z(\pi)=1$ and $m^{(1)}=1$, the proposition holds by using Proposition \ref{exchange}, where considering $t_{q}^{i_{q}}\cdot t_{p}^{i_{p}}$ in case $q-i_{q}\geq p$. Assume in induction that the proposition holds for $m^{(1)}<r$ for some $r$, and we prove it for $m^{(1)}=r$.
Therefore, by the induction hypothesis we have:
\begingroup
\large
$$t_{q}^{i_{q}}\cdot \pi=t_{\sum_{j=1}^{r-1}\imath_{j}^{(1)}+i_{q}}^{i_{q}}\cdot \prod_{j=1}^{r-1}t_{h^{(1)}_{j}+i_{q}}^{\imath_{j}^{(1)}}\cdot t_{q}^{i_{q}}\cdot t_{h^{(1)}_{r}}^{\imath_{r}^{(1)}}.$$
\endgroup
By Proposition \ref{exchange}, where considering $t_{q}^{i_{q}}\cdot t_{p}^{i_{p}}$ in case $q-i_{q}\geq p$,
\begingroup
\large
$$t_{q}^{i_{q}}\cdot t_{h^{(1)}_{r}}^{\imath_{r}^{(1)}}=t_{\imath_{r}^{(1)}+i_{q}}^{i_{q}}\cdot t_{h^{(1)}_{r}+i_{q}}^{\imath_{r}^{(1)}}\cdot t_{q}^{i_{q}}.$$
\endgroup
Notice, $\prod_{j=1}^{r-1}t_{h^{(1)}_{j}+i_{q}}^{\imath_{j}^{(1)}}$ is a sub-word of $\pi$, which is a standard $OGS$ elementary element. Therefore, by Theorem \ref{basic-canonical}, $\prod_{j=1}^{r-1}t_{h^{(1)}_{j}+i_{q}}^{\imath_{j}^{(1)}}$ is a standard $OGS$ elementary element too.
Then, by using Proposition \ref{exchange-elementary1}, for $\pi=\prod_{j=1}^{r-1}t_{h^{(1)}_{j}+i_{q}}^{\imath_{j}^{(1)}}$, ~~$x=i_{q}+\imath_{r}^{(1)}$, ~~and ~$i_{x}=i_{q}$, we get:
\begingroup
\large
$$\prod_{j=1}^{r-1}t_{h^{(1)}_{j}+i_{q}}^{\imath_{j}^{(1)}}\cdot t_{i_{q}+\imath_{r}^{(1)}}^{i_{q}}=t_{\sum_{j=1}^{r-1}\imath_{j}^{(1)}+i_{q}}^{\sum_{j=1}^{r-1}\imath_{j}^{(1)}}\cdot t_{\sum_{j=1}^{r-1}\imath_{j}^{(1)}+i_{r}^{(1)}+i_{q}}^{i_{q}}\cdot \prod_{j=1}^{r-1}t_{h^{(1)}_{j}+i_{q}}^{\imath_{j}^{(1)}}.$$
\endgroup
Thus, we get
\begingroup
\large
$$t_{q}^{i_{q}}\cdot \pi=t_{\rho^{(1)}_{1}+i_{q}}^{i_{q}}\cdot \prod_{j=1}^{m^{(1)}}t_{h^{(1)}_{j}+i_{q}}^{\imath_{j}^{(1)}}\cdot t_{q}^{i_{q}}$$
\endgroup
in case $\pi$ is a standard $OGS$ elementary element.
Now, assume in induction the proposition holds for every $z<z(\pi)$ and we prove it for $z=z(\pi)$.
Then,
\begingroup
\large
$$t_{q}^{i_{q}}\cdot \pi=\prod_{v=1}^{z(\pi)-1}t_{\rho_{1}^{(v)}+i_{q}}^{i_{q}}\cdot \prod_{j=1}^{m^{(v)}}t_{h^{(v)}_{j}+i_{q}}^{\imath_{j}^{(v)}}\cdot t_{q}^{i_{q}}\cdot \pi^{(z(\pi))}.$$
\endgroup
by the induction hypothesis.
Since $\pi^{(z(\pi))}$ is a standard $OGS$ elementary element, we have
\begingroup
\large
$$t_{q}^{i_{q}}\cdot \pi^{(z(\pi))}=t_{\rho^{(z(\pi))}_{1}+i_{q}}^{i_{q}}\cdot \prod_{j=1}^{m^{(1)}}t_{h^{(z(\pi))}_{j}+i_{q}}^{\imath_{j}^{(z(\pi)))}}\cdot t_{q}^{i_{q}}.$$
\endgroup
Therefore, we get
\begingroup
\large
$$t_{q}^{i_{q}}\cdot \pi=\prod_{v=1}^{z(\pi)}t_{\rho_{1}^{(v)}+i_{q}}^{i_{q}}\cdot \prod_{j=1}^{m^{(v)}}t_{h^{(v)}_{j}+i_{q}}^{\imath_{j}^{(v)}}\cdot t_{q}^{i_{q}}.$$
\endgroup
\end{proof}

\begin{proposition}\label{k-parabolic}
Let $\pi\in S_n$, presented in standard $OGS$ elementary factorization, with all the notations used in Definition \ref{canonical-factorization-def}. Then, the following properties are satisfied:
\begin{itemize}
\item The set of elements $\pi$ such that $maj\left(\pi^{(v)}\right)=h^{(v)}_{1}$ for all $1\leq v\leq z(\pi)$, is the parabolic subgroup generated by $s_{2}, \ldots, s_{n-1}$ (i.e., the permutations of $S_n$, with $1$ as a fixed point);
\item The set of elements $\pi$ such that $maj\left(\pi^{(v)}\right)=h^{(v)}_{1}$  and $\imath_{1}^{(v)}\geq d$ for all $1\leq v\leq z(\pi)$, is the parabolic subgroup generated by $s_{d+1}, \dots, s_{n-1}$ for some $1\leq d\leq n-2$ (i.e., the permutations of $S_n$, with $j$ as fixed point for every $1\leq j\leq d$).
\end{itemize}
\end{proposition}

\begin{proof}
Let $\pi=\prod_{v=1}^{z(\pi)}\pi^{(v)}$, where $\pi^{(v)}$ satisfies the conditions of Definition \ref{canonical-factorization-def}, \\ $h^{(v)}_{1}=maj\left(\pi^{(v)}\right)$, and $\imath_{1}^{(v)}\geq d$ for some $d\geq 1$ and for every $1\leq v\leq z(\pi)$. Let $\pi'=\prod_{v=1}^{z(\pi)}{\pi'}^{(v)}$, such that: $${\pi'}^{(v)}=\prod_{v=1}^{z(\pi)}\left(t_{h^{(v)}_{1}-d}^{\imath_{1}^{(v)}-d}\cdot \prod_{j=1}^{m^{(v)}}t_{h^{(v)}_{j}-d}^{\imath_{j}^{(v)}}\right).$$
Obviously, for every $1\leq v\leq z(\pi)$: $$maj\left({\pi'}^{(v)}\right)=\sum_{\jmath=1}^{m^{(v)}}\imath_{\jmath}^{(v)}-d=maj\left(\pi^{(v)}\right)-d=h^{(v)}_{1}-d<h^{(v)}_{2}-d$$
and for every $2\leq v\leq z(\pi)$:
$$maj\left({\pi'}^{(v)}\right)=\sum_{\jmath=1}^{m^{(v)}}\imath_{\jmath}^{(v)}-d\geq h^{(v-1)}_{m^{(v-1)}}-d.$$
Since for every $1\leq v\leq z(\pi)$, ~$\imath_{1}^{(v)}\geq d$, we have $\begin{cases} \imath_{1}^{(v)}-d=0 & ~~ \imath_{1}^{(v)}=d \\
\imath_{1}^{(v)}-d>0 & ~~ \imath_{1}^{(v)}>d\end{cases}$. Thus, by defining ${m'}^{(v)}$, ${h'}^{(v)}_{\jmath}$, and ${\imath'}_{j}^{(v)}$ for $1\leq v\leq z(\pi)$ and $1\leq \jmath\leq {m'}^{(v)}$ in the following way:
\begin{itemize}
\item In case $\imath_{1}^{(v)}>d$:  ~~${m'}^{(v)}=m^{(v)}$, ~${h'}^{(v)}_{\jmath}=h^{(v)}_{\jmath}-d$, ~${\imath'}_{1}^{(v)}={\imath}_{1}^{(v)}-d$, ~${\imath'}_{\jmath}^{(v)}={\imath}_{\jmath}^{(v)}$ ~for ~$2\leq j\leq m^{(v)}$;
\item In case $\imath_{1}^{(v)}=d$: ~~${m'}^{(v)}=m^{(v)}-1$, ~${h'}^{(v)}_{\jmath}=h^{(v)}_{\jmath+1}-d$, ~${\imath'}_{\jmath}^{(v)}={\imath}_{\jmath+1}^{(v)}$
\end{itemize}
notice that $\pi'=\prod_{v=1}^{z(\pi)}{\pi'}^{(v)}$ is a standard $OGS$ elementary factorization of $\pi'$ with the elementary factors
${\pi'}^{(v)}=\prod_{v=1}^{z(\pi)}\cdot \prod_{j=1}^{{m'}^{(v)}}t_{{h'}^{(v)}_{j}}^{{\imath'}_{j}^{(v)}}$.
Thus, by Proposition \ref{exchange-elementary2}, $\pi=t_{q}^{d}\cdot \pi'\cdot t_{q}^{-d}$ for some $q$ such that $q-d\geq {h'}^{(z({\pi'}))}_{{m'}^{(z({\pi'}))}}=h^{(z(\pi))}_{m^{(z(\pi))}}-d$ (i.e., $q\geq h^{(z(\pi))}_{m^{(z(\pi))}}$). Therefore, $\pi$ is a conjugate of $\pi'$ by $t_{q}^{-d}$. Hence, by choosing $q={h'}^{(z(\pi))}_{{m'}^{(z({\pi'}))}}+d=h^{(z(\pi))}_{m^{(z(\pi))}}$,  we have $t_{q}^{d}\cdot \pi'\cdot t_{q}^{-d}=\pi\in \langle s_{d+1},\ldots, s_{n-1}\rangle$.
\end{proof}

\begin{proposition}\label{commute-perm}
Let $\pi_{1}=\prod_{j=1}^{m}t_{k_j}^{i_{k_j}}$ and $\pi_{2}=\prod_{j=1}^{w}t_{r_j}^{i_{r_j}}$ be two elements of $S_n$ which are presented in the standard $OGS$ canonical form, with the following properties:
 \begin{itemize}
\item $maj\left(\pi_{2}\right)=\sum_{j=1}^{w}i_{r_j}=r_1$;
\item $i_{r_{1}}\geq k_{m}$.
\end{itemize}
Then, $\pi_{1}\cdot \pi_{2}=\pi_{2}\cdot \pi_{1}$ (i.e., $\pi_{1}$ and $\pi_{2}$ are commute).
\end{proposition}

\begin{proof}
Since $maj\left(\pi_{2}\right)=r_1$, ~$\pi_{2}$ is a standard $OGS$ elementary element. Therefore, by Proposition \ref{k-parabolic}, $\pi_{2}$ is in the parabolic subgroup  $\langle s_{i_{r_1}+1}, \ldots, s_{n-1}\rangle$. Since $\pi_{1}=\prod_{j=1}^{m}t_{k_j}^{i_{k_j}}$ by the presentation in standard $OGS$ canonical form, we get by the definition of $t_{k_j}$ for $1\leq j\leq m$, that $\pi_{1}$ is in the parabolic subgroup $\langle s_{1}, \ldots, s_{k_{m}-1}\rangle$. Since $i_{r_1}\geq k_{m}$, we have $(i_{r_1}+1)-(k_{m}-1)\geq 2$. Therefore, every element in $\langle s_{i_{r_1}+1}, \ldots, s_{n-1}\rangle$ commutes with every element in $\langle s_{1}, \ldots, s_{k_{m}-1}\rangle$.
\end{proof}

\subsection{The standard $OGS$ canonical form of inverse elements in $S_n$}\label{inverse-sn}

In this subsection, we give a formula for the standard $OGS$ canonical form of $\pi^{-1}$, the inverse element of $\pi$.
Obviously, $(t_{k}^{i_{k}})^{-1}=t_{k}^{k-i_{k}}$ for every $2\leq k\leq n$. Therefore, we first consider $\pi^{-1}$ for elements $\pi\in S_n$, which have the  form $t_{k_{1}}^{i_{k_{1}}}\cdot t_{k_{2}}^{i_{k_{2}}}$, where $k_{1}<k_{2}$. We get the following proposition.

\begin{proposition}\label{inversesn2}
$t_{k_{1}}^{i_{k_{1}}}\cdot t_{k_{2}}^{i_{k_{2}}}$, where $k_{1}<k_{2}$. Then, the following holds:
$$\pi^{-1}=\begin{cases}
t_{k_{1}+k_{2}-i_{k_{1}}-i_{k_{2}}}^{k_{2}-i_{k_{2}}}\cdot t_{k_{1}+k_{2}-i_{k_{2}}}^{k_{1}-i_{k_{1}}}\cdot t_{k_{2}}^{k_{2}-i_{k_{2}}} & i_{k_{2}}>k_{1} \\
\\
t_{k_{2}-i_{k_{1}}}^{k_{2}-k_{1}}\cdot t_{k_{2}}^{k_{2}-i_{k_{1}}} & i_{k_{2}}=k_{1} \\
\\
t_{k_{2}-i_{k_{2}}}^{k_{1}-i_{k_{2}}}\cdot t_{k_{1}+k_{2}-i_{k_{1}}-i_{k_{2}}}^{k_{2}-k_{1}}\cdot t_{k_{2}}^{k_{1}+k_{2}-i_{k_{1}}-i_{k_{2}}} & k_{1}-i_{k_{1}}<i_{k_{2}}<k_{1} \\
\\
t_{k_{2}-i_{k_{2}}}^{i_{k_{1}}}\cdot t_{k_{2}}^{k_{2}-k_{1}} & k_{1}-i_{k_{1}}=i_{k_{2}} \\
\\
t_{k_{1}-i_{k_{2}}}^{k_{1}-i_{k_{1}}-i_{k_{2}}}\cdot t_{k_{2}-i_{k_{2}}}^{i_{k_{1}}}\cdot t_{k_{2}}^{k_{2}-i_{k_{1}}-i_{k_{2}}} & i_{k_{2}}<k_{1}-i_{k_{1}}
\end{cases}$$
\end{proposition}

\begin{proof}
Let  $\pi=t_{k_{1}}^{i_{k_{1}}}\cdot t_{k_{2}}^{i_{k_{2}}}$. Then, by Proposition \ref{exchange-2}, $\pi$ can be presented in the dual-standard $OGS$ canonical form. Now, by considering the inverse, we get the presentation of  $\pi^{-1}$ in the standard $OGS$ canonical form.
\end{proof}
\\

Now, we introduce an algorithm for the presentation of $\pi^{-1}$ in the standard $OGS$ canonical form for an arbitrary $\pi\in S_n$, which is also presented in the standard $OGS$ canonical form.  In the first step we consider the case where $\pi$ is a standard $OGS$ elementary element of $S_{n}$.

\begin{theorem}\label{elementary-inverse}
Let $\pi=\prod_{j=1}^{m}t_{k_{j}}^{i_{k_{j}}}$, a standard $OGS$ elementary element of $S_{n}$ presented in the standard $OGS$ canonical form (i. e., $maj\left(\pi\right)\leq k_1<k_2<\cdots<k_m$), with $i_{k_{j}}>0$ for every $1\leq j\leq m$. For every $1\leq j\leq m$, let $\rho_{j}$, $\varrho_{j}$, and $\vartheta_{j}$ integers as defined in Definition \ref{rho}, and denote $\kappa_{-;j}=\varrho_{j-1}+\vartheta_{j}=k_{j}-\rho_{j}$, ~ $\kappa_{+;j}=\varrho_{j}+\vartheta_{j}=k_{j}-\rho_{j+1}$ ~(We consider $\varrho_{0}=\rho_{m+1}=0$). Then,
\begingroup
\large
$$\pi^{-1}=t_{\kappa_{+;1}}^{\vartheta_{1}}\cdot (t_{\kappa_{-;2}}^{\varrho_{1}}\cdot t_{\kappa_{+;2}}^{\vartheta_{2}})\cdot (t_{\kappa_{-;3}}^{\varrho_{2}}\cdot t_{\kappa_{+;3}}^{\vartheta_{3}})\cdots (t_{\kappa_{-;m}}^{\varrho_{m-1}}\cdot t_{\kappa_{+;m}}^{\vartheta_{m}}),$$
\endgroup
In the specific case, where $k_{1}=\rho_{1}=maj\left(\pi\right)$, we have $\kappa_{-;1}=\vartheta_{1}=0$, thus:
\begingroup
\large
$$\pi^{-1}=(t_{\kappa_{-;2}}^{\varrho_{1}}\cdot t_{\kappa_{+;2}}^{\vartheta_{2}})\cdot (t_{\kappa_{-;3}}^{\varrho_{2}}\cdot t_{\kappa_{+;3}}^{\vartheta_{3}})\cdots (t_{\kappa_{-;m}}^{\varrho_{m-1}}\cdot t_{\kappa_{+;m}}^{\vartheta_{m}}).$$
\endgroup
Therefore,
\begingroup
\large
$$des\left(\pi^{-1}\right)=\begin{cases}\{\kappa_{-;j} ~| ~1\leq j\leq m\} & ~~k_{1}>maj\left(\pi\right) \\
\{\kappa_{-;j} ~| ~2\leq j\leq m\} & ~~k_{1}=maj\left(\pi\right)\end{cases}.$$
\endgroup
\end{theorem}

\begin{proof}
Let $\pi=\prod_{j=1}^{m}t_{k_{j}}^{i_{k_{j}}}$, a standard $OGS$ elementary element of $S_{n}$ is presented in the standard $OGS$ canonical form, with $i_{k_{j}}>0$ for every $1\leq j\leq m$. Then, by the first part of Theorem \ref{basic-canonical},  $$\pi=\begin{cases} t_{k_{1}}^{\rho_{1}}\cdot (t_{k_{1}}^{\kappa_{+;1}}\cdot t_{k_{2}}^{\rho_{2}})\cdot (t_{k_{2}}^{\kappa_{+;2}}\cdot t_{k_{3}}^{\rho_{3}})\cdots (t_{k_{m-1}}^{\kappa_{+;m-1}}\cdot t_{k_{m}}^{\rho_{m}}) &  ~~ k_{1}>\rho_{1} \\  \\ (t_{k_{1}}^{\kappa_{+;1}}\cdot t_{k_{2}}^{\rho_{2}})\cdot (t_{k_{2}}^{\kappa_{+;2}}\cdot t_{k_{3}}^{\rho_{3}})\cdots (t_{k_{m-1}}^{\kappa_{+;m-1}}\cdot t_{k_{m}}^{\rho_{m}}) &  ~~ k_{1}=\rho_{1}\end{cases}.$$
Now, we prove the formula for $\pi^{-1}$ in induction on $m$. If $m=1$ and $k_{1}>\rho_{1}$,  \\ then $\pi=t_{k_{1}}^{i_{k_{1}}}$. Therefore, $\kappa_{+;1}=k_{1}$,  $\rho_{1}=i_{k_{1}}$, $\vartheta_{1}=k_{1}-i_{k_{1}}$ and obviously, \\ $\pi^{-1}=t_{k_{1}}^{k_{1}-i_{k_{1}}}=t_{\kappa_{+;1}}^{\vartheta_{1}}$ in this case. If $m=2$ and $k_{1}=\rho_{1}$, then $\pi=t_{k_{1}}^{k_{1}-i_{k_{2}}}\cdot t_{k_{2}}^{i_{k_{2}}}$. Therefore, $i_{k_{1}}=k_{1}-i_{k_{2}}=\kappa_{+;1}$, $\kappa_{-;2}=k_{2}-i_{k_{2}}$, and $\kappa_{+;2}=k_{2}$. By Proposition \ref{inversesn2}, $\pi^{-1}=t_{k_{2}-i_{k_{2}}}^{i_{k_{1}}}\cdot t_{k_{2}}^{k_{2}-k_{1}}=t_{\kappa_{-;2}}^{\varrho_{1}}\cdot t_{\kappa_{+;2}}^{\vartheta_{2}}$. Thus, the theorem holds in this case. Now, assume the theorem holds for $m'<m$, and we prove it for $m'=m$.
Consider $$\pi=\pi'\cdot \pi_{m}$$
where, $$\pi'=\begin{cases}t_{k_{1}}^{\rho_{1}}\cdot (t_{k_{1}}^{\kappa_{+;1}}\cdot t_{k_{2}}^{\rho_{2}})\cdot (t_{k_{2}}^{\kappa_{+;2}}\cdot t_{k_{3}}^{\rho_{3}})\cdots (t_{k_{m-2}}^{\kappa_{+;m-2}}\cdot t_{k_{m-1}}^{\rho_{m-1}}) &  ~~ k_{1}>\rho_{1} \\ \\ (t_{k_{1}}^{\kappa_{+;1}}\cdot t_{k_{2}}^{\rho_{2}})\cdot (t_{k_{2}}^{\kappa_{+;2}}\cdot t_{k_{3}}^{\rho_{3}})\cdots (t_{k_{m-2}}^{\kappa_{+;m-2}}\cdot t_{k_{m-1}}^{\rho_{m-1}}) &  ~~ k_{1}=\rho_{1}\end{cases}.$$
and $$\pi_{m}=t_{k_{m-1}}^{\kappa_{+;m-1}}\cdot t_{k_{m}}^{\rho_{m}}.$$
Then, $$\pi^{-1}=\pi_{m}^{-1}\cdot \pi'^{-1}.$$
By our induction hypothesis, $$\pi'^{-1}=t_{\kappa_{+;1}}^{\vartheta_{1}}\cdot (t_{\kappa_{-;2}}^{\varrho_{1}}\cdot t_{\kappa_{+;2}}^{\vartheta_{2}})\cdots (t_{\kappa_{-;m-2}}^{\varrho_{m-3}}\cdot t_{\kappa_{+;m-2}}^{\vartheta_{m-2}})\cdot (t_{\kappa_{-;m-1}}^{\varrho_{m-2}}\cdot t_{k_{m-1}}^{\vartheta_{m-1}}),$$  in case $k_{1}>\rho_{1}$.  $$\pi'^{-1}=(t_{\kappa_{-;2}}^{\varrho_{1}}\cdot t_{\kappa_{+;2}}^{\vartheta_{2}})\cdot (t_{\kappa_{-;3}}^{\varrho_{2}}\cdot t_{\kappa_{+;3}}^{\vartheta_{3}})\cdots (t_{\kappa_{-;m-2}}^{\varrho_{m-3}}\cdot t_{\kappa_{+;m-2}}^{\vartheta_{m-2}})\cdot (t_{\kappa_{-;m-1}}^{\varrho_{m-2}}\cdot t_{k_{m-1}}^{\vartheta_{m-1}}),$$  in case $k_{1}=\rho_{1}$.

By Proposition \ref{inversesn2}, $$\pi_{m}^{-1}=t_{\kappa_{-;m}}^{\kappa_{+;m-1}}\cdot t_{k_{m}}^{k_{m}-k_{m-1}}.$$
Since $\kappa_{-;j}=\varrho_{j-1}+\vartheta_{j}$ ~and ~$\kappa_{+;j}=\varrho_{j}+\vartheta_{j}$, ~the following are satisfied:
\begin{itemize}
\item $maj\left(\pi_{m}^{-1}\right)=\kappa_{+;m-1}+k_{m}-k_{m-1}=\kappa_{-;m}$;
\item $\kappa_{-;m}>\kappa_{+;x}$ for $1\leq x\leq m-2$;
\item $\kappa_{-;m}>\kappa_{-;m-1}$;
\item $\kappa_{+;m-1}>\kappa_{+;x}$ for $1\leq x\leq m-2$;
\item $\kappa_{+;m-1}>\kappa_{-;m-1}$;
\end{itemize}
Therefore, by Proposition \ref{commute-perm}, $\pi_{m}^{-1}$ commutes with the following segments of $\pi'^{-1}$:
\begin{itemize}
\item $t_{\kappa_{+;1}}^{\vartheta_{1}}$;
\item $t_{\kappa_{-;x}}^{\varrho_{x-1}}\cdot t_{\kappa_{+;x}}^{\vartheta_{x}}$, for $2\leq x\leq m-2$;
\item $t_{\kappa_{-;m-1}}^{\varrho_{m-2}}$.
\end{itemize}
Thus,
$$\pi^{-1}=\pi_{m}^{-1}\cdot {\pi'}^{-1}=t_{\kappa_{+;1}}^{\vartheta_{1}}\cdot (t_{\kappa_{-;2}}^{\varrho_{1}}\cdot t_{\kappa_{+;2}}^{\vartheta_{2}})\cdots (t_{\kappa_{-;m-2}}^{\varrho_{m-3}}\cdot t_{\kappa_{+;m-2}}^{\vartheta_{m-2}})\cdot t_{\kappa_{-;m-1}}^{\varrho_{m-2}}\cdot (\pi_{m}^{-1}\cdot t_{k_{m-1}}^{\vartheta_{m-1}}),$$
in case $k_{1}>\rho_{1}$,
and
$$\pi^{-1}=\pi_{m}^{-1}\cdot {\pi'}^{-1}=(t_{\kappa_{-;2}}^{\varrho_{1}}\cdot t_{\kappa_{+;2}}^{\vartheta_{2}})\cdots (t_{\kappa_{-;m-2}}^{\varrho_{m-3}}\cdot t_{\kappa_{+;m-2}}^{\vartheta_{m-2}})\cdot t_{\kappa_{-;m-1}}^{\varrho_{m-2}}\cdot (\pi_{m}^{-1}\cdot t_{k_{m-1}}^{\vartheta_{m-1}}),$$
in case $k_{1}=\rho_{1}$.
Now, notice,  $$\pi_{m}^{-1}\cdot t_{k_{m-1}}^{\vartheta_{m-1}}=t_{\kappa_{-;m}}^{\kappa_{+;m-1}}\cdot t_{k_{m}}^{k_{m}-k_{m-1}}\cdot t_{k_{m-1}}^{\vartheta_{m-1}}.$$
Therefore, by Proposition \ref{exchange-2}, using the exchange law $t_{q}^{i_{q}}\cdot t_{p}^{i_{p}}$, for the case  $q-i_{q}=p$, where $q=k_{m}$, $p=k_{m-1}$, $i_{q}=k_{m}-k_{m-1}$, and $i_{p}=\vartheta_{m-1}$:
$$t_{k_{m}}^{k_{m}-k_{m-1}}\cdot t_{k_{m-1}}^{\vartheta_{m-1}}=t_{\vartheta_{m}}^{k_{m}-k_{m-1}}\cdot t_{k_{m}}^{\vartheta_{m}}.$$
Again by Proposition \ref{exchange-2}, now using the exchange law $t_{q}^{i_{q}}\cdot t_{p}^{i_{p}}$, for the case  $q-i_{q}=i_{p}$, where $q=\kappa_{-;m}$, $p=\vartheta_{m}$, $i_{q}=\kappa_{+;m-1}$, and $i_{p}=k_{m}-k_{1}$:
$$t_{\kappa_{-;m}}^{\kappa_{+;m-1}}\cdot t_{\vartheta_{m}}^{k_{m}-k_{m-1}}=t_{\kappa_{+;m-1}}^{\vartheta_{m}-(k_{m}-k_{m-1})}\cdot t_{\kappa_{-;m}}^{\kappa_{-;m}-\vartheta_{m}}=t_{\kappa_{+;m-1}}^{\vartheta_{m-1}}\cdot t_{\kappa_{-;m}}^{\varrho_{m-1}},$$
since  $\vartheta_{m-1}=\vartheta_{m}-(k_{m}-k_{m-1})$ and $\kappa_{-;m}-\vartheta_{m}=\varrho_{m-1}$.

Thus, the standard $OGS$ canonical form of $\pi^{-1}$ as follows:
\begingroup
\large
\begin{equation}\label{elementary-1}
\pi^{-1}=t_{\kappa_{+;1}}^{\vartheta_{1}}\cdot (t_{\kappa_{-;2}}^{\varrho_{1}}\cdot t_{\kappa_{+;2}}^{\vartheta_{2}})\cdot (t_{\kappa_{-;3}}^{\varrho_{2}}\cdot t_{\kappa_{+;3}}^{\vartheta_{3}})\cdots (t_{\kappa_{-;m}}^{\varrho_{m-1}}\cdot t_{\kappa_{+;m}}^{\vartheta_{m}}).
\end{equation}
\endgroup
In the specific case, where $k_{1}=maj\left(\pi\right)$:
\begingroup
\large
\begin{equation}\label{elementary-2}
\pi^{-1}=(t_{\kappa_{-;2}}^{\varrho_{1}}\cdot t_{\kappa_{+;2}}^{\vartheta_{2}})\cdot (t_{\kappa_{-;3}}^{\varrho_{2}}\cdot t_{\kappa_{+;3}}^{\vartheta_{3}})\cdots (t_{\kappa_{-;m}}^{\varrho_{m-1}}\cdot t_{\kappa_{+;m}}^{\vartheta_{m}}).
\end{equation}
\endgroup
Since for every $1\leq j\leq m$, $\varrho_{j-1}+\vartheta_{j}=\kappa_{-;j}$, obviously, Equations \ref{elementary-1} and \ref{elementary-2}, are standard $OGS$ elementary factorizations of $\pi^{-1}$ with the following properties:
\begin{itemize}
\item If $k_{1}>maj\left(\pi\right)$, ~then ~$\pi^{-1}$ ~has ~$m$ ~elementary factors of the form ~$t_{\kappa_{-;j}}^{\varrho_{j-1}}\cdot t_{\kappa_{+;j}}^{\vartheta_{j}}$ ~(we consider $\varrho_{0}=0$), ~for $1\leq j\leq m$;
\item If $k_{1}=maj\left(\pi\right)$, ~then ~$\pi^{-1}$ ~has  ~$m-1$ ~elementary factors of the form ~$t_{\kappa_{-;j}}^{\varrho_{j-1}}\cdot t_{\kappa_{+;j}}^{\vartheta_{j}}$, ~for $2\leq j\leq m$;
\end{itemize}
 Hence, we get the desired result of the theorem about the descent set of $\pi^{-1}$.
\end{proof}

\begin{example}
Let $\pi=t_{17}^{9}\cdot t_{19}^{2}\cdot t_{22}^{3}\cdot t_{24}^{3}$. Since $k_{1}=maj\left(\pi\right)=17$, $\pi$ is a standard $OGS$ elementary element of $S_{24}$.
The permutation presentation of $\pi$:
$$\pi=[1; ~2; ~3; ~4; ~5; ~6; ~7; ~8; ~9; ~12; ~13; ~17; ~18; ~19; ~22; ~23; ~24; ~10; ~11; ~14; ~15; ~16; ~20; ~21].$$
Notice, $\pi(j)=j$ for $1\leq j\leq 9$, which verifies the result of Proposition \ref{k-parabolic},  since $k_{1}=maj\left(\pi\right)=17$, and $i_{k_{1}}=9$.
By Theorem \ref{basic-canonical}  we have:
\begin{align*}
norm(\pi) &= s_{17}\cdot s_{16}\cdot s_{15}\cdot s_{14}\cdot s_{13}\cdot s_{12}\cdot s_{11}\cdot s_{10}\cdot \\ & ~~\cdot s_{18}\cdot s_{17}\cdot s_{16}\cdot s_{15}\cdot s_{14}\cdot s_{13}\cdot s_{12}\cdot s_{11}\cdot \\ & ~~\cdot s_{19}\cdot s_{18}\cdot s_{17}\cdot s_{16}\cdot s_{15}\cdot s_{14}\cdot \\ & ~~\cdot s_{20}\cdot s_{19}\cdot s_{18}\cdot s_{17}\cdot s_{16}\cdot s_{15}\cdot \\ & ~~\cdot s_{21}\cdot s_{20}\cdot s_{19}\cdot s_{18}\cdot s_{17}\cdot s_{16}\cdot \\ & ~~\cdot s_{22}\cdot s_{21}\cdot s_{20}\cdot \\ & ~~\cdot s_{23}\cdot s_{22}\cdot s_{21}
\end{align*}
We derive $norm(\pi^{-1})$, which we get just by reading the letters of $\pi$ in the direction of right to left, and down to up.
Thus,
\begin{align*}
norm(\pi^{-1}) &= ~~~~~~~~~~~~~~~~~~~~~~~~~~~~ s_{11}\cdot s_{10}\cdot \\ & ~~~~~~~~~~~~~~~~~~~~~~~~~~~~~~\cdot s_{12}\cdot s_{11}\cdot \\ & ~~~~~~~~~~~~~\cdot s_{16}\cdot s_{15}\cdot s_{14}\cdot s_{13}\cdot s_{12}\cdot \\ & ~~~~~~~~~~~~~\cdot s_{17}\cdot s_{16}\cdot s_{15}\cdot s_{14}\cdot s_{13}\cdot \\ & ~~~~~~~~~~~~~\cdot s_{18}\cdot s_{17}\cdot s_{16}\cdot s_{15}\cdot s_{14}\cdot \\ &~~\cdot s_{21}\cdot s_{20}\cdot s_{19}\cdot s_{18}\cdot s_{17}\cdot s_{16}\cdot s_{15}\cdot \\ &~~\cdot s_{22}\cdot s_{21}\cdot s_{20}\cdot s_{19}\cdot s_{18}\cdot s_{17}\cdot s_{16}\cdot \\ &~~\cdot s_{23}\cdot s_{22}\cdot s_{21}\cdot s_{20}\cdot s_{19}\cdot s_{18}\cdot s_{17}
\end{align*}

Now, we compute the standard $OGS$ canonical form of $\pi^{-1}$ by the formula in Theorem \ref{elementary-inverse}.

$$k_{1}=17, ~~k_{2}=19, ~~k_{3}=22, ~~k_{4}=24;$$
$$i_{k_{1}}=9, ~~i_{k_{2}}=2, ~~i_{k_{3}}=3, ~~i_{k_{4}}=3.$$
Thus,
$$\varrho_{1}=9, ~~\varrho_{2}=9+2=11, ~~\varrho_{3}=9+2+3=14, ~~\varrho_{4}=maj\left(\pi\right)=9+2+3+3=17;$$

$$\vartheta_{1}=17-17=0, ~~\vartheta_{2}=19-17=2, ~~\vartheta_{3}=22-17=5, ~~\vartheta_{4}=24-17=7.$$
and
$$\kappa_{-;2}=\varrho_{1}+\vartheta_{2}=9+2=11, ~~~~\kappa_{+;2}=\varrho_{2}+\vartheta_{2}=11+2=13,$$
$$\kappa_{-;3}=\varrho_{2}+\vartheta_{3}=11+5=16, ~~~~\kappa_{+;3}=\varrho_{3}+\vartheta_{3}=14+5=19,$$
$$\kappa_{-;4}=\varrho_{3}+\vartheta_{4}=14+7=21, ~~~~\kappa_{+;4}=\varrho_{4}+\vartheta_{4}=17+7=24.$$
Thus, by using Theorem \ref{elementary-inverse} for the case of $k_{1}=\rho_{1}$,
$$\pi^{-1}=t_{11}^{9}\cdot t_{13}^{2}\cdot t_{16}^{11}\cdot t_{19}^{5}\cdot t_{21}^{14}\cdot t_{24}^{7}.$$
Indeed, the permutation presentation of $\pi^{-1}$:
$$\pi^{-1}=[1; ~2; ~3; ~4; ~5; ~6; ~7; ~8; ~9; ~18; ~19; ~10; ~11; ~20; ~21; ~22; ~12; ~13; ~14; ~23; ~24; ~15; ~16; ~17].$$
Notice:
\begin{itemize}
\item $\kappa_{-;2}=11$ and $\pi^{-1}(11)=19>10=\pi^{-1}(12)$;
\item $\kappa_{-;3}=16$ and $\pi^{-1}(16)=22>12=\pi^{-1}(17)$;
\item $\kappa_{-;4}=21$ and $\pi^{-1}(21)=24>15=\pi^{-1}(22)$.
\end{itemize}
Thus, $$des\left(\pi^{-1}\right)=\{\kappa_{-;2}, ~\kappa_{-;3}, ~\kappa_{-;4}\}=\{11, ~16, ~21\}.$$
\end{example}

Now, we generalize Theorem \ref{elementary-inverse} for general $\pi\in S_{n}$.  We consider the standard $OGS$ elementary factorization
\begingroup
\large
$\pi=\prod_{v=1}^{z(\pi)}\pi^{(v)}$,
 \endgroup
 with the standard $OGS$ elementary factor
 \begingroup
 \large
 $\pi^{(v)}=\prod_{j=1}^{m^{(v)}}t_{h^{(v)}_{j}}^{\imath_{j}^{(v)}}$,
  \endgroup
  for $1\leq v\leq z(\pi)$. First, we show the standard $OGS$ canonical form for  $\pi^{-1}$ in a special case, where $\pi^{(v)}$ and $\pi^{(w)}$ are commute for all $1\leq v,w\leq z(\pi)$.
 \\
  Recall from Definition \ref{rho}:
\begin{itemize}
\item $\rho^{(v)}_{j}$ is defined to be $\sum_{x=j}^{m^{(v)}}\imath_{x}^{(v)}$;
\item $\varrho^{(v)}_{j}$ is defined to be $\sum_{x=1}^{j}\imath_{x}^{(v)}$;
\item $\vartheta^{(v)}_{j}$ is defined to be $h^{(v)}_{j}-\rho^{(v)}_{1}$.
\end{itemize}

\begin{theorem}\label{elementary-commute}
Let $\pi\in S_n$, presented in standard $OGS$ elementary factorization, with all the notations used in Definition \ref{canonical-factorization-def}. Assume $\pi^{(v)}$ satisfies  the following conditions:
\begin{itemize}
\item $maj\left(\pi^{(v)}\right)=h^{(v)}_{1}$ for $2\leq v\leq z(\pi)$;
\item $i_{1}^{(v)}\geq h^{(v-1)}_{ m^{(v-1)}}$ for $2\leq v\leq z(\pi)$.
\end{itemize}
For every $1\leq v\leq z(\pi)$, $1\leq j\leq m^{(v)}$, and $v\leq r\leq z(\pi)$, let $\varrho^{(v)}_{j}$ and $\vartheta^{(v)}_{j}$  be integers as defined in Definition \ref{rho}. Denote  ~$\kappa^{(v)}_{-;j}=\varrho^{(v)}_{j-1}+\vartheta^{(v)}_{j}=h^{(v)}_{j}-\rho^{(v)}_{j}$ and \\ $\kappa^{(v)}_{+;j}=\varrho^{(v)}_{j}+\vartheta^{(v)}_{j}=h^{(v)}_{j}-\rho^{(v)}_{j+1}$ ~(We consider $\varrho^{(v)}_{0}=\rho^{(v)}_{m^{(v)}+1}=0$). Then, the following holds:
\begin{itemize}
\item $\pi^{(v)}$ and $\pi^{(w)}$ are commute for all $1\leq v,w\leq z(\pi)$;
\begingroup
\Large
\item $$\pi^{-1}=\prod_{v=1}^{z(\pi)}\prod_{j=1}^{m^{(v)}}(t_{\kappa^{(v)}_{-;j}}^{\varrho^{(v)}_{j-1}}\cdot t_{\kappa^{(v)}_{+;j}}^{\vartheta^{(v)}_{j}}).$$
\endgroup
\end{itemize}
\end{theorem}

\begin{proof}

Let $\pi^{(v)}=\prod_{j=1}^{m^{(v)}}t_{h^{(v)}_{j}}^{\imath_{j}^{(v)}}$, for $1\leq v\leq z(\pi)$.
Since for every $w>v$, ~$maj\left(\pi^{(w)}\right)=h^{(w)}_{1}$, and $i_{1}^{(w)}\geq h^{(w-1)}_{ m^{(w-1)}}\geq h^{(v)}_{ m^{(v)}}$, by Proposition \ref{commute-perm}, $\pi^{(v)}$ commutes with $\pi^{(w)}$, for every $1\leq v<w\leq z(\pi)$. Therefore, $\pi^{-1}=\prod_{v=1}^{z(\pi)}{\left(\pi^{(v)}\right)}^{-1}$.
 Then, by using Theorem \ref{elementary-inverse}, we get the desired result.
\end{proof}
\\

Now, we turn to the general case, where considering $\pi\in S_{n}$ such that the standard $OGS$ elementary factors do not necessarily commute. In the first step, we generalize the definitions of $\kappa^{(v)}_{-;j}$ and $\kappa^{(v)}_{+;j}$ for an arbitrary $\pi\in S_n$, which were defined in Theorems \ref{elementary-inverse} and \ref{elementary-commute} for the special cases of $\pi$.
\\

\begin{definition}\label{kappa}
For every $1\leq v\leq z(\pi)$, $1\leq j\leq m^{(v)}$, and $v\leq r\leq z(\pi)$, define $\kappa^{(v\rightarrow r)}_{-;j}$ and $\kappa^{(v\rightarrow r)}_{+;j}$ in the following recursive way:
\begin{itemize}
\item $$\kappa^{(v\rightarrow v)}_{-;j}=\vartheta^{(v)}_{j}+\varrho^{(v)}_{j-1}=h^{(v)}_{j}-\rho^{(v)}_{j},$$  $$\kappa^{(v\rightarrow v)}_{+;j}=\vartheta^{(v)}_{j}+\varrho^{(v)}_{j}=h^{(v)}_{j}-\rho^{(v)}_{j+1},$$
for every $1\leq v\leq z(\pi)$  and for every $1\leq j\leq m^{(v)}$ (where in the formula for $j=1$, and for $j=m^{(v)}$ we consider $\varrho^{(v)}_{0}=\rho^{(v)}_{m^{(v)}+1}=0$);
\item If ~ $\varrho^{(r+1)}_{j'-1}\leq\kappa^{(v\rightarrow r)}_{-;j}<\varrho^{(r+1)}_{j'}$, ~for some ~$1\leq j'\leq m^{(r+1)}$,
then $$\kappa^{(v\rightarrow r+1)}_{-;j}=\kappa^{(v\rightarrow r)}_{-;j}+\vartheta^{(r+1)}_{j'};$$
\item If ~ $\varrho^{(r+1)}_{j''-1}<\kappa^{(v\rightarrow r)}_{+;j}\leq\varrho^{(r+1)}_{j''}$, ~for some ~$1\leq j''\leq m^{(r+1)}$,
then $$\kappa^{(v\rightarrow r+1)}_{+;j}=\kappa^{(v\rightarrow r)}_{+;j}+\vartheta^{(r+1)}_{j''};$$
\end{itemize}
In case of $r=z(\pi)$:
\begin{itemize}
\item Denote $\kappa^{(v\rightarrow z(\pi))}_{-;j}$ ~by just $\kappa^{(v)}_{-;j}$;
\item Denote $\kappa^{(v\rightarrow z(\pi))}_{+;j}$ ~by just $\kappa^{(v)}_{+;j}$.
\end{itemize}
\end{definition}

The next proposition shows that Definition \ref{kappa} just generalizes the definitions  of $\kappa^{(v)}_{-;j}$ and  $\kappa^{(v)}_{+;j}$  in the special case of Theorem \ref{elementary-commute}.
\\

\begin{proposition}\label{kappa-commute}
If $\kappa^{(v\rightarrow r)}_{+;m^{(v)}}=\kappa^{(v\rightarrow v)}_{+;m^{(v)}}$, for every $1\leq v\leq z(\pi)$, ~$v\leq r\leq z(\pi)$, then the following holds:
\begin{itemize}
\item $\kappa^{(v\rightarrow r)}_{-;j}=\kappa^{(v\rightarrow v)}_{-;j}$ and $\kappa^{(v\rightarrow r)}_{+;j}=\kappa^{(v\rightarrow v)}_{+;j}$, for every $1\leq v\leq z(\pi)$, ~$v\leq r\leq z(\pi)$, and $1\leq j\leq m^{(v)}$;
\item $\pi^{(v)}$ commutes with $\pi^{(r)}$, for every $1\leq v\leq z(\pi)$ and $v\leq r\leq z(\pi)$.
\end{itemize}
\end{proposition}

\begin{proof}

Assume $\kappa^{(v\rightarrow r)}_{+;m^{(v)}}=\kappa^{(v\rightarrow v)}_{+;m^{(v)}}$, for every $1\leq v\leq z(\pi)$ and $v\leq r\leq z(\pi)$.
Then, in particular, $$h^{(v)}_{m^{(v)}}=\kappa^{(v\rightarrow v)}_{+;m^{(v)}}=\kappa^{(v\rightarrow v+1)}_{+;m^{(v)}}$$ for every $1\leq v<z(\pi)$. Thus, by Definition \ref{kappa},  $$h^{(v)}_{m^{(v)}}=\kappa^{(v\rightarrow v)}_{+;m^{(v)}}\leq\varrho^{(v+1)}_{j'}$$ such that $\vartheta^{(v+1)}_{j'}=0$. Then, by Proposition \ref{order-rho}, $$j'=1 ~~~~ maj\left(\pi^{(v+1)}\right)=h^{(v+1)}_{1}.$$
Notice, by Definition \ref{kappa}, $\kappa^{(v\rightarrow v)}_{-;j}<\kappa^{(v\rightarrow v)}_{+;m^{(v)}}$ and $\kappa^{(v\rightarrow v)}_{+;j}\leq\kappa^{(v\rightarrow v)}_{+;m^{(v)}}$ ~for every \\ $1\leq j\leq m^{(v)}$. Thus, we get
$$\kappa^{(v\rightarrow v)}_{-;j}<\varrho^{(v+1)}_{1}, ~~~~\kappa^{(v\rightarrow v)}_{+;j}\leq\varrho^{(v+1)}_{1}$$
as well, for every ~$1\leq j\leq m^{(v)}$. Hence,
$$\kappa^{(v\rightarrow v+1)}_{-;j}=\kappa^{(v\rightarrow v)}_{-;j}+\vartheta^{(v+1)}_{1}=\kappa^{(v\rightarrow v)}_{-;j}, ~~\kappa^{(v\rightarrow v+1)}_{+;j}=\kappa^{(v\rightarrow v)}_{+;j}+\vartheta^{(v+1)}_{1}=\kappa^{(v\rightarrow v)}_{+;j},$$
for every $1\leq v<z(\pi)$ ~and ~$1\leq j\leq m^{(v)}$. Hence, $\kappa^{(v\rightarrow r)}_{-;j}=\kappa^{(v\rightarrow v)}_{-;j}$ and $\kappa^{(v\rightarrow r)}_{+;j}=\kappa^{(v\rightarrow v)}_{+;j}$, for every $1\leq v\leq z(\pi)$, ~$v\leq r\leq z(\pi)$, and $1\leq j\leq m^{(v)}$.
Now, since  $$\kappa^{(v\rightarrow v)}_{+;m^{(v)}}=h^{(v)}_{m^{(v)}}\leq\varrho^{(v+1)}_{1}=\imath_{1}^{(v+1)},  ~~~~ maj\left(\pi^{(v+1)}\right)=h^{(v+1)}_{1},$$ by Theorem \ref{elementary-commute}, $\pi^{(v)}$ commutes with $\pi^{(r)}$ for every $1\leq v<z(\pi)$, and $v<r\leq z(\pi)$.
\end{proof}
\\

The next six propositions describe some important properties of $\kappa^{(v\rightarrow r)}_{-;j}$ and $\kappa^{(v\rightarrow r)}_{+;j}$, which are essential for understanding the formula of the standard $OGS$ canonical form of $\pi^{-1}$.

\begin{proposition}\label{order-kappa}
For every $1\leq v\leq z(\pi)$, $1\leq j\leq m^{(v)}$, and $v\leq r\leq z(\pi)$, ~$\kappa^{(v\rightarrow r)}_{-;j}$ and  $\kappa^{(v\rightarrow r)}_{+;j}$ satisfy the following properties:
\begin{itemize}
\item $\kappa^{(v\rightarrow r)}_{-;j}<\kappa^{(v\rightarrow r)}_{+;j}$,
and in particular,
$\kappa^{(v)}_{-;j}<\kappa^{(v)}_{+;j}$,
for every $1\leq v\leq z(\pi)$, \\ $v\leq r\leq z(\pi)$, and  $1\leq j\leq m^{(v)}$;
\item $\kappa^{(v\rightarrow r)}_{+;j}<\kappa^{(v\rightarrow r)}_{-;j+1}$,
and in particular,
$\kappa^{(v)}_{+;j}<\kappa^{(v)}_{-;j+1}$,
for every $1\leq v\leq z(\pi)$, \\ $v\leq r\leq z(\pi)$, and $2\leq j\leq m^{(v)}-1$.
\end{itemize}
\end{proposition}

\begin{proof}
The proof is in induction on $r$.  For $r=v$, the results are obvious by the definitions. Assume by induction the correctness of the proposition for $r'\leq r$, and we prove it for $r'=r+1$. By Definition \ref{kappa}, \\ \\ $\kappa^{(v\rightarrow r+1)}_{-;j}=\kappa^{(v\rightarrow r)}_{-;j}+\vartheta^{(r+1)}_{j'}$, ~~where ~~$\varrho^{(r+1)}_{j'-1}\leq\kappa^{(v\rightarrow r)}_{-;j}<\varrho^{(r+1)}_{j'}$, \\ \\ $\kappa^{(v\rightarrow r+1)}_{+;j}=\kappa^{(v\rightarrow r)}_{+;j}+\vartheta^{(r+1)}_{j''}$, ~~where ~~$\varrho^{(r+1)}_{j''-1}<\kappa^{(v\rightarrow r)}_{+;j}\leq\varrho^{(r+1)}_{j''}$. \\

By the induction hypothesis, $$\kappa^{(v\rightarrow r)}_{-;j}<\kappa^{(v\rightarrow r)}_{+;j}, ~~~~\kappa^{(v\rightarrow r)}_{+;j}<\kappa^{(v\rightarrow r)}_{-;j+1}.$$ Therefore, $$\varrho^{(r+1)}_{j'-1}\leq\kappa^{(v\rightarrow r)}_{-;j}<\kappa^{(v\rightarrow r)}_{+;j}\leq\varrho^{(r+1)}_{j''}.$$ Thus, by Proposition \ref{order-rho}, $j'\leq j''$, and $\vartheta^{(r+1)}_{j'}\leq \vartheta^{(r+1)}_{j''}$. Hence, $$\kappa^{(v\rightarrow r+1)}_{-;j}=\kappa^{(v\rightarrow r)}_{-;j}+\vartheta^{(r+1)}_{j'}<\kappa^{(v\rightarrow r)}_{+;j}+\vartheta^{(r+1)}_{j''}=\kappa^{(v\rightarrow r+1)}_{+;j},$$ for every $1\leq v\leq z(\pi)$, $1\leq j\leq m^{(v)}$. By a similar argument, we have $$\kappa^{(v\rightarrow r+1)}_{+;j}<\kappa^{(v\rightarrow r+1)}_{-;j+1},$$ for every $1\leq v\leq z(\pi)$, $1\leq j\leq m^{(v)}-1$.
\end{proof}

\begin{proposition}\label{v1-x-v2}
For every $1\leq v<z(\pi)$ and $w>v$, let ~$j'_{w:v}$ ~and ~$j''_{w:v}$  ~be integers, such that
$$\varrho^{(w)}_{j'_{w:v}-1}\leq \kappa^{(v\rightarrow w-1)}_{-;j}<\varrho^{(w)}_{j'_{w:v}},  ~~~~\varrho^{(w)}_{j''_{w:v}-1}<\kappa^{(v\rightarrow w-1)}_{+;j}\leq\varrho^{(w)}_{j''_{w:v}}.$$
Then, the following holds:

$$\kappa^{(w\rightarrow r)}_{-;j'_{w:v}}\leq \kappa^{(v\rightarrow r)}_{-;j}<\kappa^{(w\rightarrow r)}_{+;j'_{w:v}}, ~~~~\kappa^{(w\rightarrow r)}_{-;j''_{w:v}}<\kappa^{(v\rightarrow r)}_{+;j}\leq\kappa^{(w\rightarrow r)}_{+;j''_{w:v}}.$$
for every $r$ such that $w\leq r\leq z(\pi)$.
In particular,
$$\kappa^{(w)}_{-;j'_{w:v}}\leq \kappa^{(v)}_{-;j}<\kappa^{(w)}_{+;j'_{w:v}}, ~~~~\kappa^{(w)}_{-;j''_{w:v}}<\kappa^{(v)}_{+;j}\leq\kappa^{(w)}_{+;j''_{w:v}}.$$
\end{proposition}

\begin{proof}
The proof of the result is in induction on the value of $r$. By Definition \ref{kappa}, we have:

$\kappa^{(v\rightarrow w)}_{-;j}=\kappa^{(v\rightarrow w-1)}_{-;j}+\vartheta^{(w)}_{j'_{w:v}}$, ~in case ~$\varrho^{(w)}_{j'_{w:v}-1}\leq\kappa^{(v\rightarrow w-1)}_{-;j}<\varrho^{(w)}_{j'_{w:v}}$,
  \\

  Where, $$\kappa^{(w\rightarrow w)}_{-;j'_{w:v}}=\varrho^{(w)}_{j'_{w:v}-1}+\vartheta^{(w)}_{j'_{w:v}}, ~~~~\kappa^{(w\rightarrow w)}_{+;j'_{w:v}}=\varrho^{(w)}_{j'_{w:v}}+\vartheta^{(w)}_{j'_{w:v}}.$$ Therefore, by again using Definition \ref{kappa}, we have  $\kappa^{(w\rightarrow w)}_{-;j'_{w:v}}\leq\kappa^{(v\rightarrow w)}_{-;j}<\kappa^{(w\rightarrow w)}_{+;j'_{w:v}}$, ~for every $w>v$.
  \\

  Now assume by induction ~$\kappa^{(w\rightarrow r')}_{-;j'_{w:v}}\leq\kappa^{(v\rightarrow r')}_{-;j}<\kappa^{(w\rightarrow r')}_{+;j'_{w:v}}$ ~for every $r'$ such that \\ $w\leq r'\leq r$. Now, by Definition \ref{kappa}, notice the following,
\begin{itemize}
\item $\kappa^{(w\rightarrow r+1)}_{-;j'_{w:v}}=\kappa^{(w\rightarrow r)}_{-;j'_{w:v}}+\vartheta^{(r+1)}_{\hat{j}}$ ~ in case ~$\varrho^{(r+1)}_{\hat{j}-1}\leq\kappa^{(w\rightarrow r)}_{-;j'_{w:v}}<\varrho^{(r+1)}_{\hat{j}}$;
\item $\kappa^{(v\rightarrow r+1)}_{-;j}=\kappa^{(v\rightarrow r)}_{-;j}+\vartheta^{(r+1)}_{\hat{\hat{j}}}$ ~ in case
    ~$\varrho^{(r+1)}_{\hat{\hat{j}}-1}\leq\kappa^{(v\rightarrow r)}_{-;j}<\varrho^{(r+1)}_{\hat{\hat{j}}}$;
\item $\kappa^{(w\rightarrow r+1)}_{+;j'_{w:v}}=\kappa^{(w\rightarrow r)}_{+;j'_{w:v}}+\vartheta^{(r+1)}_{\hat{\hat{\hat{j}}}}$ ~ in case ~$\varrho^{(r+1)}_{\hat{\hat{\hat{j}}}-1}<\kappa^{(w\rightarrow r)}_{+;j'_{w:v}}\leq\varrho^{(r+1)}_{\hat{\hat{\hat{j}}}}$;
\end{itemize}
Now, by the induction hypothesis, ~$\kappa^{(w\rightarrow r)}_{-;j'_{w:v}}\leq \kappa^{(v\rightarrow r)}_{-;j}<\kappa^{(w\rightarrow r)}_{+;j'_{w:v}}$. \\

Therefore, $\varrho^{(r+1)}_{\hat{j}}\leq \varrho^{(r+1)}_{\hat{\hat{j}}}\leq\varrho^{(r+1)}_{\hat{\hat{\hat{j}}}}$, which by Proposition \ref{order-rho} implies:  $\hat{j}\leq \hat{\hat{j}}\leq\hat{\hat{\hat{j}}}$.
\\

Then, by Proposition \ref{order-rho}, ~$\vartheta^{(r+1)}_{\hat{j}}\leq \vartheta^{(r+1)}_{\hat{\hat{j}}}\leq\vartheta^{(r+1)}_{\hat{\hat{\hat{j}}}}$.
\\

 Thus, we get $\kappa^{(w\rightarrow r)}_{-;j'_{w:v}}\leq \kappa^{(v\rightarrow r)}_{-;j}<\kappa^{(w\rightarrow r)}_{+;j'_{w:v}}$.
\\

By the same argument, we get
\\

$\kappa^{(w\rightarrow r)}_{-;j''_{w:v}}<\kappa^{(v\rightarrow r)}_{+;j}\leq\kappa^{(w\rightarrow r)}_{+;j''_{w:v}}$, which implies $\kappa^{(w\rightarrow r+1)}_{-;j''_{w:v}}<\kappa^{(v\rightarrow w)}_{+;j}\leq\kappa^{(w\rightarrow r+1)}_{+;j''_{w:v}}$, for every $1\leq v\leq z(\pi)-1$. Thus, the proposition holds for every $1\leq v\leq z(\pi)$, $w>v$, and $w\leq r\leq z(\pi)$.
\end{proof}

\begin{proposition}\label{v2-x-v1}
Assume one of the following conditions holds for some \\ $1\leq v, v'\leq z(\pi)$, ~$1\leq j\leq m^{(v)}-1$, ~$1\leq j'\leq m^{(v')}$, $v,v'\leq r\leq z(\pi)$, ~and ~$c\in\{+, -\}$:
\begin{itemize}
\item $\kappa^{(v\rightarrow r)}_{+;j}<\kappa^{(v'\rightarrow r)}_{c;j'}<\kappa^{(v\rightarrow r)}_{-;j+1}$;
\item $\kappa^{(v'\rightarrow r)}_{c;j'}<\kappa^{(v\rightarrow r)}_{-;1}$;
\item $\kappa^{(v\rightarrow r)}_{+;m^{(v)}}<\kappa^{(v'\rightarrow r)}_{c;j'}$.
\end{itemize}
then necessarily $v'>v$ (i.e.,the interval $\left(\kappa^{(v\rightarrow r)}_{+;j}, \kappa^{(v\rightarrow r)}_{-;j+1}\right)$ does not contain $\kappa^{(v'\rightarrow r)}_{c;j'}$ ~such that, $v'\leq v$).
\end{proposition}

\begin{proof}
Consider $\kappa^{(v'\rightarrow r)}_{c;j'}$. Then, by Proposition \ref{v1-x-v2},  for every $v>v'$ there exists $1\leq j\leq m^{(v)}$ such that
$\kappa^{(v\rightarrow r)}_{+;j}\leq\kappa^{(v'\rightarrow r)}_{c;j'}\leq\kappa^{(v\rightarrow r)}_{+;j}$. Therefore, the case where,
\\

 $\kappa^{(v\rightarrow r)}_{+;j}<\kappa^{(v'\rightarrow r)}_{c;j'}<\kappa^{(v\rightarrow r)}_{-;j+1} ~~~~ or ~~~~\kappa^{(v'\rightarrow r)}_{c;j'}<\kappa^{(v\rightarrow r)}_{-;1} ~~~~ or ~~~~ \kappa^{(v\rightarrow r)}_{+;m^{(v)}}<\kappa^{(v'\rightarrow r)}_{c;j'}$
 \\

 can not happen for $v>v'$. By Proposition \ref{order-kappa}, the mentioned cases can not happen for $v=v'$ as well. Therefore, necessarily, $v'>v$ under the conditions of the proposition.

\end{proof}

\begin{proposition}\label{no-common-eps12}
$\kappa^{(v')}_{-;j'}\neq \kappa^{(v'')}_{+;j''}$, for any $1\leq v', ~v''\leq z(\pi)$, $1\leq j'\leq m^{(v')}$, and  $1\leq j''\leq m^{(v'')}$.
\end{proposition}

\begin{proof}
If $v'=v''$ for some $1\leq v'\leq z(\pi)$, then by Proposition \ref{order-kappa}, we have $\kappa^{(v')}_{-;j'}\neq \kappa^{(v')}_{+;j''}$, for all $1\leq v' \leq z(\pi)$, ~$1\leq j', j''\leq m^{(v)}$. Now, assume $v'<v''$. Then, by Proposition \ref{v1-x-v2}, there exists $1\leq j\leq m^{(v'')}$, such that ~$\kappa^{(v'')}_{-;j}\leq \kappa^{(v')}_{-;j'}<\kappa^{(v'')}_{+;j}$.
\\

By our assumption, $\kappa^{(v')}_{-;j'}=\kappa^{(v'')}_{+;j''}$.
\\

Therefore, $\kappa^{(v'')}_{-;j}\leq \kappa^{(v'')}_{+;j''}<\kappa^{(v'')}_{+;j}$.
\\

By Proposition \ref{order-kappa}, $$\kappa^{(v'')}_{+;j}>\kappa^{(v'')}_{+;j''} ~~if ~and ~only ~if ~j>j'', ~~~~ \kappa^{(v'')}_{-;j}\leq\kappa^{(v'')}_{+;j''} ~~ if ~and ~only ~if ~j\leq j''.$$

Therefore, there is no $j$  such that ~$\kappa^{(v'')}_{-;j}\leq \kappa^{(v'')}_{+;j''}<\kappa^{(v'')}_{+;j}$. ~Thus, ~$\kappa^{(v')}_{-;j'}\neq \kappa^{(v'')}_{+;j''}$, ~for all $1\leq j'\leq m^{(v')}$, and  all $1\leq j''\leq m^{(v'')}$, in case $v'<v''$.

By a similar argument we have ~$\kappa^{(v')}_{-;j'}\neq \kappa^{(v'')}_{+;j''}$, ~for any $1\leq j'\leq m^{(v')}$ and  any $1\leq j''\leq m^{(v'')}$, in case $v'>v''$.
\end{proof}

\begin{proposition}\label{max-min-kappa}
For every integer $v_{1}$, ~$v_{2}$, ~and $j$, such that $1\leq v_{1}<v_{2}\leq z(\pi)$, and  $1\leq j\leq m^{(v_{1})}$, the parameters $\kappa^{(v_{1})}_{-;j}$, ~$\kappa^{(v_{1})}_{+;j}$, ~$\kappa^{(v_{2})}_{+;m^{(v_{2})}}$, ~$\kappa^{(v_{2})}_{-;1}$ ~satisfy the following properties:
\begin{itemize}
\item If $1\leq v_{1}<v_{2}\leq z(\pi)$, then $\kappa^{(v_{1})}_{+;j}\leq \kappa^{(v_{2})}_{+;m^{(v_{2})}}$ and $\kappa^{(v_{1})}_{-;j}< \kappa^{(v_{2})}_{+;m^{(v_{2})}}$, for every \\ $1\leq j\leq m^{(v)}$. In particular, $\kappa^{(z(\pi))}_{+;m^{(z(\pi))}}\geq \kappa^{(v)}_{c;j}$, for every $1\leq v\leq z(\pi)$, \\ $1\leq j\leq m^{(v)}$, $c\in\{+,-\}$;
\item If $1\leq v_{1}<v_{2}\leq z(\pi)$, then $\kappa^{(v_{2})}_{-;1}\leq \kappa^{(v_{1})}_{-;j}$  and $\kappa^{(v_{2})}_{-;1}< \kappa^{(v_{1})}_{+;j}$, for every $1\leq j\leq m^{(v)}$. In particular, $\kappa^{(z(\pi))}_{-;1}\leq \kappa^{(v)}_{c;j}$, for every $1\leq v\leq z(\pi)$, $1\leq j\leq m^{(v)}$, $c\in\{+,-\}$.
\end{itemize}
\end{proposition}
\begin{proof}
Let $v_{1}$, and $v_{2}$ be integers such that $1\leq v_{1}<v_{2}\leq z(\pi)$. Let $1\leq j\leq m^{(v_{1})}$. Then, by Proposition \ref{v1-x-v2}, there exists $1\leq j', j''\leq m^{(v_{2})}$ such that $\kappa^{(v_{1})}_{-;j'}<\kappa^{(v_{1})}_{+;j}\leq \kappa^{(v_{2})}_{+;j'}$ and $\kappa^{(v_{1})}_{-;j''}\leq\kappa^{(v_{1})}_{-;j}<\kappa^{(v_{2})}_{+;j''}$. By Proposition \ref{order-kappa}, $\kappa^{(v_{2})}_{+;j'}\leq \kappa^{(v_{2})}_{+;m^{(v_{2})}}$ and $\kappa^{(v_{2})}_{+;j''}\leq \kappa^{(v_{2})}_{+;m^{(v_{2})}}$. Hence, we get the first part of the proposition. By a similar argument we prove the second part of the proposition.
\end{proof}

\begin{proposition}\label{kappa-zero}
Assume $\kappa^{(v)}_{c;j}=0$ for some $1\leq v\leq z(\pi)$, ~$1\leq j\leq m^{(v)}$, ~and ~$c\in\{+, -\}$. Then, necessarily $j=1$, ~$c=(-)$, ~and for every $v'\geq v$, ~$m^{(v')}>1$, ~and both $\kappa^{(v')}_{-;1}=0$ ~and ~$\vartheta^{(v')}_{1}=0$. Moreover, if $\vartheta^{(v')}_{1}=0$ for every $v'\geq v$, for some $1\leq v\leq z(\pi)$, then necessarily $\kappa^{(v')}_{-;1}=0$ for every $v'\geq v$, as well.
\end{proposition}

\begin{proof}
Since $\kappa^{(v)}_{+;j}$ is the sum of $\varrho^{(v)}_{j}$ with elements of the form $\vartheta^{(v')}_{j'}$, where by Proposition \ref{order-rho}, $\varrho^{(v)}_{j}>0$  and  $\vartheta^{(v')}_{j'}\geq 0$ ~for every $1\leq v\leq z(\pi)$ and  $1\leq j\leq m^{(v)}$, obviously, $\kappa^{(v)}_{+;j}>0$. By the same argument, $\kappa^{(v)}_{-;j}>0$ for $j\geq 2$ as well. Thus, if  $\kappa^{(v)}_{-;j}=0$ for some $1\leq v\leq z(\pi)$ and $1\leq j\leq m^{(v)}$, then necessarily $j=1$. Now, by using Proposition \ref{max-min-kappa}, $\kappa^{(v')}_{-;1}\leq \kappa^{(v)}_{-;1}$ for every $v'>v$. Thus, we get $\kappa^{(v')}_{-;1}=0$ for every $v'\geq v$. Now, since $\kappa^{(v')}_{-;1}\geq \kappa^{(v'\rightarrow v')}_{-;1}=\vartheta^{(v')}_{1}$, and by Proposition \ref{order-rho}, $\vartheta^{(v')}_{1}\geq 0$, ~we conclude $\vartheta^{(v')}_{1}=0$ for every $v'\geq v$ as well, and therefore, $m^{(v')}>1$ by using the same proposition. Now, assume $\vartheta^{(v')}_{1}=0$ for every $v'\geq v$, for some $1\leq v\leq z(\pi)$, then by Definition \ref{kappa}, we get $\kappa^{(v')}_{-;1}=0$ for every $v'\geq v$, as well.
\end{proof}
\\

Now, we introduce two new parameters $\chi$ and $\eta$, which are defined by the different values of $\kappa^{(v)}_{-;j}$ and $\kappa^{(v)}_{+;j}$, for $1\leq v\leq z(\pi)$ and $1\leq j\leq m^{(v)}$, and which are essential for the formula of the standard $OGS$ canonical form of $\pi^{-1}$.
\\

\begin{definition}\label{chi}
For every $1\leq v\leq z(\pi)$ and $1\leq j<m^{(v)}$, ~$\chi^{(v)}_{j}$ is defined in the following way: \\
If there exists integers $v'<v$ and $1\leq j'<m^{(v')}$, such that
$$\kappa^{(v')}_{+;j'}\leq\kappa^{(v)}_{+;j}<\kappa^{(v')}_{-;j'+1}, ~~ or ~~ \kappa^{(v)}_{+;j}<\kappa^{(v')}_{-;1}, ~~ or ~~ \kappa^{(v')}_{+;m^{(v')}}\leq\kappa^{(v)}_{+;j},$$
then define $\chi^{(v)}_{j}$ to be the largest $v'$ such that $v'<v$, otherwise $\chi^{(v)}_{j}$ is defined to be $0$.
\end{definition}

\begin{definition}\label{eta}
For every $1\leq v\leq z(\pi)$ and $1\leq j<m^{(v)}$, ~$\eta^{(v)}_{j}$ is defined in the following way:

$$\eta^{(v)}_{j}=\sum_{v'=\chi^{(v)}_{j}+1}^{v-1}\vartheta^{(v)}_{j_{v':v}},$$
such that
$$\kappa^{(v')}_{-;j_{v':v}}<\kappa^{(v)}_{+;j}<\kappa^{(v')}_{+;j_{v':v}};$$
for some $1\leq j_{v':v}\leq m^{(v')}$.
\end{definition}

In the next three propositions we find important connections between the parameters  $\chi^{(v)}_{j}$, ~$\eta^{(v)}_{j}$, ~and the parameters $\varrho^{(v)}_{j}$, ~$\vartheta^{(v)}_{j}$, ~$\kappa^{(v)}_{-;j}$,  and ~$\kappa^{(v)}_{+;j}$, where the third proposition (Proposition \ref{inverse-factorization}), leads to the standard $OGS$ elementary factorization of $\pi^{-1}$.
\\
\begin{proposition}\label{eta-varrho}
For every $1\leq v\leq z(\pi)$ and $1\leq j\leq m^{(v)}$, $$\eta^{(v)}_{j}<\varrho^{(v)}_{j}.$$
\end{proposition}

\begin{proof}
The proof is in induction on the value of $v-\chi^{(v)}_{j}$. If $v-\chi^{(v)}_{j}=1$, then $\chi^{(v)}_{j}=v-1$, and $\eta^{(v)}_{j}=0$. Therefore, obviously, $\eta^{(v)}_{j}<\varrho^{(v)}_{j}$ in case $\eta^{(v)}_{j}=0$. Now, assume in induction that $\eta^{(v)}_{j}<\varrho^{(v)}_{j}$, in case $v-\chi^{(v)}_{j}\leq \hat{v}$, for some $\hat{v}\geq 1$, and we prove the proposition for the case $v-\chi^{(v)}_{j}=\hat{v}+1$. By Definition \ref{eta}, $\eta^{(v)}_{j}=\sum_{v'=\chi^{(v)}_{j}+1}^{v-1}\vartheta^{(v')}_{j_{v':v}}$, such that $\kappa^{(v')}_{-;j_{v':v}}<\kappa^{(v)}_{+;j}<\kappa^{(v')}_{+;j_{v':v}}$. Since $v-\chi^{(v)}_{j}=\hat{v}+1$, we get $\eta^{(v)}_{j}=\sum_{v'=v-\hat{v}}^{v-1}\vartheta^{(v')}_{j_{v':v}}$, such that $\kappa^{(v')}_{-;j_{v':v}}<\kappa^{(v)}_{+;j}<\kappa^{(v')}_{+;j_{v':v}}$. Thus, consider the values of $\kappa^{(v')}_{-;j_{v'v}}$ for $v-\hat{v}\leq v'\leq v-1$. Denote by $\breve{v}$ the smallest integer such that $v-\hat{v}\leq \breve{v}\leq v-1$ and for every $v-\hat{v}\leq v'\leq v-1$, $\kappa^{(v')}_{-;j_{v':v}}\leq \kappa^{(\breve{v})}_{-;j_{\breve{v}:v}}$. Then, the following holds:
\begin{itemize}
\item $\eta^{(\breve{v})}_{j_{\breve{v:v}}-1}=\sum_{v'=v-\hat{v}}^{\breve{v}-1}\vartheta^{(v')}_{j_{v':v}}$;
\item $\kappa^{(\breve{v}\rightarrow v-1)}_{-;j_{\breve{v}:v}}=\varrho^{(\breve{v})}_{j_{\breve{v}:v}-1}+\sum_{v'=\breve{v}}^{v-1}\vartheta^{(v')}_{j_{v':v}}$.
\end{itemize}
Since $\breve{v}<v$, by the induction hypothesis, $\eta^{(\breve{v})}_{j_{\breve{v}:v}-1}<\varrho^{(\breve{v})}_{j_{\breve{v}:v}-1}$, and since $\kappa^{(\breve{v})}_{-;j_{\breve{v}:v}}<\kappa^{(v)}_{+;j}$, by Definition \ref{kappa} we conclude, $\kappa^{(\breve{v}\rightarrow v-1)}_{-;j_{\breve{v}:v}}<\varrho^{(v)}_{j}$. Hence,
$$\eta^{(v)}_{j}=\eta^{(\breve{v})}_{j_{\breve{v}:v}-1}+\kappa^{(\breve{v}\rightarrow v-1)}_{-;j_{\breve{v}:v}}-\varrho^{(\breve{v})}_{j_{\breve{v}:v}-1}<\varrho^{(\breve{v})}_{j_{\breve{v}:v}-1}+\varrho^{(v)}_{j}-\varrho^{(\breve{v})}_{j_{\breve{v}:v}-1}=
\varrho^{(v)}_{j}.$$
\end{proof}

\begin{proposition}\label{kappa-adjacent}
Consider a non-decreasing monotonic  sequence  $\{a_{x}\}_{x=1}^{r}$ of all the integers of the form $\kappa^{(v)}_{-;j}$ and  $\kappa^{(v)}_{+;j}$, where $1\leq v\leq z(\pi)$ and $1\leq j\leq m^{(v)}$, with the following condition:
\begin{itemize}
\item If $\kappa^{(v_{1})}_{-;j_{1}}=\kappa^{(v_{2})}_{-;j_{2}}$ for some $1\leq v_{1}<v_{2}\leq z(\pi)$, and $1\leq j_{x}\leq m^{(v_{x})}$ then $\kappa^{(v_{2})}_{-;j_{2}}$ is prior to $\kappa^{(v_{1})}_{-;j_{1}}$ in the sequence  $\{a_{x}\}_{x=1}^{r}$;
\item If $\kappa^{(v_{1})}_{+;j_{1}}=\kappa^{(v_{2})}_{+;j_{2}}$ for some $1\leq v_{1}<v_{2}\leq z(\pi)$, and $1\leq j_{x}\leq m^{(v_{x})}$ then $\kappa^{(v_{1})}_{+;j_{1}}$ is prior to $\kappa^{(v_{2})}_{+;j_{2}}$ in the sequence  $\{a_{x}\}_{x=1}^{r}$.
\end{itemize}
 Then, the following holds:
\begin{itemize}
\item $a_{1}=\kappa^{(z(\pi))}_{-;1}$, $a_{r}=\kappa^{(z(\pi))}_{+;m^{(z(\pi))}}$;
\item If $a_{x}=\kappa^{(v_{x})}_{+;j_{x}}$ and $a_{x+1}=\kappa^{(v_{x+1})}_{c_{x+1};j_{x+1}}$, for some $1\leq x\leq r-1$, $1\leq j_{x}\leq m^{(v_{x})}$, $1\leq j_{x+1}\leq m^{(v_{x+1})}$, and $c_{x+1}\in \{-,+\}$, then $v_{x+1}\geq v_{x}$, where the following are satisfied:
    \begin{itemize}
    \item In case $c_{x+1}=(-)$, then necessarily, $v_{x+1}=v_{x}$ and $j_{x+1}=j_{x}+1$;
    \item In case $c_{x+1}=(+)$ where, $a_{x}<a_{x+1}$, then necessarily, $v_{x+1}>v_{x}$ and \\ $\chi^{(v_{x+1})}_{j_{x+1}}=v_{x}$.
    \end{itemize}
\item If $a_{x}=\kappa^{(v_{x})}_{-;j_{x}}$ and $a_{x-1}=\kappa^{(v_{x-1})}_{c_{x-1};j_{x-1}}$, for some $2\leq x\leq r$, $1\leq j_{x}\leq m^{(v_{x})}$, $1\leq j_{x-1}\leq m^{(v_{x-1})}$, and $c_{x-1}\in \{-,+\}$, then $v_{x-1}\geq v_{x}$, where the following are satisfied:
    \begin{itemize}
    \item In case $c_{x-1}=(-)$ where, $a_{x}>a_{x-1}$, necessarily, $v_{x-1}>v_{x}$ and  \\ $\chi^{(v_{x-1})}_{j_{x-1}-1}=v_{x}$;
    \item In case $c_{x-1}=(+)$, necessarily, $v_{x-1}=v_{x}$ and $j_{x-1}=j_{x}-1$.
    \end{itemize}
\item If $a_{x}=\kappa^{(v_{x})}_{-;j_{x}}$ and $a_{x+1}=\kappa^{(v_{x+1})}_{+;j_{x+1}}$, for some $1\leq x\leq r-1$, \\ $1\leq j_{x}\leq m^{(v_{x})}$, $1\leq j_{x+1}\leq m^{(v_{x+1})}$, then either $v_{x+1}=v_{x}$ and then $j_{x+1}=j_{x}$ as well, or necessarily, $\chi^{(v_{x})}_{j_{x}-1}=\chi^{(v_{x+1})}_{j_{x+1}}$.
\end{itemize}
\end{proposition}

\begin{proof}
The first part of the proposition is a direct consequence of Proposition \ref{max-min-kappa}. Thus, we turn to the second part of the proposition.\\
Assume, $a_{x}=\kappa^{(v_{x})}_{+;j_{x}}$  ~and ~$a_{x+1}=\kappa^{(v_{x+1})}_{c_{x+1};j_{x+1}}$, for some $1\leq x\leq r-1$, \\ $1\leq v_{x}, v_{x+1}\leq z(\pi)$,  $1\leq j_{x}\leq m^{(v_{x})}$, $1\leq j_{x+1}\leq m^{(v_{x+1})}$, and $c_{x+1}\in \{+,-\}$. Since $a_{x+1}$ is a successive element to $a_{x}=\kappa^{(v_{x})}_{+;j_{x}}$, necessarily $\kappa^{(v_{x})}_{+;j_{x}}\leq a_{x+1}\leq\kappa^{(v_{x})}_{-;j_{x}+1}$ (or just $\kappa^{(v_{x})}_{+;j_{x}}\leq a_{x+1}$, in case $j_{x}=m^{(v_{x})}$). Therefore, in case $a_{x+1}>a_{x}$, by Proposition \ref{v2-x-v1}, we have $v_{x+1}>v_{x}$, unless $a_{x+1}=\kappa^{(v_{x})}_{-;j_{x}+1}$. If $a_{x+1}=a_{x}$, then by the assumption of the proposition $v_{x+1}>v_{x}$. Now, assume $c_{x+1}=(-)$. Then, $a_{x+1}=\kappa^{(v_{x+1})}_{-;j_{x+1}}$. Therefore, by Proposition \ref{v2-x-v1}, we have $v_{x}>v_{x+1}$ unless, $a_{x+1}=\kappa^{(v_{x})}_{-;j_{x}+1}$. Hence, in case ~$c_{x+1}=(-)$ ~and ~$c_{x}=(+)$, the only possibly case is $$a_{x}=\kappa^{(v_{x})}_{+;j_{x}}, ~~~~a_{x+1}=\kappa^{(v_{x})}_{-;j_{x}+1}.$$
Now, assume $c_{x+1}=(+)$. Then, $a_{x}=\kappa^{(v_{x})}_{+;j_{x}}$  ~and ~$a_{x+1}=\kappa^{(v_{x+1})}_{+;j_{x+1}}$, where we have already shown $v_{x+1}>v_{x}$. By Proposition \ref{v1-x-v2}, for every $v'>v_{x}$, there exists $j'_{v':v_{x}}$, such that $\kappa^{(v')}_{-;j'_{v':v_{x}}}<a_{x}=\kappa^{(v_{x})}_{+;j_{x}}\leq \kappa^{(v')}_{+;j'_{v':v_{x}}}$. Since ~$a_{x+1}=\kappa^{(v_{x+1})}_{+;j_{x+1}}$ ~is a successive element to $a_{x}$, we have  $\kappa^{(v')}_{-;j'_{v':v_{x}}}<a_{x+1}=\kappa^{(v_{x+1})}_{+;j_{x+1}}\leq \kappa^{(v')}_{+;j'_{v':v_{x}}}$, for every $v_{x}+1\leq v'\leq v_{x+1}-1$. Thus, if $a_{x+1}>a_{x}$, then ~$\kappa^{(v')}_{-;j'_{v':v_{x}}}<a_{x+1}=\kappa^{(v_{x+1})}_{+;j_{x+1}}< \kappa^{(v')}_{+;j'_{v':v_{x}}}$, for every $v_{x}+1\leq v'\leq v_{x+1}-1$. Hence, by Definition \ref{chi}, $\chi^{(v_{x+1})}_{j_{x+1}}\leq v_{x}$. On the other hand, since
$a_{x+1}$ is successive to $a_{x}=\kappa^{(v_{x})}_{+;j_{x}}$, we have $\chi^{(v_{x+1})}_{j_{x+1}}\geq v_{x}$. Thus, we conclude $\chi^{(v_{x+1})}_{j_{x+1}}=v_{x}$ in case $a_{x+1}>a_{x}$ and $c_{x+1}=(+)$.
By a similar argument the third part of the proposition can be proved as well. Hence, we turn to the last part of the proposition. \\
Assume, $a_{x}=\kappa^{(v_{x})}_{-;j_{x}}$ and $a_{x+1}=\kappa^{(v_{x+1})}_{+;j_{x+1}}$, for some $1\leq x\leq r-1$, ~$1\leq j_{x}\leq m^{(v_{x})}$, \\ $1\leq j_{x+1}\leq m^{(v_{x+1})}$. If $v_{x+1}>v_{x}$, then by Proposition \ref{v1-x-v2}, for every \\ $v_{x}+1\leq v'\leq v_{x+1}-1$, there exists $1\leq j'_{v':v_{x}}\leq m^{(v')}$ such that
$$\kappa^{(v')}_{-;j'_{v':v_{x}}}\leq a_{x}=\kappa^{(v_{x})}_{-;j_{x}}<\kappa^{(v')}_{+;j'_{v':v_{x}}}.$$
Since $a_{x+1}=\kappa^{(v_{x+1})}_{+;j_{x+1}}$ is a successive to  $a_{x}=\kappa^{(v_{x})}_{-;j_{x}}$, we get
$$\kappa^{(v')}_{-;j'_{v':v_{x}}}<a_{x+1}=\kappa^{(v_{x+1})}_{+;j_{x+1}}<\kappa^{(v')}_{+;j'_{v':v_{x}}}$$
as well.
Thus, by Definition \ref{chi}, ~$\chi^{(v_{x+1})}_{j_{x+1}}\leq v_{x}$. On the other hand,  since
\\

$\kappa^{(v_{x})}_{-;j_{x}}=a_{x}<a_{x+1}=\kappa^{(v_{x+1})}_{+;j_{x+1}}$,
~and by Proposition \ref{order-kappa}, ~$\kappa^{(v_{x})}_{-;j_{x}}<\kappa^{(v_{x})}_{+;j_{x}}$,
we get $$\kappa^{(v_{x})}_{-;j_{x}}<\kappa^{(v_{x+1})}_{+;j_{x+1}}<\kappa^{(v_{x})}_{+;j_{x}}.$$
Moreover, since $a_{x+1}=\kappa^{(v_{x+1})}_{+;j_{x+1}}$ is a successive to  $a_{x}=\kappa^{(v_{x})}_{-;j_{x}}$, we have for every $v'<v_{x}$,  one of the following holds:
\begin{itemize}
\item $\kappa^{(v')}_{-;j'}\leq\kappa^{(v_{x})}_{-;j_{x}}<\kappa^{(v_{x+1})}_{+;j_{x+1}}\leq \kappa^{(v')}_{+;j'}$;
\item $\kappa^{(v')}_{+;j'}<\kappa^{(v_{x})}_{-;j_{x}}<\kappa^{(v_{x+1})}_{+;j_{x+1}}<\kappa^{(v')}_{-;j'+1}$;
\item $\kappa^{(v_{x})}_{-;j_{x}}<\kappa^{(v_{x+1})}_{+;j_{x+1}}<\kappa^{(v')}_{-;1}$;
\item $\kappa^{(v')}_{+;m^{(v')}}<\kappa^{(v_{x})}_{-;j_{x}}<\kappa^{(v_{x+1})}_{+;j_{x+1}}$.
\end{itemize}
for some $1\leq j'\leq m^{(v')}$.
Therefore,  by Definition \ref{chi}, we conclude ~$\chi^{(v_{x+1})}_{j_{x+1}}=\chi^{(v_{x})}_{j_{x}-1}$ in case $v_{x+1}>v_{x}$.
By a similar argument it can be shown that ~$\chi^{(v_{x+1})}_{j_{x+1}}=\chi^{(v_{x})}_{j_{x}-1}$ in case $v_{x+1}<v_{x}$ as well.
If $v_{x+1}=v_{x}$, then by Proposition \ref{order-kappa}, ~$a_{x}=\kappa^{(v_{x})}_{-;j_{x}}$ and $a_{x+1}=\kappa^{(v_{x})}_{+;j_{x}}$.
\end{proof}

\begin{proposition}\label{inverse-factorization}
Let $\{a_{x}\}_{x=1}^{r}$ be a non-decreasing monotonic sequence of $\kappa^{(v)}_{+;j}$ and $\kappa^{(v)}_{-;j}$ with the same conditions as in Proposition \ref{kappa-adjacent}. Let
\begingroup
\large
$\tilde{\pi}=\prod_{x=1}^{r}t_{a_{x}}^{i_{x}}$,
\endgroup
where,
$$i_{x}=\begin{cases}\vartheta^{(v)}_{j}+\eta^{(v)}_{j} & a_{x}=\kappa^{(v)}_{+;j} \\ \\ \kappa^{(v)}_{-;j}-(\vartheta^{(v)}_{j}+\eta^{(v)}_{j-1}) & a_{x}=\kappa^{(v)}_{-;j}\end{cases}.$$
For every $1\leq w\leq z(\tilde{\pi})$, ~consider $$\tilde{\pi}=\prod_{w=1}^{z(\tilde{\pi})}\tilde{\pi}^{(w)},$$
the standard $OGS$ elementary factorization of $\tilde{\pi}$, with the elementary factors
\begingroup
\large
$$\tilde{\pi}^{(w)}=\prod_{\jmath=1}^{m^{(w)}}t_{\tilde{h}^{(w)}_{\jmath}}^{\tilde{\imath}_{\jmath}^{(w)}}.$$
\endgroup
Define $\tilde{\pi}^{(w)}_{\jmath}$ to be the the following terminal segment of $\tilde{\pi}^{(w)}$:
\begingroup
\large
$$\tilde{\pi}^{(w)}_{\jmath}=\prod_{x=\jmath}^{m^{(w)}}t_{\tilde{h}^{(w)}_{\jmath}}^{\tilde{\imath}_{\jmath}^{(w)}}.$$
\endgroup
Then, the following holds:

\begin{itemize}
\item $z(\tilde{\pi})$, the number of standard $OGS$ elementary factors of $\tilde{\pi}$ equals to the number of different non-zero values $\kappa^{(v)}_{-;j}$, where $1\leq v\leq z(\pi)$, and $1\leq j\leq m^{(v)}$;
\item For every $1\leq w\leq z(\tilde{\pi})$, ~$maj\left(\tilde{\pi}^{(w)}\right)=\kappa^{(v)}_{-;j}$ for some $1\leq v\leq z(\pi)$, and for every given non-zero value of $\kappa^{(v)}_{-;j}$, there exists exactly one $w$ such that $$maj\left(\tilde{\pi}^{(w)}\right)=\kappa^{(v)}_{-;j}$$ (i.e., There is a one-to-one correspondence between the non-zero different values of $\kappa^{(v)}_{-;j}$ and $maj\left(\tilde{\pi}^{(w)}\right)$);
\item If $\jmath\neq 1$, or $\jmath\neq m^{(w)}$ for some $1\leq w\leq z(\tilde{\pi})$, then $\tilde{h}^{(w)}_{\jmath}=\kappa^{(v)}_{+;j}$ for some  \\ $1\leq v\leq z(\pi)$, and $1\leq j\leq m^{(v)}$. (i.e., If $\tilde{h}^{(w)}_{\jmath}=\kappa^{(v)}_{-;j}$ for some $1\leq v\leq z(\pi)$, and $1\leq j\leq m^{(v)}$, then necessarily $\jmath=1$ or $\jmath=m^{(w)}$);
\item If  $\tilde{h}^{(w)}_{m^{(w)}}=\kappa^{(v)}_{-;j}$ for some $1\leq v\leq z(\pi)$ and $1\leq j\leq m^{(v)}$, then
$$maj\left(\tilde{\pi}^{(w)}_{m^{(w)}}\right)=\tilde{\imath}_{m^{(w)}}^{(w)}=\sum_{v'=\ddot{v}+1}^{z(\pi)}\vartheta^{(v')}_{j'_{v':v}},$$
where, $\kappa^{(v')}_{-;j'_{v':v}}\leq\kappa^{(v)}_{-;j}<\kappa^{(v')}_{+;j'_{v':v}}$ and $\ddot{v}$ is the maximal value of $1\leq v\leq z(\pi)$ such that $\kappa^{(\ddot{v})}_{-;\ddot{j}}=\kappa^{(v)}_{-;j}$  for some $1\leq \ddot{j}\leq m^{(\ddot{v})}$;
\item If  $\tilde{h}^{(w)}_{\jmath}=\kappa^{(v)}_{+;j}$ for some $1\leq v\leq z(\pi)$ and $1\leq j\leq m^{(v)}$, then
$$maj\left(\tilde{\pi}^{(w)}_{\jmath}\right)=\sum_{v'=\chi^{(\dot{v})}_{\dot{j}}+1}^{z(\pi)}\vartheta^{(v')}_{j'_{v':v}},$$
where, $\kappa^{(v')}_{-;j'_{v':v}}<\kappa^{(v)}_{+;j}\leq\kappa^{(v')}_{+;j'_{v':v}}$ and $\dot{v}$ is the minimal value of $1\leq v\leq z(\pi)$ such that $\kappa^{(\dot{v})}_{+;\dot{j}}=\kappa^{(v)}_{+;j}$  for some $1\leq \dot{j}\leq m^{(\dot{v})}$;
\item If $\tilde{h}^{(w)}_{m^{(w)}}=\kappa^{(v)}_{+;j}$ for some $1\leq v\leq z(\pi)$, and $1\leq j\leq m^{(v)}$ then one of the following holds:
    \begin{itemize}
     \item $v=z(\pi)$;
     \item For every $v'>v$, ~$\kappa^{(v')}_{-;1}=\vartheta^{(v')}_{1}=0$, and $\kappa^{(v)}_{-;j+1}<\kappa^{(v')}_{+;1}$;
     \end{itemize}
\item If $\tilde{h}^{(w)}_{1}=\kappa^{(v)}_{-;j}$ for some $1\leq v\leq z(\pi)$, and $1\leq j\leq m^{(v)}$, then ${\imath}_{1}^{(w)}=\varrho^{(\dot{v})}_{\dot{j}-1}-\eta^{(\dot{v})}_{\dot{j}-1}$, where $\dot{v}$ is the minimal value of $1\leq v\leq z(\pi)$ such that $\kappa^{(\dot{v})}_{-;\dot{j}}=\kappa^{(v)}_{-;j}$  for some $2\leq \dot{j}\leq m^{(\dot{v})}$;
\item If $\tilde{h}^{(w)}_{1}=\kappa^{(v)}_{+;j}$ for some $1\leq v\leq z(\pi)$, and $1\leq j\leq m^{(v)}$ then $$maj\left(\tilde{\pi}^{(w)}\right)=\kappa^{(
    \bar{v})}_{-;1}<\tilde{h}^{(w)}_{1}=\kappa^{(v)}_{+;j},$$
    for some $1\leq \bar{v}\leq z(\pi)$;
\item If ~$maj\left(\tilde{\pi}^{(w)}\right)=\kappa^{(v)}_{-;1}$, for some $1\leq v\leq z(\pi)$, then $$maj\left(\tilde{\pi}^{(w)}\right)<\tilde{h}^{(w)}_{1};$$
\item $\tilde{h}^{(w-1)}_{m^{(w-1)}}=\kappa^{(v)}_{+;j}$ and $\tilde{h}^{(w)}_{1}=\kappa^{(\breve{v})}_{+;\breve{j}}$ for some $1\leq v, \breve{v}\leq z(\pi)$, ~$1\leq j\leq m^{(v)}$, and  $1\leq \breve{j}\leq m^{(\breve{v})}$ if and only if
     $$\tilde{h}^{(w-1)}_{m^{(w-1)}}<maj\left(\tilde{\pi}^{(w)}\right)=\kappa^{(\dot{v})}_{-;1}=\kappa^{(v)}_{-;j+1}<\tilde{h}^{(w)}_{1},$$
    for some $1\leq \dot{v}<v\leq z(\pi)$, and $v$ satisfied one of the following conditions:
    \begin{itemize}
     \item $v=z(\pi)$;
     \item For every $v'>v$, ~$\kappa^{(v')}_{-;1}=\vartheta^{(v')}_{1}=0$, and $\kappa^{(v)}_{-;j+1}<\kappa^{(v')}_{+;1}$;
     \end{itemize}
    \end{itemize}

\end{proposition}

\begin{proof}
Consider the standard $OGS$ elementary factorization of $\tilde{\pi}$, with the elementary factors
\begingroup
\large
$$\tilde{\pi}^{(w)}=\prod_{\jmath=1}^{m^{(w)}}t_{\tilde{h}^{(w)}_{\jmath}}^{\tilde{\imath}_{\jmath}^{(w)}};$$
\endgroup
First, look at $\tilde{\pi}^{(w)}$ for $w=z(\tilde{\pi}^{(w)})$ (i.e., the last elementary factor of $\tilde{\pi}^{(w)}$). Then, by Proposition \ref{kappa-adjacent},
$$\tilde{h}^{(w)}_{m^{(w)}}=\kappa^{(z(\pi))}_{+;m^{(z(\pi))}}.$$

Let $\beta$ be the largest $1\leq x\leq r$ such that $a_{x}=\kappa^{(v_{x})}_{-;j_{x}}$ for some $1\leq v_{x}\leq z(\pi)$, and $1\leq j_{x}\leq m^{(v_{x})}$. Then, for every $x>\beta$, ~$a_{x}=\kappa^{(v_{x})}_{+;j_{x}}$. Therefore, by Proposition \ref{order-kappa}, $j_{\beta}=m^{(v_{\beta})}$ and  $j_{v_{x}}=m^{(v_{x})}$, for all $x>\beta$. Moreover, since $a_{x}=\kappa^{(v_{x})}_{+;j_{x}}$, for $x>\beta$, by Proposition \ref{kappa-adjacent}, we conclude, $v_{x_{1}}<v_{x_{2}}$ for every $\beta< x_{1}<x_{2}\leq r$. Thus, by considering
\begingroup
\large
$$\prod_{x=\beta}^{r}t_{a_{x}}^{i_{x}}$$
\endgroup
The following holds:
\begin{itemize}
\item
\begingroup
\large
$$a_{\beta}=\kappa^{(v_{\beta})}_{-;m^{(v_{\beta})}}, ~~i_{\beta}=\kappa^{(v_{\beta})}_{-;m^{(v_{\beta})}}-\vartheta^{(v_{\beta})}_{m^{(v_{\beta})}}-\eta^{(v_{\beta})}_{m^{(v_{\beta})}-1},$$
\endgroup
where,
\\
\begin{itemize}
\item By Definition \ref{kappa}: ~~~~
\begingroup
\large
$\kappa^{(v_{\beta})}_{-;m^{(v_{\beta})}}= \varrho^{(v_{\beta})}_{m^{(v_{\beta})}-1}+\sum _{v'=v_{\beta}+1}^{z(\pi)}\vartheta^{(v')}_{m^{(v')}},$
\endgroup
~such that: \\ \\
\begingroup
\large
$\kappa^{(v')}_{-;m^{(v')}}\leq\kappa^{(v_{\beta})}_{-;m^{(v_{\beta})}}<\kappa^{(v')}_{+;m^{(v')}};$
\endgroup
\\

\item By Definition \ref{eta}: ~~~~
\begingroup
\large
$\eta^{(v_{\beta})}_{m^{(v_{\beta})}-1}=\sum _{v'=\chi^{(v_{\beta})}_{m^{(v_{\beta})}-1}+1}^{v_{\beta}-1}\vartheta^{(v')}_{m^{(v')}},$
\endgroup
~such that: \\ \\
\begingroup
\large
$\kappa^{(v')}_{-;m^{(v')}}<\kappa^{(v_{\beta})}_{-;m^{(v_{\beta})}}<\kappa^{(v')}_{+;m^{(v')}}.$
\endgroup
\end{itemize}

\item For $\beta+1\leq x\leq r$ :
\begingroup
\large
$$a_{x}=\kappa^{(v_{x})}_{+;m^{(v_{x})}}, ~~i_{x}=\vartheta^{(v_{x})}_{m^{(v_{x})}},$$
\endgroup
such that by Definition \ref{chi}, and by Propositions \ref{v1-x-v2}, \ref{kappa-adjacent}, the following properties hold:
\begin{itemize}
\item $v_{\beta+1}=\chi^{(v_{\beta})}_{m^{(v_{\beta})}-1}+1$;
\item $v_{x+1}=v_{x}+1$, for $\beta<x<r$;
\item $v_{r}=z(\pi)$.
\end{itemize}
\end{itemize}
Thus, by using Definition \ref{kappa}, for every $x$ such that $\beta+1\leq x\leq r$ the following are satisfied:
\begingroup
\large
$$a_{x}=\kappa^{(v_{x})}_{+;m^{(v_{x})}}=\varrho^{(v_{x})}_{m^{(v_{x})}}+\sum_{v'=v_{x}}^{z(\pi)}\vartheta^{(v')}_{m^{(v')}}>
\sum_{v'=v_{x}}^{z(\pi)}\vartheta^{(v')}_{m^{(v')}}=\sum_{x'=x}^{r}i_{x'}.$$
\endgroup
Hence, by Definition \ref{canonical-factorization-def},
\begingroup
\large
$$\prod_{x=\beta+1}^{r}t_{a_{x}}^{i_{x}}$$
\endgroup
is a terminal segment of $\tilde{\pi}^{(w)}$, for $w=z(\tilde{\pi})$. \\

Since
\begingroup
\large
~$\eta^{(v_{\beta})}_{m^{(v_{\beta})}-1}=\sum_{v'=\chi^{(v_{\beta})}_{m^{(v_{\beta})}-1}+1}^{v_{\beta}-1}~\vartheta^{(v')}_{m^{(v')}}$,
\endgroup
by using Proposition \ref{eta-varrho}:

\begingroup
\large
\begin{align*}
a_{\beta}&=\kappa^{(v_{\beta})}_{-;m^{(v_{\beta})}}=\varrho^{(v_{\beta})}_{m^{(v_{\beta})}-1}+\sum_{v'=v_{\beta}}^{z(\pi)}\vartheta^{(v')}_{m^{(v')}} \\ \\&>
\sum_{v'=\chi^{(v_{\beta})}_{m^{(v_{\beta})}-1}+1}^{v_{\beta}-1}\vartheta^{(v')}_{m^{(v'_{\beta})}}+\sum_{v'=v_{\beta}}^{z(\pi)}\vartheta^{(v')}_{m^{(v')}} \\ \\ &= \sum_{x=\beta+1}^{r}i_{x}.
\end{align*}
\endgroup
We have also
\begingroup
\large
\begin{align*}
\sum_{x=\beta}^{r}i_{x}&=i_{\beta}+\sum_{x=\beta+1}^{r}i_{x} \\ &=
\kappa^{(v_{\beta})}_{-;m^{(v_{\beta})}}-\sum_{v'=\chi^{(v_{\beta})}_{m^{(v_{\beta})}-1}+1}^{v_{\beta}}\vartheta^{(v')}_{m^{(v')}}+
\sum_{v'=\chi^{(v_{\beta})}_{m^{(v_{\beta})}-1}+1}^{z(\pi)}\vartheta^{(v')}_{m^{(v')}} \\ \\ &=\kappa^{(v_{\beta})}_{-;m^{(v_{\beta})}}+\sum_{v'=v_{\beta}+1}^{z(\pi)}\vartheta^{(v')}_{m^{(v')}} \\ \\ &\geq \kappa^{(v_{\beta})}_{-;m^{(v_{\beta})}} ~= ~a_{\beta}.
\end{align*}
\endgroup

Therefore, by considering
\begingroup
\large
$$\tilde{\pi}^{(w)}=\prod_{\jmath=1}^{m^{(w)}}t_{\tilde{h}^{(w)}_{\jmath}}^{\tilde{\imath}_{\jmath}^{(w)}},$$
\endgroup
for $w=z(\tilde{\pi})$, the following holds:
\begin{itemize}
\item In case $m^{(v_{\beta})}>1$ (i.e., $\varrho^{(v_{\beta})}_{m^{(v_{\beta})}-1}>0$):
\begingroup
\large
$$\tilde{h}^{(w)}_{1}=\kappa^{(v_{\beta})}_{-;m^{(v_{\beta})}}, ~~~~\tilde{\imath}_{1}^{(w)}=\varrho^{(v_{\beta})}_{m^{(v_{\beta})}-1}-\eta^{(v_{\beta})}_{m^{(v_{\beta})}-1}.$$
\endgroup

If $2\leq\jmath\leq m^{z(\tilde{\pi})}$, then
\begingroup
\large
$$\tilde{h}^{(w)}_{\jmath}=\kappa^{(v_{x})}_{+;m^{(v_{x})}},$$
\endgroup
for some $x>\beta$.
Moreover, by Definition \ref{chi}  and Proposition \ref{v1-x-v2}, for every $\chi^{(v)}_{m^{(v)}-1}<v\leq z(\pi)$, there exists $2\leq\jmath\leq m^{z(\tilde{\pi})}$, such that
\begingroup
\large
$$\tilde{h}^{(w)}_{\jmath}=\kappa^{(v)}_{+;m^{(v)}}.$$
\endgroup
Since there is a possibility for $\kappa^{(v)}_{+;m^{(v)}}=\kappa^{(v')}_{+;m^{(v)}}$, where $v'\neq v$,
\begingroup
\large
$$\tilde{\imath}_{\jmath}^{(w)}=\sum_{x=\dot{x}}^{\ddot{x}}\vartheta^{(v_{x})}_{m^{(v_{x})}},$$
\endgroup
such that $\kappa^{(v_{x})}_{+;m^{(v_{x})}}=\kappa^{(v)}_{+;m^{(v)}}$ for every $x$ where, $\dot{x}\leq x\leq \ddot{x}$.
Therefore, for every $2\leq \jmath\leq m^{(w)}$,
\begingroup
\large
$$maj\left(\tilde{\pi}_{\jmath}\right)=\sum_{x=\dot{x}}^{z(\pi)}\vartheta^{(v_{x})}_{m^{(v_{x})}},$$
\endgroup
such that $\dot{x}$ is the minimal $x>\beta$, such that
\begingroup
\large
$$\tilde{h}^{(w)}_{\jmath}=\kappa^{(v_{\dot{x}})}_{+;m^{(v_{\dot{x}})}}.$$
\endgroup
By former observation:
\begingroup
\large
$$maj\left(\tilde{\pi}^{(w)}\right)=\sum_{\jmath=1}^{m^{(w)}}\tilde{\imath}_{\jmath}^{(w)}=\kappa^{(v_{\beta})}_{-;m^{(v_{\beta})}};$$
\endgroup
\item In case, where $m^{(v_{\beta})}=1$, we have,
\begingroup
\large
$\varrho^{(v_{\beta})}_{m^{(v_{\beta})}-1}=0$.
\endgroup
 Thus, by Propositions \ref{eta-varrho}, \ref{kappa-adjacent}
 \begingroup
 \large
 ~$\eta^{(v_{\beta})}_{m^{(v_{\beta})}-1}=0$, ~~$\chi^{(v_{\beta})}_{m^{(v_{\beta})}-1}=v_{\beta}-1$.
 \endgroup
 Thus,
    \begingroup
    \large
    $${\tilde{h}}^{(w)}_{1}\neq \kappa^{(v_{\beta})}_{-;m^{(v_{\beta})}},$$
    \endgroup
    and therefore, we have for every $1\leq \jmath\leq m^{(w)}$:
   \begingroup
   \large
    $$\tilde{h}^{(w)}_{\jmath}=\kappa^{(v_{x})}_{+;m^{(v_{x})}}$$
\endgroup
for some $x>\beta$, where the other observations are the same as in the case of $m^{(v_{\beta})}>1$.

In particular,
\begingroup
\large
$$maj\left(\tilde{\pi}^{(w)}\right)=\sum_{\jmath=1}^{m^{(w)}}\tilde{\imath}_{\jmath}^{(w)}=\kappa^{(v_{\beta})}_{-;m^{(v_{\beta})}};$$
\endgroup
even in case where,
 \begingroup
 \large
  ~${\tilde{h}}^{(w)}_{1}\neq \kappa^{(v_{\beta})}_{-;m^{(v_{\beta})}}.$
\endgroup
\end{itemize}
Now, look at
\begingroup
\large
${\tilde{h}}^{(w-1)}_{m^{(w-1)}}$,
\endgroup
for $w=z(\pi)$. Let $\zeta$ be the number of different values of $x$, such that $a_{x}=a_{\beta}$, i.e.,
\begingroup
\large
$$a_{x}=\kappa^{(v_{x})}_{-;m^{(v_{x})}}=\kappa^{(v_{\beta})}_{-;m^{(v_{\beta})}}=a_{\beta},$$
\endgroup
for $\beta-\zeta+1\leq x\leq \beta$. Notice, by the assumption of the proposition, $v_{x_{1}}>v_{x_{2}}$ for $\beta-\zeta+1\leq x_{1}<x_{2}\leq \beta$. Thus, the maximal value of $v_{x}$ such that $\kappa^{(v_{x})}_{-;m^{(v_{x})}}=\kappa^{(v_{\beta})}_{-;m^{(v_{\beta})}}$ is for $x=\beta-\zeta+1$. Notice,
\begingroup
\large
$\tilde{\imath}_{1}^{(w)}=\varrho^{(v_{\beta})}_{m^{(v_{\beta})}-1}-\eta^{(v_{\beta})}_{m^{(v_{\beta})}-1},$
\endgroup
in case
\begingroup
\large
$\tilde{h}^{(w)}_{1}=\kappa^{(v_{\beta})}_{-;m^{(v_{\beta})}}$.
\endgroup
Therefore,
\begingroup
\large
$$\sum_{x=\beta-\zeta+1}^{\beta}i_{x}-\left(\varrho^{(v_{\beta})}_{m^{(v_{\beta})}-1}-\eta^{(v_{\beta})}_{m^{(v_{\beta})}-1}\right),$$
\endgroup
the possibly exponent of
\begingroup
\large
${\tilde{h}}^{(w-1)}_{m^{(w-1)}}$,
\endgroup
in case
\begingroup
\large
${\tilde{h}}^{(w-1)}_{m^{(w-1)}}=\kappa^{(v_{\beta})}_{-;m^{(v_{\beta})}}.$
\endgroup
Then, by using the definitions of $i_{x}$, ~$\kappa^{(v_{x})}_{-;m^{(v_{x})}}$, and  $\eta^{(v_{x})}_{m^{(v_{x})}-1}$ ~for $\beta-\zeta+1\leq x\leq \beta$, we conclude the following property:
\begingroup
\large
\begin{align*}
&\sum_{x=\beta-\zeta+1}^{\beta}i_{x}-\left(\varrho^{(v_{\beta})}_{m^{(v_{\beta})}-1}-\eta^{(v_{\beta})}_{m^{(v_{\beta})}-1}\right) \\ &=\sum_{x=\beta-\zeta+1}^{\beta-1}
\left(\kappa^{(v_{x})}_{-;m^{(v_{x})}}-\vartheta^{(v_{x})}_{m^{(v_{x})}}-\eta^{(v_{x})}_{m^{(v_{x})}-1}\right)+ \\ &+
\sum_{v'=v_{\beta}+1}^{w}\vartheta^{(v')}_{m^{(v')}}-
\left(\varrho^{(v_{\beta})}_{m^{(v_{\beta})}-1}-\eta^{(v_{\beta})}_{m^{(v_{\beta})}-1}\right) \\ &= (\zeta-1)\cdot \kappa^{(v_{\beta})}_{-;m^{(v_{\beta})}}+
\sum_{v'=v_{\beta-\zeta+1}+1}^{w}\vartheta^{(v')}_{m^{(v')}}.
\end{align*}
\endgroup
Therefore, the following holds:
\begin{itemize}
\item In case, where there exists at least one $v'>v_{\beta-\zeta+1}$, such that $\vartheta^{(v')}_{m^{(v')}}>0$,
\begingroup
\large
$${\tilde{h}}^{(w-1)}_{m^{(w-1)}}=\kappa^{(v_{\beta})}_{-;m^{(v_{\beta})}},  ~~~~\tilde{\imath}_{m^{(w-1)}}^{(w-1)}=\sum_{v'=v_{\beta-\zeta+1}+1}^{w}\vartheta^{(v')}_{m^{(v')}};$$
\endgroup
\item In case, where $\vartheta^{(v')}_{m^{(v')}}=0$, for every $v'>v_{\beta-\zeta+1}$, which means necessarily \\ $v_{\beta-\zeta+1}=z(\pi)$ (The case where for every $v'>v_{\beta-\zeta+1}$, we have ~$\kappa^{(v')}_{-;1}=\vartheta^{(v')}_{1}=0$, ~$\kappa^{(v_{\beta})}_{-;m^{(v_{\beta})}}<\kappa^{(v')}_{+;1}$~ and ~$m^{(v')}=1$ cannot happen by Proposition \ref{kappa-zero}.):
    \begingroup
    \large
    $$\sum_{v'=v_{\beta-\zeta+1}+1}^{z(\pi)}\vartheta^{(v')}_{m^{(v')}}=0.$$
    \endgroup
     Therefore,
    \begingroup
    \large
    $${\tilde{h}}^{(w-1)}_{m^{(w-1)}}\neq \kappa^{(v_{\beta})}_{-;m^{(v_{\beta})}}.$$
    \endgroup
\end{itemize}

Notice the following observation:
\\
In case, where both $m^{(v_{\beta})}=1$ and $v_{\beta-\zeta+1}=z(\pi)$:
\begingroup
\large
$$a_{\beta}=\kappa^{(v_{\beta})}_{-;1}=\vartheta^{(v_{\beta})}_{1}, ~~~~i_{\beta}=0.$$
\endgroup
Therefore, there is no $1\leq w\leq z(\tilde{\pi})$, and $1\leq \jmath\leq m^{(w)}$, such that
\begingroup
\large
$${\tilde{h}}^{(w)}_{\jmath}=\kappa^{(v_{\beta})}_{-;1},$$
\endgroup
Although,
\begingroup
\large
$$maj\left(\tilde{\pi}^{(z(\tilde{\pi}))}\right)=\kappa^{(v_{\beta})}_{-;1}.$$
\endgroup
i.e.,
\begingroup
\large
$${\tilde{h}}^{(z(\tilde{\pi})-1)}_{m^{(z(\tilde{\pi})-1)}}<\kappa^{(v_{\beta})}_{-;m^{(v_{\beta})}}<{\tilde{h}}^{(z(\tilde{\pi}))}_{1}.$$
\endgroup
Thus, all the parts of the proposition hold for the terminal elementary factor $\tilde{\pi}^{(z(\tilde{\pi}))}$ of $\tilde{\pi}$. Now, assume in induction, the proposition holds for  $\tilde{\pi}^{(w')}$, where $w'>w$ for some $1\leq w<z(\tilde{\pi})$, and we prove it for $w'=w$.
Consider
\begingroup
\large
$$\tilde{\pi}^{(w)}=\prod_{\jmath=1}^{m^{(w)}}t_{\tilde{h}^{(w)}_{\jmath}}^{\tilde{\imath}_{\jmath}^{(w)}}.$$
\endgroup
By the induction hypothesis,
\begingroup
\large
$$a_{\alpha}=maj\left(\tilde{\pi}^{(w+1)}\right)=\kappa^{(v_{\alpha})}_{-;j_{\alpha}}, $$
\endgroup
for some $1\leq v_{\alpha}\leq z(\pi)$ and $1\leq j_{\alpha}\leq m^{(v_{\alpha})}$
~($\beta$ is just $\alpha$ for $w'=z(\tilde{\pi})$).
For every $v'>\chi^{(v_{\alpha})}_{-;j_{\alpha}-1}$, let $j_{v':\alpha}$ be a positive integer such that
\begingroup
\large
$$\kappa^{(v')}_{-;j_{v':\alpha}}\leq \kappa^{(v_{\alpha})}_{-;j_{\alpha}}<\kappa^{(v')}_{+;j_{v':\alpha}}.$$
\endgroup
Let $\ddot{v}_{\alpha}$ ~be the maximal value of $v_{\alpha}\leq v\leq z(\pi)$, such that
\begingroup
\large
$\kappa^{(\ddot{v}_{\alpha})}_{-;\ddot{j}_{\alpha}}=\kappa^{(v_{\alpha})}_{-;j_{\alpha}}$
\endgroup
for some \\ $1\leq \ddot{j}_{\alpha}\leq m^{(\ddot{v}_{\alpha})}$.
Then, by the induction hypothesis the following holds:
\begin{itemize}
\item If there exists at least one $v'>\ddot{v}_{\alpha}$, such that $\vartheta^{(v')}_{j_{v':\alpha}}>0$,
\begingroup
\large
$${\tilde{h}}^{(w)}_{m^{(w)}}=\kappa^{(v_{\alpha})}_{j_{\alpha}}, ~~~~\tilde{\imath}_{m^{(w)}}^{(w)}=\sum_{\ddot{v}_{\alpha}+1}^{z(\pi)}\vartheta^{(v')}_{j_{v:\alpha}};$$
\endgroup
\item If $\vartheta^{(v')}_{j_{v',\alpha}}=0$, for every $v'>\ddot{v}_{\alpha}$ (i.e., either $\ddot{v}_{\alpha}=z(\pi)$ ~or for every $v'>\ddot{v}_{\alpha}$ ~we have $~\vartheta^{(v')}_{1}=0$ and $\kappa^{(v_{\alpha})}_{-;j_{\alpha}}<\kappa^{(v')}_{+;1}$~), then:
    \begingroup
    \large
    $$\sum_{v'=\ddot{v}_{\alpha}+1}^{z(\pi)}\vartheta^{(v')}_{j_{v',\alpha}}=0.$$
    \endgroup
     Therefore,
    \begingroup
    \large
    $$\tilde{h}^{(w)}_{m^{(w)}}\neq\kappa^{(v_{\alpha})}_{-;j_{\alpha}}.$$
    \endgroup
Since by the induction hypothesis,
\begingroup
\large
$$a_{\alpha}=maj\left(\tilde{\pi}^{(w+1)}\right)=\kappa^{(v_{\alpha})}_{-;j_{\alpha}},$$
\endgroup
by Proposition \ref{kappa-adjacent}, either
\begingroup
\large
$$\tilde{h}^{(w)}_{m^{(w)}}=\kappa^{(\bar{v})}_{-;\bar{j}}$$
\endgroup
such that $\bar{v}>v_{\alpha}$,
or
\begingroup
\large
$$\tilde{h}^{(w)}_{m^{(w)}}=\kappa^{(\ddot{v}_{\alpha})}_{+;\ddot{j}_{\alpha}-1}.$$
\endgroup
\end{itemize}

Let $\alpha'$ be the largest integer such that $\alpha'<\alpha$, and $a_{\alpha'}=\kappa^{(v_{\alpha'})}_{-;j_{\alpha'}}<a_{\alpha}$, for some $1\leq v_{\alpha'}\leq z(\pi)$ and $1\leq j_{\alpha'}\leq m^{(v_{\alpha'})}$. Therefore, $a_{x}=\kappa^{(v_{x})}_{+;j_{x}}$ for every $x$ such that $x>\alpha'$ and $a_{x}<\kappa^{(v_{\alpha})}_{-;j_{\alpha}}$. If $x_{1}<x_{2}$ and $\kappa^{(v_{\alpha'})}_{-;j_{\alpha'}}<a_{x_{1}}\leq a_{x_{2}}<\kappa^{(v_{\alpha})}_{-;j_{\alpha}}$, then by Proposition \ref{kappa-adjacent}, $v_{x_{1}}<v_{x_{2}}$. Notice, the maximal possible value of $v_{x}$ such that $a_{x}=\kappa^{(v_{x})}_{+;j_{x}}$  is $v_{x}=\ddot{v}_{\alpha}$ and then $a_{x}=\kappa^{(\ddot{v}_{\alpha})}_{+;\ddot{j}_{{\alpha}-1}}$, where $a_{x+1}=\kappa^{(\ddot{v}_{\alpha})}_{-;\ddot{j}_{\alpha}}=\kappa^{(v_{\alpha})}_{-;j_{\alpha}}$. Therefore, if
\begingroup
\large
$$a_{x}=\kappa^{(v_{x})}_{+;j_{x}}=\varrho^{(v_{x})}_{j_{x}}+\sum_{v'>v_{x+1}}^{z(\pi)}\vartheta^{(v')}_{j_{v':x}},$$
\endgroup
where, $j_{v':x}$ is a positive integer such that
\begingroup
\large
$$\kappa^{(v')}_{-;j_{v':x}}<\kappa^{(v_{x})}_{+;j_{x}}\leq\kappa^{(v')}_{+;j_{v':x}},$$
\endgroup
then, by Proposition \ref{kappa-adjacent},
\begingroup
\large
$\chi^{(v_{x})}_{j_{x}}=v_{x-1}$,
\endgroup
in case $x>\alpha'+1$. Therefore,
\begingroup
\large
$$i_{x}=\vartheta^{(v_{x})}_{j_{x}}+\eta^{(v_{x})}_{j_{x}}=\sum_{v'=v_{x-1}+1}^{v_{x}-1}\vartheta^{(v')}_{j_{v':x}}+\vartheta^{(v_{x})}_{j_{x}},$$
\endgroup
and for $x=\alpha'+1$:
\begingroup
\large
$$i_{x}=\vartheta^{(v_{x})}_{j_{x}}+\eta^{(v_{x})}_{j_{x}}=\sum_{v'=\chi^{(v_{x})}_{j_{x}}+1}^{v_{x}-1}\vartheta^{(v')}_{j_{v':x}}+\vartheta^{(v_{x})}_{j_{x}}.$$
\endgroup
Therefore, by using Proposition \ref{eta-varrho}
\begingroup
\large
\begin{align*}
a_{x}&=\kappa^{(v_{x})}_{+;j_{x}}=\varrho^{(v_{x})}_{j_{x}}+\vartheta^{(v_{x})}_{j_{x}}+\sum_{v'=v_{x}+1}^{z(\pi)}\vartheta^{(v')}_{j_{v':x}} \\ &>
\eta^{(v_{x})}_{j_{x}}+\vartheta^{(v_{x})}_{j_{x}}+\sum_{v'=v_{x}+1}^{\ddot{v}_{\alpha}}\vartheta^{(v')}_{j_{v':x}}+
\sum_{v'=\ddot{v}_{\alpha}+1}^{z(\pi)}\vartheta^{(v')}_{j_{v':x}} \\ &=\sum_{x'=x}^{\alpha-1}i_{x'}=maj\left(\tilde{\pi}^{(w)}_{x}\right).
\end{align*}
\endgroup
Hence, by Definition \ref{canonical-factorization-def},
\begingroup
\large
$$\prod_{x=\alpha'+1}^{\alpha-1}t_{a_{x}}^{i_{x}}$$
\endgroup
is a terminal segment of $\tilde{\pi}^{(w)}$.\\
Notice, $x=\alpha'$ is the largest $x$, such that $a_{\alpha'}=\kappa^{(v_{\alpha'})}_{-;j_{\alpha'}}<\kappa^{(v_{\alpha})}_{-;j_{\alpha}}$. Therefore, $j_{v':x}=j_{v':\alpha'}$, for every $\chi^{(v_{x})}_{j_{x}}<v'<v_{x}$, and $j_{v'}=j_{v':\alpha'}$ for $v'=v_{x}$, in case $a_{x}=\kappa^{(v_{x})}_{+;j_{x}}$ for $x>\alpha'$ and $a_{x}<\kappa^{(v_{\alpha})}_{-;j_{\alpha}}$. By Proposition \ref{kappa-adjacent}, $\chi^{(v_{\alpha'})}_{j_{\alpha'}-1}=\chi^{(v_{\alpha'+1})}_{j_{\alpha'+1}}$, in case $a_{\alpha'+1}=\kappa^{(v_{\alpha'+1})}_{+;j_{\alpha'+1}}$.
Thus,
\begingroup
\large
\begin{align*}
a_{\alpha'}&=\kappa^{(v_{\alpha'})}_{-;j_{\alpha'}} =\varrho^{(v_{\alpha'})}_{j_{\alpha'}-1}+\vartheta^{(v_{\alpha'})}_{j_{\alpha'}}+\sum_{v'=v_{\alpha'+1}}^{z(\pi)}\vartheta^{(v')}_{j_{v':\alpha'}}\\&>
\eta^{(v_{\alpha'})}_{j_{\alpha'}-1}+\vartheta^{(v_{\alpha'})}_{j_{\alpha'}}+\sum_{v'=v_{\alpha'}+1}^{z(\pi)}\vartheta^{(v')}_{j_{v':\alpha'}}\\ \\ &=
\sum_{v'=\chi^{(v_{\alpha'})}_{j_{\alpha'}-1}+1}^{v_{\alpha'}-1}\vartheta^{(v')}_{j_{v':\alpha'}}+\vartheta^{(v_{\alpha'})}_{j_{\alpha'}}+
\sum_{v'=v_{\alpha'}+1}^{z(\pi)}\vartheta^{(v')}_{j_{v':\alpha'}} \\ \\ &= \sum_{x=\alpha'+1}^{\alpha-1}i_{x}=maj\left(\tilde{\pi}^{(w)}_{\alpha'+1}\right).
\end{align*}
\endgroup
We have also
\begingroup
\large
\begin{align*}
\sum_{x=\alpha'}^{\alpha-1}i_{x}&=i_{\alpha'}+\sum_{x=\alpha'+1}^{\alpha-1}i_{x}=
\kappa^{(v_{\alpha'})}_{-;j_{\alpha'}}-\vartheta^{(v_{\alpha'})}_{j_{\alpha'}} -\sum_{v'=\chi^{(v_{\alpha'})}_{j_{\alpha'}-1}+1}^{v_{\alpha'}-1}\vartheta^{(v')}_{j_{v':\alpha'}}+
\sum_{v'=\chi^{(v_{\alpha'})}_{j_{\alpha'}-1}+1}^{z(\pi)}\vartheta^{(v')}_{j_{v':\alpha'}}
\\ \\ &=\kappa^{(v_{\alpha'})}_{-;j_{\alpha'}}+\sum_{v'=v_{\alpha'}+1}^{z(\pi)}\vartheta^{(v')}_{j_{v':\alpha'}}\geq \kappa^{(v_{\alpha'})}_{-;j_{\alpha'}} ~= ~a_{\alpha'}.
\end{align*}
\endgroup
Hence, by the same argument as in the case of $\tilde{\pi}^{(z(\tilde{\pi}))}$ the following holds for
\begingroup
\large
$$\tilde{\pi}^{(w)}=\prod_{\jmath=1}^{m^{(w)}}t_{\tilde{h}^{(w)}_{\jmath}}^{\tilde{\imath}_{\jmath}^{(w)}}:$$
\endgroup
\begin{itemize}
\item In case $j_{\alpha'}>1$ (i.e., $\varrho^{(v_{\alpha'})}_{j_{\alpha'}}>0$):
\begingroup
\large
$$\tilde{h}^{(w)}_{1}=\kappa^{(v_{\alpha'})}_{-;j_{\alpha'}}, ~~~~\tilde{\imath}_{1}^{(w)}=\varrho^{(v_{\alpha'})}_{j_{\alpha'}-1}-\eta^{(v_{\alpha'})}_{j_{\alpha'}-1}.$$
\endgroup
If $\jmath\neq 1$ or $\jmath\neq m^{(w)}$ (i.e., $\tilde{h}^{(w)}_{\jmath}$ is not the first and is not the last letter of $\tilde{\pi}^{(w)}$)
then,
\begingroup
\large
$$\tilde{h}^{(w)}_{\jmath}=\kappa^{(v_{x})}_{+;j_{x}},$$
\endgroup
for some $x>\alpha'$, such that $v_{x}\leq \ddot{v}_{\alpha}$.
Now we consider
\begingroup
\large
$\tilde{\imath}_{\jmath}^{(w)}$.
\endgroup
Let  $\zeta'$ be the number of $x>\alpha'$ such that $a_{x}=\kappa^{(v)}_{+;j}$ , for a specific $1\leq v\leq z(\pi)$ and $1\leq j\leq m^{(v)}$ (i.e., $\kappa^{(v_{x})}_{+;j_{x}}=\kappa^{(v)}_{+;j}$, for every ~$\alpha'<x'\leq x\leq x'+\zeta'-1$). Then,
\begingroup
\large
$$\tilde{\imath}_{\jmath}^{(w)}=\sum_{v'=\chi^{(v_{x'})}_{j_{x'}}+1}^{v_{x'+\zeta'-1}}\vartheta^{(v')}_{j_{v':\alpha'}},$$
\endgroup
where,
\begin{itemize}
\item $\chi^{(v_{x'})}_{j_{x'}}=v_{x'-1}$ for $x'>\alpha'+1$;
\item $\chi^{(v_{x'})}_{j_{x'}}=\chi^{(v_{\alpha'})}_{j_{\alpha'}-1}$, for $x'=\alpha'+1$;
\item $x=x'$ is the minimal value of $v_{x}$ such that $\kappa^{(v_{x'})}_{+;j_{x'}}=\kappa^{(v)}_{+;j}$;
\item $x=x'+\zeta'-1$ is the maximal value of $v_{x}$ such that $\kappa^{(v_{x'+\zeta'-1})}_{+;j_{x'+\zeta'-1}}=\kappa^{(v)}_{+;j}$ for some $1\leq v\leq z(\pi)$ and $1\leq j\leq m^{(z(\pi))}$.
\end{itemize}

Now, we consider $\tilde{h}^{(w)}_{m^{(w)}}$ By the induction hypothesis: In case of
\begingroup
\large
$\tilde{h}^{(w)}_{m^{(w)}}=\kappa^{(v)}_{+;j}$
\endgroup
for some $1\leq v\leq z(\pi)$, and $1\leq j\leq m^{(v)}$, one of the following holds:
    \begin{itemize}
     \item either ~$v=z(\pi)$;
     \item or for every $v'>v$, $\kappa^{(v')}_{-;1}=0$, and $\kappa^{(v)}_{-;j+1}<\kappa^{(v')}_{+;1}$.
     \end{itemize}
Otherwise,
\begingroup
\large
$$\tilde{h}^{(w)}_{m^{(w)}}=\kappa^{(v_{\alpha})}_{-;j_{\alpha}}, ~~~~
\tilde{\imath}_{m^{(w)}}^{(w)}=\sum_{v'=\ddot{v}_{\alpha}+1}^{z(\pi)}\vartheta^{(v')}_{j'_{v'}}.$$
\endgroup
Therefore,
\begingroup
\large
$$maj\left(\tilde{\pi}^{(w)}\right)=\sum_{\jmath=1}^{m^{(w)}}\tilde{\imath}_{\jmath}^{(w)}=\kappa^{(v_{\alpha'})}_{-;m^{(v_{\alpha'})}};$$
\endgroup

\item In case, where $j_{\alpha'}=1$, we have,
\begingroup
\large
$\varrho^{(v_{\alpha'})}_{j_{\alpha'}-1}=0$.
\endgroup
 Thus, by Propositions \ref{eta-varrho}, \ref{kappa-adjacent}
 \begingroup
 \large
 ~$\eta^{(v_{\alpha'})}_{j_{\alpha'}-1}=0$, ~~$\chi^{(v_{\alpha'})}_{j_{\alpha'}-1}=v_{\alpha'}-1$.
 \endgroup
 Therefore,
    \begingroup
    \large
    $${\tilde{h}}^{(w)}_{1}\neq \kappa^{(v_{\alpha'})}_{-;j_{\alpha'}}.$$
\endgroup
The further
\begingroup
\large
${\tilde{h}}^{(w)}_{\jmath}$
\endgroup
is the same like in the case of
\begingroup
\large
$j_{\alpha'}>1$.
\endgroup
Also,
\begingroup
\large
$$maj\left(\tilde{\pi}^{(w)}\right)=\sum_{\jmath=1}^{m^{(w)}}\tilde{\imath}_{\jmath}^{(w)}=\kappa^{(v_{\alpha'})}_{-;j_{\alpha'}}$$
\endgroup
even in case where,
 \begingroup
 \large
  ~${\tilde{h}}^{(w)}_{1}\neq \kappa^{(v_{\alpha'})}_{-;j_{\alpha'}}.$
\endgroup
\end{itemize}

Now, look at
\begingroup
\large
${\tilde{h}}^{(w-1)}_{m^{(w-1)}}$.
\endgroup
 Let $\bar{\zeta}$ be the number of different values of $x$, such that $a_{x}=a_{\alpha'}$, i.e.,
\begingroup
\large
$$a_{x}=\kappa^{(v_{x})}_{-;j_{x}}=\kappa^{(v_{\alpha'})}_{-;j_{\alpha'}}=a_{\alpha'},$$
\endgroup
for $\alpha'-\bar{\zeta}+1\leq x\leq \alpha'$.
Then, by the same argument as in the case of $w=z(\tilde{\pi})$ the following holds:
\begin{itemize}
\item In case, where there exists at least one $v'>v_{\alpha'-\bar{\zeta}+1}$, such that $\vartheta^{(v')}_{j_{v':\alpha'}}>0$,
\begingroup
\large
$${\tilde{h}}^{(w-1)}_{m^{(w-1)}}=\kappa^{(v_{\alpha'})}_{-;j_{\alpha'}},  ~~~~\tilde{\imath}_{m^{(w-1)}}^{(w-1)}=\sum_{v'=v_{\alpha'-\bar{\zeta}+1}+1}^{w}\vartheta^{(v')}_{j_{v':\alpha'}};$$
\endgroup
\item In case, where $\vartheta^{(v')}_{j_{v':\alpha'}}=0$, for every $v'>v_{\alpha'-\bar{\zeta}+1}$, by Proposition \ref{kappa-zero}, either $v_{\alpha'-\bar{\zeta}+1}=z(\pi)$ ~or for every $v'>v_{\alpha'-\bar{\zeta}+1}$, we have ~$\kappa^{(v')}_{-;1}=\vartheta^{(v')}_{1}=0$ and \\ $\kappa^{(v_{\alpha'})}_{-;j_{\alpha'}}<\kappa^{(v')}_{+;1}$. Then, the following holds:
    \begingroup
    \large
    $$\sum_{v'=v_{\alpha'-\bar{\zeta}+1}}^{z(\pi)}\vartheta^{(v')}_{j_{v':\alpha'}}=0.$$
    \endgroup
     Therefore,
    \begingroup
    \large
    $${\tilde{h}}^{(w-1)}_{m^{(w-1)}}\neq \kappa^{(v_{\alpha'})}_{-;j_{\alpha'}}.$$
    \endgroup
\end{itemize}
Thus, the proposition has been proved for $w'=w$. Hence, the proposition holds for $\tilde{\pi}$.
\end{proof}

The next theorem proves the formula for the standard $OGS$ canonical form of $\pi^{-1}$ for an arbitrary $\pi\in S_{n}$, which is presented in the standard $OGS$ canonical form as well.

\begin{theorem}\label{inverse-general-sn}
Let $\pi\in S_n$, presented in standard $OGS$ elementary factorization, with all the notations used in Definition \ref{canonical-factorization-def}. For every $1\leq v\leq z(\pi)$ and $1\leq j\leq m^{(v)}$, let $\vartheta^{(v)}_{j}$, ~$\kappa^{(v)}_{-;j}$, ~$\kappa^{(v)}_{+;j}$, ~and ~$\eta^{(v)}_{j}$ be non-negative integers as defined in Definitions \ref{rho}, \ref{kappa}, and \ref{eta}. Then, the standard $OGS$ canonical form of $\pi^{-1}$ is $\prod_{j=1}^{\tilde{m}}t_{\tilde{k}_j}^{\tilde{i}_{\tilde{k}_j}}$, where $\tilde{k}_{\tilde{j}}$, $\tilde{i}_{\tilde{k}_{\tilde{j}}}$ satisfy the following conditions for every $1\leq \tilde{j}\leq \tilde{m}$:
\begin{itemize}
\item $\tilde{k}_{\tilde{j}}=\kappa^{(v)}_{-;j}$ or $\tilde{k}_{\tilde{j}}=\kappa^{(v)}_{+;j}$, for $1\leq v\leq z(\pi)$, and $1\leq j\leq m^{(v)}$;
\item If $\tilde{k}_{\tilde{j}}=\kappa^{(v)}_{+;j}$, then $\tilde{i}_{\tilde{k}_{\tilde{j}}}=\sum_{\gamma}\left(\vartheta^{(v_{\gamma})}_{j_{\gamma}}+\eta^{(v_{\gamma})}_{j_{\gamma}}\right),$
such that $\kappa^{(v)}_{+;j}=\kappa^{(v_{\gamma})}_{+;j_{\gamma}};$
\item If $\tilde{k}_{\tilde{j}}=\kappa^{(v)}_{-;j}$, then $\tilde{i}_{\tilde{k}_{\tilde{j}}}=\kappa^{(v)}_{-;j}-\sum_{\gamma}\left(\vartheta^{(v_{\gamma})}_{j_{\gamma}}+\eta^{(v_{\gamma})}_{j_{\gamma}-1}\right),$
such that $\kappa^{(v)}_{-;j}=\kappa^{(v_{\gamma})}_{-;j_{\gamma}};$
\item $z(\pi^{-1})$, the number of standard $OGS$ elementary factors of $\pi^{-1}$ equals to the number of different non-zero values $\kappa^{(v)}_{-;j}$, where $1\leq v\leq z(\pi)$, and $1\leq j\leq m^{(v)}$;
\item For every $1\leq w\leq z(\pi^{-1})$, ~$maj\left({\pi^{-1}}^{(w)}\right)=\kappa^{(v)}_{-;j}$ for some $1\leq v\leq z(\pi)$, and for every given non-zero value of $\kappa^{(v)}_{-;j}$, there exists exactly one $w$ such that $maj\left({\pi^{-1}}^{(w)}\right)=\kappa^{(v)}_{-;j}$ (i.e., There is a one-to-one correspondence between the non-zero different values of $\kappa^{(v)}_{-;j}$ and the descent set of ~$\pi^{-1}$).
\end{itemize}
\end{theorem}

\begin{proof}
The proof is in induction on $z(\pi)$. If $z(\pi)=1$, then $\pi$ is a standard $OGS$ elementary element. Then, by Definition \ref{chi}, $\chi^{(1)}_{j}=0$, for every $1\leq j\leq m$. Therefore, by Definition \ref{eta}, $\eta^{(1)}_{j}=0$. By Definition \ref{kappa}, $\kappa^{(1)}_{-;j}-\vartheta^{(1)}_{j}=\varrho^{(1)}_{j-1}$.
Therefore, the result of the theorem for the case of $\pi$ a standard $OGS$ elementary element holds by Theorem \ref{elementary-inverse}. Now, assume in induction that the theorem holds for $z'<z(\pi)$, and we prove the correctness of it for $z'=z(\pi)$. Denote by $\pi'$ the element $\prod_{v=1}^{z(\pi)-1}\pi^{(v)}$, and consider $\pi=\pi'\cdot \pi^{(z(\pi))}$. For a convenience we write $\varrho_{j}$, $\vartheta_{j}$, $\kappa_{+;j}$, $\kappa_{-;j}$, $\chi_{j}$, and $\eta_{j}$ instead of $\varrho^{(z(\pi))}_{j}$, $\vartheta^{(z(\pi))}_{j}$, $\kappa^{(z(\pi))}_{+;j}$, $\kappa^{(z(\pi))}_{-;j}$, $\chi^{(z(\pi))}_{j}$, and $\eta^{(z(\pi))}_{j}$. Obviously,
$$\pi^{-1}=(\prod_{v=1}^{z(\pi)}\pi^{(v)})^{-1}=(\pi^{(z(\pi))})^{-1}\cdot \pi'^{-1}.$$
Thus, by using Theorem \ref{elementary-inverse},
\begingroup
\large
$$\pi^{-1}=\prod_{j=1}^{m^{(z(\pi))}}(t_{{\kappa}_{-;j}}^{\varrho_{j-1}}\cdot t_{{\kappa}_{+;j}}^{\vartheta_{j}})\cdot \pi'^{-1},$$
\endgroup
where, we consider $\varrho_{0}=0$.
By the induction hypothesis, the standard $OGS$ canonical form of ${\pi'}^{-1}$ is as follows:
 \begingroup
 \large
 $$\pi'^{-1}=\prod_{d=1}^{m'}t_{k_{d}'}^{i_{k_{d}'}},$$
 \endgroup
 where
 \begingroup
 \large
 $$k_{d}'\in \{\kappa^{(v\rightarrow z(\pi)-1)}_{-;j}, \kappa^{(v\rightarrow z(\pi)-1)}_{+;j} ~| ~ 1\leq v\leq z(\pi)-1, ~ 1\leq j\leq m^{(v)}\},$$
\endgroup
\begin{itemize}
\item If $k_{d}'=\kappa^{(v\rightarrow z(\pi)-1)}_{-;j}$, then
$$i_{k_{d}'}=\kappa^{(v\rightarrow z(\pi)-1)}_{-;j}-\sum_{\gamma}\left(\vartheta^{(v_{\gamma})}_{j_{\gamma}}+\eta^{(v_{\gamma})}_{j_{\gamma}-1}\right),$$
where,
\begingroup
\large
$\kappa^{(v\rightarrow z(\pi)-1)}_{-;j}=\kappa^{(v_{\gamma}\rightarrow z(\pi)-1)}_{-;j_{\gamma}}$.
\endgroup
\item If $k_{d}'=\kappa^{(v\rightarrow z(\pi)-1)}_{+;j}$, then
$$i_{k_{d}'}=\sum_{\gamma}\left(\vartheta^{(v_{\gamma})}_{j_{\gamma}}+\eta^{(v_{\gamma})}_{j_{\gamma}}\right)$$
where,
\begingroup
\large
 $\kappa^{(v\rightarrow z(\pi)-1)}_{+;j}=\kappa^{(v_{\gamma}\rightarrow z(\pi)-1)}_{+;j_{\gamma}}.$
\endgroup
\end{itemize}

Now, consider $t_{k_{d}'}^{i_{k_{d}'}}$ for $1\leq d\leq m'$. By Proposition \ref{commute-perm}, $t_{k_{d}'}^{i_{k_{d}'}}$ commutes with every sub-word of ${\left(\pi^{(z(\pi))}\right)}^{-1}$ of the form  $t_{\kappa_{-;j}}^{\varrho_{j-1}}\cdot t_{\kappa_{+;j}}^{\vartheta_{j}}$, such that $k_{d}'\leq\varrho_{j-1}$.
Therefore, the following holds:
If $k_{d}'$ for some $1\leq d\leq m'$ satisfies that $\varrho_{g-1}<k_{d}'\leq \varrho_{g}$, for some $1\leq g\leq m^{(z(\pi))}$, then  $t_{k_{d}'}^{i_{k_{d}'}}$ commutes with the sub-word
\begingroup
\large
$$\prod_{j=g+1}^{m^{(z(\pi))}}\left(t_{\kappa_{-;j}}^{\varrho_{j-1}}\cdot t_{\kappa_{+;j}}^{\vartheta_{j}}\right),$$
\endgroup
but does not commute with the sub-word
\begingroup
\large
$$t_{\kappa_{-;g}}^{\varrho_{g-1}}\cdot t_{\kappa_{+;g}}^{\vartheta_{g}}$$
\endgroup
of $\pi^{(z(\pi))}$.
For every $1\leq j\leq m^{(z(\pi))}$, denote by $\mu_{j}$ the number of integers $k_{d}'$, such that  $\varrho_{j-1}<k_{d}'\leq \varrho_{j}$, where $1\leq d\leq m'$ , denote by $(\pi')_{j}^{-1}$ the sub-word $\prod_{x=1}^{\mu_{j}}t_{k_{j_{x}}'}$ of $\pi^{-1}$, where $\varrho_{j-1}<k_{j_{x}}'\leq \varrho_{j}$. (It might happen that ~$\mu_{j}=0$ for some $j$, if there is no $k_{d}'$ such that $\varrho_{j-1}<k_{d}'\leq \varrho_{j}$. Then, $(\pi')^{-1}_{j}$ is the empty word).
Then, we have
\begingroup
\large
\begin{equation}\label{pi-1-order}
\pi^{-1}=\prod_{j=1}^{m^{(z(\pi))}}\left(t_{\kappa_{-;j}}^{\varrho_{j-1}}\cdot t_{\kappa_{+;j}}^{\vartheta_{j}}\cdot (\pi')_{j}^{-1}\right)=\prod_{j=1}^{m^{(z(\pi))}}\left(t_{\kappa_{-;j}}^{\varrho_{j-1}}\cdot t_{\kappa_{+;j}}^{\vartheta_{j}}\cdot\prod_{x=1}^{\mu_{j}}t_{k_{j_{x}}'}^{i_{k_{j_{x}}'}}\right).
\end{equation}
\endgroup
If
\begingroup
\large
${k_{m'}'}\leq \varrho_{m^{(z(\pi))}-1}$,
\endgroup
then there is a terminal segment of $\pi^{-1}$, which is a terminal segment of ${(\pi^{(z(\pi))})}^{-1}$ too. Then, denote by $\varphi$ the value of $j$ such that $\varrho_{j-1}<{k_{m'}'}\leq\varrho_{j}$, and denote by $\pi^{-1}_{0}$ the terminal segment of $\pi^{-1}$ of the form:
\begingroup
\large
\begin{equation}\label{terminal-pi-0}
\pi^{-1}_{0}=\prod_{j=\varphi+1}^{m^{(z(\pi))}}\left(t_{\kappa_{-;j}}^{\varrho_{j-1}}\cdot t_{\kappa_{+;j}}^{\vartheta_{j}}\right).
\end{equation}
\endgroup
Since $k_{m'}'\leq \varrho_{\varphi}$, and $k_{d}'<k_{m'}'$ for every $d<m'$, we have $k_{d}'\leq \varrho_{j}$ for every $j\geq \varphi$. Moreover, by the induction hypothesis, every $\kappa^{(v\rightarrow \pi(z)-1)}_{+;m^{(v)}}$ equals to $k_{d}'$ for some\\ $1\leq d\leq m'$.  Therefore, $\kappa^{(v\rightarrow \pi(z)-1)}_{+;m^{(v)}}\leq\varrho_{j}$ for every $j\geq \varphi$. Then, by Proposition \ref{v1-x-v2}, $\kappa^{(v)}_{+;m^{(v)}}\leq \kappa_{+;\varphi}$. Since
by Proposition \ref{order-kappa}, $\kappa_{+;\varphi}<\kappa_{c;j}$, ~$\kappa^{(v)}_{c;j'}<\kappa^{(v)}_{+;m^{(v)}}$, and $\kappa^{(v)}_{-;m^{(v)}}<\kappa^{(v)}_{+;m^{(v)}}$ for every $j>\varphi$, ~$1\leq v\leq z(\pi)-1$, ~$1\leq j'\leq m^{(v)}-1$, and $c\in\{+,-\}$, we conclude
$\kappa^{(v)}_{c';j'}<\kappa_{c;j}$. Thus, by Definition \ref{eta}, $\eta_{j}=0$ for every $j\geq \varphi$. Therefore, the Theorem holds for the terminal segment $\pi^{-1}_{0}$ of $\pi^{-1}$.
\vskip 0.5cm
Now, we turn to the sub-word
\begingroup
\Large
\begin{equation}\label{pi-phi-segment}
t_{\kappa_{-;\varphi}}^{\varrho_{\varphi-1}}\cdot t_{\kappa_{+;\varphi}}^{\vartheta_{\varphi}}\cdot {(\pi')}^{-1}_{\varphi}=t_{\kappa_{-;\varphi}}^{\varrho_{\varphi-1}}\cdot t_{\kappa_{+;\varphi}}^{\vartheta_{\varphi}}\cdot\prod_{x=1}^{\mu_{\varphi}}t_{k_{\varphi_{x}}'}^{i_{k_{\varphi_{x}}'}},
\end{equation}
\endgroup
of $\pi^{-1}$.

Notice, by the induction hypothesis, all the conditions of Proposition \ref{inverse-factorization} hold for $\tilde{\pi}={(\pi')}^{-1}$.
Consider the standard $OGS$ elementary factorization of ${(\pi')}^{-1}$:
\begingroup
\Large
\begin{equation}\label{pi-ogs-fac}
{(\pi')}^{-1}=\prod_{w'=1}^{z({(\pi')}^{-1})}\cdot\prod_{x=1}^{m^{(w')}}
t_{\tilde{h}^{(w')}_{x}}^{\tilde{\imath}_{x}^{(w')}}
\end{equation}
\endgroup

Define ${(\pi')}^{-1}_{\varphi}$ to be
\begingroup
\Large
$${(\pi')}^{-1}_{\varphi}=\prod_{x=1}^{\mu(\varphi)}t_{k_{\varphi_{x}}'}^{i_{k_{\varphi_{x}}'}}.$$
\endgroup
Since ${(\pi')}^{-1}_{\varphi}$ is a terminal segment of ${(\pi')}^{-1}$, its standard $OGS$ elementary factorization is as follows:
\begingroup
\Large
\begin{equation}\label{pi-phi-ogs}
{(\pi')}^{-1}_{\varphi}=\tilde{\pi}^{(w_{\varphi})}_{\jmath_{\varphi}}\cdot \prod_{w'=w_{\varphi}+1}^{z({(\pi')}^{-1})}\cdot\prod_{x=1}^{m^{(w')}}
t_{\tilde{h}^{(w')}_{x}}^{\tilde{\imath}_{x}^{(w')}},
\end{equation}
\endgroup
for some $1\leq w\leq z({\pi'}^{-1})$ and $1\leq \jmath_{\varphi}\leq m^{(w_{\varphi})}$,
where
\begingroup
\Large
\begin{equation}\label{tilde-jmath-phi}
\tilde{\pi}^{(w_{\varphi})}_{\jmath_{\varphi}}=\prod_{x=\jmath_{\varphi}}^{m^{(w_{\varphi})}}t_{\tilde{h}^{(w_{\varphi})}_{x}}^{\tilde{\imath}_{x}^{(w_{\varphi})}}.
\end{equation}
\endgroup
Notice, all the elementary factors of ${(\pi')}^{-1}_{\varphi}$ apart from the first factor $\tilde{\pi}^{(w_{\varphi})}_{\jmath_{\varphi}}$ coincide with the terminal elementary factors ${{(\pi')}^{-1}}^{(w')}$ ~of ~${(\pi')}^{-1}$, for $w'\geq w_{\varphi}+1$.
\\

Now, look at
\begingroup
\Large
$$t_{\kappa_{-;\varphi}}^{\varrho_{\varphi-1}}\cdot t_{\kappa_{+;\varphi}}^{\vartheta_{\varphi}}\cdot {(\pi')}^{-1}_{\varphi}.$$
\endgroup
Notice, by Definition \ref{kappa}, $\kappa_{+;\varphi}-\vartheta_{\varphi}=\varrho_{\varphi}$. By our assumption, $\varrho_{\varphi}\geq\tilde{h}^{(w_{\varphi})}_{x}$ for \\ $\jmath_{\varphi}\leq x\leq m^{(w_{\varphi})}$ and also  $\varrho_{\varphi}\geq\tilde{h}^{(w')}_{x}$ for every $w'>w_{\varphi}$ and $1\leq x\leq m^{(w')}$. Thus, by Proposition \ref{exchange-elementary2}, for $q=\kappa_{+;\varphi}$ and $i_{q}=\vartheta_{\varphi}$, we have:
\begingroup
\Large
\begin{equation}\label{exchange-z-1}
t_{\kappa_{+;\varphi}}^{\vartheta_{\varphi}}\cdot {(\pi')}^{-1}_{\varphi}=t_{\rho^{(w_{\varphi})}_{\jmath_{\varphi}}+\vartheta_{\varphi}}^{\vartheta_{\varphi}}\cdot {(\breve{\pi}')}^{-1}_{\varphi}\cdot t_{\kappa_{+;\varphi}}^{\vartheta_{\varphi}}.
\end{equation}
\endgroup
where,
\begingroup
\Large
\begin{equation}\label{breve}
{(\breve{\pi}')}^{-1}_{\varphi}=\prod_{x={\jmath}_{\varphi}}^{m^{(w_{\varphi})}}
t_{\tilde{h}^{(w')}_{x}+\vartheta_{\varphi}}^{\tilde{\imath}_{x}^{(w')}}\cdot\prod_{w'=w_{\varphi}+1}^{z({(\pi)'}^{-1})}\cdot \left(t_{\rho^{(w')}_{1}+\vartheta_{\varphi}}^{\vartheta_{\varphi}}\cdot \prod_{x=1}^{m^{(w')}}
t_{\tilde{h}^{(w')}_{x}+\vartheta_{\varphi}}^{\tilde{\imath}_{x}^{(w')}}\right).
\end{equation}
\endgroup
Thus, by Proposition \ref{exchange-elementary2}, the standard $OGS$ of ${(\breve{\pi}')}^{-1}_{\varphi}$ is the following:
\begin{itemize}
\item
\begingroup
\large
$t_{k_{\varphi_{j}}'+\vartheta_{\varphi}}$,
\endgroup
 for $1\leq j\leq \mu(\varphi)$;
\item
\begingroup
\large
$t_{maj\left({{(\pi')}^{-1}}^{(w')}\right)+\vartheta_{\varphi}}$,
\endgroup
 for some $w'>w_{\varphi}$, such that  $$\tilde{h}^{(w'-1)}_{m^{(w'-1)}}<maj\left({{(\pi')}^{-1}}^{(w')}\right)<\tilde{h}^{(w')}_{1}.$$
\end{itemize}
Notice also, by Definition \ref{kappa}, $\kappa_{-;\varphi}=\varrho_{\varphi-1}+\vartheta_{\varphi}$, ~$\kappa_{+;\varphi}=\varrho_{\varphi}+\vartheta_{\varphi}$. By our assumption $\varrho_{\varphi-1}<k_{\varphi_{j}}'\leq \varrho_{\varphi}$ for every $1\leq j\leq \mu(\varphi)$. Therefore, $$\kappa_{-;\varphi}<k_{\varphi_{j}}'+\vartheta_{\varphi}\leq \kappa_{+;\varphi}.$$
We have also, $\tilde{h}^{(w'-1)}_{m^{(w'-1)}}\leq maj\left({{(\pi')}^{-1}}^{(w')}\right)\leq\tilde{h}^{(w')}_{1}$, for every $w'>w_{\varphi}$. Therefore, $$\kappa_{-;\varphi}<maj\left({{(\pi')}^{-1}}^{(w')}\right)+\vartheta_{\varphi}\leq \kappa_{+;\varphi}$$
also for every $w'>w_{\varphi}$. Thus, the non-zero exponent standard $OGS$ of ${(\breve{\pi}')}^{-1}_{\varphi}$ included in the non-zero exponent standard $OGS$ of $\pi^{-1}$.

Now, notice, If $w'>w_{\varphi}$ and $1\leq x\leq m^{(w_{\varphi})}$, or $w'=w_{\varphi}$ and $\jmath_{\varphi}\leq x\leq m^{(w_{\varphi})}$, then by Proposition \ref{inverse-factorization}, the following holds:
\begin{itemize}
 \item $$\rho^{(w')}_{1}=maj\left({{(\pi')}^{-1}_{\varphi}}^{(w')}\right)=\kappa^{(v\rightarrow z(\pi)-1)}_{-;j},$$
for some $1\leq v\leq z(\pi)$ and $1\leq j\leq m^{(v)}$;
\item $$\tilde{h}^{(w')}_{x}=\kappa^{(v\rightarrow z(\pi)-1)}_{c;j},$$ for some $1\leq v\leq z(\pi)$ and $1\leq j\leq m^{(v)}$, and $c\in \{+,-\}$, such that in case $$\tilde{h}^{(w')}_{x}=\kappa^{(v\rightarrow z(\pi)-1)}_{-;j},$$ necessarily, $x=1$ or $x=m^{(w')}$ (i.e.,
    \begingroup
    \large
    $t_{\kappa^{(v\rightarrow z(\pi)-1)}_{-;j}}$
    \endgroup
     can appear just as a first or the last letter of an elementary factor ${{(\pi')}^{-1}}^{(w')}$).
\end{itemize}
Since for every $w'\geq w_{\varphi}$ and $1\leq x\leq m^{(w_{\varphi})}$, $$\varrho_{\varphi-1}<\tilde{h}^{(w')}_{x}=\kappa^{(v\rightarrow z(\pi)-1)}_{c;j}\leq\varrho_{\varphi}$$
and for $w'>w_{\varphi}$
$$\varrho_{\varphi-1}< maj\left({{(\pi')}^{-1}}^{(w')}\right)=\kappa^{(v\rightarrow z(\pi)-1)}_{-;j}\leq\varrho_{\varphi},$$
by Definition \ref{kappa}:
$$\tilde{h}^{(w')}_{x}+\vartheta_{\varphi}=\kappa^{(v\rightarrow z(\pi)-1)}_{c;j}+\vartheta_{\varphi}=\kappa^{(v)}_{c;j},$$

$$maj\left({{(\pi')}^{-1}}^{(w')}\right)+\vartheta_{\varphi}=\kappa^{(v\rightarrow z(\pi)-1)}_{-;j}+\vartheta_{\varphi}=\kappa^{(v)}_{-;j},$$
for some $1\leq v\leq z(\pi)$ and $1\leq j\leq m^{(v)}$.
Therefore, the standard $OGS$ canonical form of ${(\breve{\pi}')}^{-1}_{\varphi}$ is a product of exponents of $t_{\kappa^{(v)}_{-;j}}$ and $t_{\kappa^{(v)}_{+;j}}$, such that the following holds:
\begin{itemize}
\item If $k_{\varphi_{j'}}'+\vartheta_{\varphi}=\kappa^{(v)}_{+;j}$  ~for some $1\leq v\leq z(\pi)$, ~$1\leq j\leq m^{(v)}$, and \\ $1\leq j'\leq \mu(\varphi)$, then by Proposition \ref{no-common-eps12}, $k_{\varphi_{j'}}'+\vartheta_{\varphi}\neq \kappa^{(\bar{v})}_{-;\bar{j}}$, for any $1\leq \bar{v}\leq z(\pi)$ and $1\leq \bar{j}\leq m^{(\bar{v})}$. Thus, $k_{\varphi_{j'}}'+\vartheta_{\varphi}\neq maj\left({{(\pi')}^{-1}}^{(w')}\right)+\vartheta_{\varphi}$ for any $w'>w_{\varphi}$. Therefore, by Proposition \ref{exchange-elementary2}, $$\tilde{i}_{\tilde{k}_{\tilde{j}}}=i'_{k_{\varphi_{j'}}'}=i'_{\kappa^{(v\rightarrow z(\pi)-1)}_{+;j}},$$
where, $\tilde{k}_{\tilde{j}}=\kappa^{(v)}_{+;j}$ for some $1\leq v\leq z(\pi)$, $1\\leq j\leq m^{(v)}$.
Therefore, by the induction hypothesis,
\begingroup
\large
\begin{equation}\label{kappa+exponent}
\tilde{i}_{\kappa^{(v)}_{+;j}}=i'_{\kappa^{(v\rightarrow z(\pi)-1)}_{+;j}}=\sum_{\gamma}\left(\vartheta^{(v_{\gamma})}_{j_{\gamma}}+\eta^{(v_{\gamma})}_{j_{\gamma}}\right),
\end{equation}
\endgroup
such that $\kappa^{(v_{\gamma})}_{+;j_{\gamma}}=\kappa^{(v)}_{+;j}$, where
\begingroup
\large
$$\kappa_{-;\varphi}=\varrho_{\varphi-1}+\vartheta_{\varphi}<\kappa^{(v)}_{+;j}\leq \varrho_{\varphi}+\vartheta_{\varphi}=\kappa_{+;\varphi};$$
\endgroup
\item If $k_{\varphi_{j'}}'+\vartheta_{\varphi}=\kappa^{(v)}_{-;j}$  ~for some $1\leq v\leq z(\pi)$, ~$1\leq j\leq m^{(v)}$, and $1\leq j'\leq \mu_{\varphi}$, then by Proposition \ref{inverse-factorization}, $k_{\varphi_{j'}}'+\vartheta_{\varphi}=maj\left({{(\pi')}^{-1}}^{(w')}\right)+\vartheta_{\varphi}$ for some $w'>w_{\varphi}$. Therefore, by Proposition \ref{exchange-elementary2},
    \begingroup
    \large
    $$\tilde{i}_{\tilde{k}_{\tilde{j}}}=i'_{k_{\varphi_{j'}}'}+\vartheta_{\varphi}=i'_{\kappa^{(v\rightarrow z(\pi)-1)}_{-;j}}+\vartheta_{\varphi},$$
    \endgroup
    where, $\tilde{k}_{\tilde{j}}=\kappa^{(v)}_{-;j}$ for some $1\leq v\leq z(\pi)$, ~$1\leq j\leq m^{(v)}$. Therefore, by the induction hypothesis,
    \begingroup
    \large
    \begin{align*}\label{kappa-exponent}
    \tilde{i}_{\kappa^{(v)}_{-;j}}=i'_{\kappa^{(v\rightarrow z(\pi)-1)}_{-;j}}+\vartheta_{\varphi}&=\kappa^{(v\rightarrow z(\pi)-1)}_{-;j}+\vartheta_{\varphi}-\sum_{\gamma}\left(\vartheta^{(v_{\gamma})}_{j_{\gamma}}+\eta^{(v_{\gamma})}_{j_{\gamma}-1}\right) \\&=
    \kappa^{(v)}_{-;j}-\sum_{\gamma}\left(\vartheta^{(v_{\gamma})}_{j_{\gamma}}+\eta^{(v_{\gamma})}_{j_{\gamma}-1}\right),
    \end{align*}
    \endgroup
such that $\kappa^{(v_{\gamma})}_{-;j_{\gamma}}=\kappa^{(v)}_{-;j}$, where
\begingroup
\large
$$\kappa_{-;\varphi}=\varrho_{\varphi-1}+\vartheta_{\varphi}<\kappa^{(v)}_{-;j}<\varrho_{\varphi}+\vartheta_{\varphi}=\kappa_{+;\varphi};$$
\endgroup
\item If $\tilde{h}^{(w'-1)}_{m^{(w'-1)}}<maj\left({{(\pi')}^{-1}}^{(w')}\right)<\tilde{h}^{(w')}_{1}$, for some $w'>w_{\varphi}$, then by Proposition \ref{exchange-elementary2}, the exponent of ~$t_{maj\left({{(\pi')}^{-1}}^{(w')}\right)}$ is ~$i_{maj\left({{(\pi')}^{-1}}^{(w')}\right)}=0$ in the standard $OGS$ canonical form. By Proposition \ref{inverse-factorization}, $maj\left({{(\pi')}^{-1}}^{(w')}\right)=\kappa^{(v\rightarrow z(\pi)-1)}_{-;j}$  ~for some \\ $1\leq v\leq z(\pi)$ and $1\leq j\leq m^{(v)}$. Therefore, by the induction hypothesis, $$i_{maj\left({{(\pi')}^{-1}}^{(w')}\right)}=0=\kappa^{(v\rightarrow z(\pi)-1)}_{-;j}-\sum_{\gamma}\left(\vartheta^{(v_{\gamma})}_{j_{\gamma}}+\eta^{(v_{\gamma})}_{j_{\gamma}-1}\right),$$
such that $\kappa^{(v_{\gamma})}_{-;j_{\gamma}}=\kappa^{(v\rightarrow z(\pi)-1)}_{-;j}$. Thus, by Proposition \ref{inverse-factorization},
\begingroup
\large
\begin{align*}
    \tilde{i}_{maj\left({{(\pi')}^{-1}}^{(w')}\right)}=\vartheta_{\varphi}&=\kappa^{(v\rightarrow z(\pi)-1)}_{-;j}+\vartheta_{\varphi}-\sum_{\gamma}\left(\vartheta^{(v_{\gamma})}_{j_{\gamma}}+\eta^{(v_{\gamma})}_{j_{\gamma}-1}\right) \\&=
    \kappa^{(v)}_{-;j}-\sum_{\gamma}\left(\vartheta^{(v_{\gamma})}_{j_{\gamma}}+\eta^{(v_{\gamma})}_{j_{\gamma}-1}\right),
    \end{align*}
\endgroup
such that $\kappa^{(v_{\gamma})}_{-;j_{\gamma}}=\kappa^{(v)}_{-;j}$, where
\begingroup
\large
$$\kappa_{-;\varphi}=\varrho_{\varphi-1}+\vartheta_{\varphi}<\kappa^{(v)}_{-;j}<\varrho_{\varphi}+\vartheta_{\varphi}=\kappa_{+;\varphi};$$
\endgroup
\end{itemize}
Now, look at the sub-word
\begingroup
\Large
$$t_{\kappa_{-;\varphi}}^{\varrho_{\varphi-1}}\cdot t_{\rho^{(w_{\varphi})}_{\jmath_{\varphi}}+\vartheta_{\varphi}}^{\vartheta_{\varphi}}$$
\endgroup
of $\pi^{-1}$. If $\jmath_{\varphi}>1$, by Proposition \ref{inverse-factorization}, $\rho^{(w_{\varphi})}_{\jmath_{\varphi}}<\tilde{h}^{(w_{\varphi})}_{\jmath_{\varphi}-1}$. By our assumption,\\ $\tilde{h}^{(w_{\varphi})}_{\jmath_{\varphi}-1}<k_{\varphi_{1}}'\leq \varrho_{\varphi-1}$. Thus,
\begingroup
\large
$$\rho^{(w_{\varphi})}_{\jmath_{\varphi}}+\vartheta_{\varphi}<\varrho_{\varphi-1}+\vartheta_{\varphi}=\kappa_{-,\varphi}.$$
\endgroup
Therefore, by using Proposition \ref{exchange-2}, for the case $q-i_{q}=i_{p}$, where $q=\kappa_{-;\varphi}$, ~$i_{q}=\varrho_{\varphi-1}$, ~$p=\rho^{(w_{\varphi})}_{\jmath_{\varphi}}+\vartheta_{\varphi}$, and ~$i_{p}=\vartheta_{\varphi}$, we have:
\begingroup
\Large
\begin{equation}\label{exchange-eta}
t_{\kappa_{-;\varphi}}^{\varrho_{\varphi-1}}\cdot t_{\rho^{(w_{\varphi})}_{\jmath_{\varphi}}+\vartheta_{\varphi}}^{\vartheta_{\varphi}}=t_{\varrho_{\varphi-1}}^{\rho^{(w_{\varphi})}_{\jmath_{\varphi}}}\cdot t_{\kappa_{-;\varphi}}^{\varrho_{\varphi-1}-\rho^{(w_{\varphi})}_{\jmath_{\varphi}}}.
\end{equation}
\endgroup
By Equation \ref{tilde-jmath-phi},
$$\rho^{(w_{\varphi})}_{\jmath_{\varphi}}=maj\left(\tilde{\pi}^{(w_{\varphi})}_{\jmath_{\varphi}}\right).$$
Therefore, by Proposition \ref{inverse-factorization} the following holds:
\begin{itemize}
\item If $\tilde{h}^{(w_{\varphi})}_{\jmath_{\varphi}}=\kappa^{(v\rightarrow z(\pi)-1)}_{+;j}$, then
$$\rho^{(w_{\varphi})}_{\jmath_{\varphi}}=\sum_{v'=\chi^{(v)}_{j}+1}^{z(\pi)-1}\vartheta^{(v')}_{j'_{v':v}}$$ where, $$\kappa^{(v'\rightarrow z(\pi)-1)}_{-;j'_{v':v}}<\kappa^{(v\rightarrow z(\pi)-1)}_{+;j}\leq\kappa^{(v'\rightarrow z(\pi)-1)}_{+;j'_{v':v}};$$
\item If $\tilde{h}^{(w_{\varphi})}_{\jmath_{\varphi}}=\kappa^{(v\rightarrow z(\pi)-1)}_{-;j}$ and $\jmath_{\varphi}=m^{(w_{\varphi})}$, then
$$\rho^{(w_{\varphi})}_{\jmath_{\varphi}}=\sum_{v'=v+1}^{z(\pi)-1}\vartheta^{(v')}_{j'_{v':v}}$$
where,
$$\kappa^{(v'\rightarrow z(\pi)-1)}_{-;j'_{v':v}}\leq\kappa^{(v\rightarrow z(\pi)-1)}_{-;j}<\kappa^{(v'\rightarrow z(\pi)-1)}_{+;j'_{v':v}}.$$
 \end{itemize}
 Notice, $\tilde{h}^{(w_{\varphi})}_{\jmath_{\varphi}}+\vartheta_{\pi}=\kappa^{(v)}_{c;j}$ in case $\tilde{h}^{(w_{\varphi})}_{\jmath_{\varphi}}=\kappa^{(v\rightarrow z(\pi)-1)}_{c;j}$ for $c\in \{+, -\}$. Notice also, $\tilde{h}^{(w_{\varphi})}_{\jmath_{\varphi}}+\vartheta_{\pi}$ is the smallest value of $\kappa^{(v)}_{c;j}$, which is greater than  $\kappa_{-;\varphi}$, unless $\jmath_{\varphi}=1$ and $\tilde{h}^{(w_{\varphi})}_{\jmath_{\varphi}}=maj\left(\tilde{\pi}^{(w_{\varphi})}\right)$.  Therefore, by Proposition \ref{kappa-adjacent} the following holds:
 \begin{itemize}
 \item If $\tilde{h}^{(w_{\varphi})}_{\jmath_{\varphi}}+\vartheta_{\pi}=\kappa^{(v)}_{-;j}$ (i.e., $\tilde{h}^{(w_{\varphi})}_{\jmath_{\varphi}}=\kappa^{(v\rightarrow z(\pi)-1)}_{-;j}$), for some $1\leq v\leq z(\pi)$ and $1\leq j\leq m^{(v)}$, then $\chi_{\varphi-1}=v$;
 \item If $\tilde{h}^{(w_{\varphi})}_{\jmath_{\varphi}}+\vartheta_{\pi}=\kappa^{(v)}_{+;j}$ (i.e., $\tilde{h}^{(w_{\varphi})}_{\jmath_{\varphi}}=\kappa^{(v\rightarrow z(\pi)-1)}_{+;j}$), for some $1\leq v\leq z(\pi)$ and $1\leq j\leq m^{(v)}$, then $\chi_{\varphi-1}=\chi^{(v)}_{j}$;
 \end{itemize}
Therefore, unless $\jmath_{\varphi}=1$ and $\tilde{h}^{(w_{\varphi})}_{\jmath_{\varphi}}=maj\left(\tilde{\pi}^{(w_{\varphi})}\right)$,
$$\rho^{(w_{\varphi})}_{\jmath_{\varphi}}=\sum_{v'=\chi_{\varphi}}^{z(\pi)-1}\vartheta^{(v')}_{j'_{v'}}$$ where, $$\kappa^{(v')}_{-;j'_{v'}}<\kappa_{-;\varphi}<\kappa^{(v')}_{+;j'_{v'}}.$$
Thus, by Definition \ref{eta},
\begingroup
\large
\begin{equation}\label{rho-j-eta}
\rho^{(w_{\varphi})}_{\jmath_{\varphi}}=\eta_{\varphi-1}.
\end{equation}
\endgroup
Hence, Equation \ref{exchange-eta} can be written in the form
\begingroup
\Large
\begin{equation}\label{exchange-eta2}
t_{\kappa_{-;\varphi}}^{\varrho_{\varphi-1}}\cdot t_{\rho^{(w_{\varphi})}_{\jmath_{\varphi}}+\vartheta_{\varphi}}^{\vartheta_{\varphi}}=t_{\varrho_{\varphi-1}}^{\eta_{\varphi-1}}\cdot t_{\kappa_{-;\varphi}}^{\varrho_{\varphi-1}-\eta_{\varphi-1}}.
\end{equation}
\endgroup
Therefore, the following holds unless $\jmath_{\varphi}=1$ and  $maj\left(\tilde{\pi}^{(w_{\varphi})}\right)\geq \varrho_{j-1}$:
\begingroup
\Large
\begin{equation}\label{phi-segment}
t_{\kappa_{-;\varphi}}^{\varrho_{\varphi-1}}\cdot t_{\kappa_{+;\varphi}}^{\vartheta_{\varphi}}\cdot {(\pi')}^{-1}_{\varphi}=t_{\varrho_{\varphi-1}}^{\eta_{\varphi-1}}\cdot t_{\kappa_{-;\varphi}}^{\varrho_{\varphi-1}-\eta_{\varphi-1}}\cdot {(\breve{\pi}')}^{-1}_{\varphi}\cdot t_{\kappa_{+;\varphi}}^{\vartheta_{\varphi}}.
\end{equation}
\endgroup

 Now, assume $\jmath_{\varphi}=1$ and  $maj\left(\tilde{\pi}^{(w_{\varphi})}\right)\geq \varrho_{j-1}$. By Proposition \ref{inverse-factorization}, \\ $maj\left(\tilde{\pi}^{(w_{\varphi})}\right)=\kappa^{(v_{0}\rightarrow z(\pi)-1)}_{-;j_{0}}$, for some $1\leq v_{0}\leq z(\pi)-1$ and $1\leq j_{0}\leq m^{(v)}$. Notice, also $\tilde{h}^{(w_{\varphi}-1)}_{m^{(w_{\varphi}-1)}}<maj\left(\tilde{\pi}^{(w_{\varphi})}\right)$, since otherwise $w_{\varphi}$ would be smaller by one, and $\jmath_{\varphi}$ would be $m^{(w_{\varphi})}$. Thus, by Proposition \ref{inverse-factorization}, either $v_{0}=z(\pi)-1$, or for every $v_{0}<v'\leq z(\pi)-1$, we have $\kappa^{(v')}_{-;1}=\vartheta^{(v')}_{1}=0$ and  $\kappa^{(v_{0}\rightarrow z(\pi)-1)}_{-;j_{0}}<\kappa^{(v'\rightarrow z(\pi)-1)}_{+;1}$. By Proposition \ref{v1-x-v2}, the same holds for  $\kappa^{(v_{0})}_{-;j_{0}}=maj\left(\tilde{\pi}^{(w_{\varphi})}\right)+\vartheta_{\varphi}$ too (i.e., either $v_{0}=z(\pi)-1$, or for every \\ $v_{0}<v'\leq z(\pi)-1$, we have $\kappa^{(v')}_{-;1}=\vartheta^{(v')}_{1}=0$ and $\kappa^{(v_{0})}_{-;j_{0}}<\kappa^{(v')}_{+;1}$).\\ Notice also, $\kappa^{(v_{0})}_{-;j_{0}}=maj\left(\tilde{\pi}^{(w_{\varphi})}\right)+\vartheta_{\varphi}$, is the smallest $\kappa^{(v)}_{c;j}$ which is larger than $\kappa_{-;\varphi}$. Thus, by Proposition \ref{kappa-adjacent}, $\chi_{\varphi-1}=v_{0}$. Then, by Definition \ref{eta}, $\eta_{\varphi-1}=0$.
Thus, in case $\eta_{\varphi-1}=0$, we have
\begingroup
\Large
\begin{equation}\label{phi-segment-eta-zero}
t_{\kappa_{-;\varphi}}^{\varrho_{\varphi-1}}\cdot t_{\kappa_{+;\varphi}}^{\vartheta_{\varphi}}\cdot {(\pi')}^{-1}_{\varphi}= t_{\kappa_{-;\varphi}}^{\varrho_{\varphi-1}}\cdot t_{\rho^{(w_{\varphi})}_{1}+\vartheta_{\varphi}}^{\vartheta_{\varphi}}\cdot{(\breve{\pi}')}^{-1}_{\varphi}\cdot t_{\kappa_{+;\varphi}}^{\vartheta_{\varphi}},
\end{equation}
\endgroup
where, $\jmath_{\varphi}=1$ and $\rho^{(w_{\varphi})}_{1}=maj\left(\tilde{\pi}^{(w_{\varphi})}\right)=\tilde{h}^{(w_{\varphi})}_{1}\geq \varrho_{\varphi-1}$.

Thus, by defining
\begingroup
\large
$${(\check{\pi}')}^{-1}_{\varphi}=\begin{cases} {(\breve{\pi}')}^{-1}_{\varphi} & ~~\eta_{\varphi-1}>0 \\ \\ t_{\rho^{(w_{\varphi})}_{1}+\vartheta_{\varphi}}^{\vartheta_{\varphi}}\cdot{(\breve{\pi}')}^{-1}_{\varphi} & ~~\eta_{\varphi-1}=0\end{cases}$$
\endgroup
we have:
\begingroup
\Large
\begin{equation}\label{phi-segment-eta-general}
t_{\kappa_{-;\varphi}}^{\varrho_{\varphi-1}}\cdot t_{\kappa_{+;\varphi}}^{\vartheta_{\varphi}}\cdot {(\pi')}^{-1}_{\varphi}=t_{\varrho_{\varphi-1}}^{\eta_{\varphi-1}}\cdot t_{\kappa_{-;\varphi}}^{\varrho_{\varphi-1}-\eta_{\varphi-1}}\cdot {(\check{\pi}')}^{-1}_{\varphi}\cdot t_{\kappa_{+;\varphi}}^{\vartheta_{\varphi}}.
\end{equation}
\endgroup

If $\varphi=1$, then $\varrho_{\varphi-1}=0$. Thus, the Theorem has been proved for case $\varphi=1$. If $\varphi>1$, then consider a segment of $\pi^{-1}$  by combining Equations \ref{pi-1-order} and \ref{phi-segment-eta-general}.
\begingroup
\large
\begin{align*}
\pi^{-1}=\left[\prod_{j=1}^{\varphi-2}\left(t_{\kappa_{-;j}}^{\varrho_{j-1}}\cdot t_{\kappa_{+;j}}^{\vartheta_{j}}\cdot (\pi')_{j}^{-1}\right)\right]&\cdot\left(t_{\kappa_{-;\varphi-1}}^{\varrho_{\varphi-2}}\cdot t_{\kappa_{+;\varphi-1}}^{\vartheta_{\varphi-1}}\cdot(\pi')_{\varphi-1}^{-1}\cdot t_{\varrho_{\varphi-1}}^{\eta_{\varphi-1}}\right) \\ &\cdot\left(t_{\kappa_{-;\varphi}}^{\varrho_{\varphi-1}-\eta_{\varphi-1}}\cdot {(\check{\pi}')}^{-1}_{\varphi}\cdot t_{\kappa_{+;\varphi}}^{\vartheta_{\varphi}}\right)\cdot \pi^{-1}_{0},
\end{align*}
\endgroup
where, by Equation \ref{pi-1-order},
\begingroup
\large
$$(\pi')_{\varphi-1}^{-1}=\prod_{x=1}^{\mu_{\varphi-1}}t_{k_{{(\varphi-1)}_{x}}'}^{i_{k_{{(\varphi-1)}_{x}}'}}.$$
\endgroup
such that
\begingroup
\large
$$\varrho_{\varphi-2}<k_{{(\varphi-1)}_{x}}'<k_{{(\varphi-1)}_{x+1}}'\leq \varrho_{\varphi-1},$$
\endgroup
for every $1\leq x<\mu_{\varphi-1}$.
\\

Consider the segment
\begingroup
\large
\begin{equation}\label{phi-1-order}
t_{\kappa_{-;\varphi-1}}^{\varrho_{\varphi-2}}\cdot t_{\kappa_{+;\varphi-1}}^{\vartheta_{\varphi-1}}\cdot\prod_{x=1}^{\mu_{\varphi-1}}t_{k_{{(\varphi-1)}_{x}}'}^{i_{k_{{(\varphi-1)}_{x}}'}}\cdot t_{\varrho_{\varphi-1}}^{\eta_{\varphi-1}}
\end{equation}
\endgroup

\begingroup
\Large
\begin{equation}\label{pi-phi-1-ogs}
{(\pi')}^{-1}_{\varphi-1}\cdot t_{\varrho_{\varphi-1}}^{\eta_{\varphi-1}}=\tilde{\pi}^{(w_{\varphi-1})}_{\jmath_{\varphi-1}}\cdot \prod_{w'=w_{\varphi-1}+1}^{w_{\varphi}-1}\cdot\prod_{x=1}^{m^{(w')}}
t_{\tilde{h}^{(w')}_{x}}^{\tilde{\imath}_{x}^{(w')}}\cdot \prod_{x=1}^{\jmath_{\varphi}-1}
t_{\tilde{h}^{(w_{\varphi})}_{x}}^{\tilde{\imath}_{x}^{(w_{\varphi})}}\cdot t_{\varrho_{\varphi-1}}^{\eta_{\varphi-1}},
\end{equation}
\endgroup
for some $1\leq w\leq z({\pi'}^{-1})$ and $1\leq \jmath\leq m^{(w_{\varphi})}$,
where
\begingroup
\Large
\begin{equation}\label{tilde-jmath-phi-1}
\tilde{\pi}^{(w_{\varphi-1})}_{\jmath_{\varphi-1}}=\prod_{x=\jmath_{\varphi-1}}^{m^{(w_{\varphi-1})}}t_{\tilde{h}^{(w_{\varphi-1})}_{x}}^{\tilde{\imath}_{x}^{(w_{\varphi-1})}}.
\end{equation}
\endgroup

Consider the terminal segment of ${(\pi')}^{-1}_{\varphi-1}\cdot t_{\varrho_{\varphi-1}}^{\eta_{\varphi-1}}$ in Equation \ref{pi-phi-1-ogs}:

\begingroup
\large
\begin{equation}\label{terminal-ph-1}
\prod_{x=1}^{\jmath_{\varphi}-1}
t_{\tilde{h}^{(w_{\varphi})}_{x}}^{\tilde{\imath}_{x}^{(w_{\varphi})}}\cdot t_{\varrho_{\varphi-1}}^{\eta_{\varphi-1}}
\end{equation}
\endgroup

By Equation \ref{rho-j-eta}, $\eta_{\varphi-1}=\rho^{(w_{\varphi})}_{\jmath_{\varphi}}$, and by Definition \ref{rho}, $\rho^{(w_{\varphi})}_{\jmath_{\varphi}}=\sum_{x=\jmath_{\varphi}}^{m^{(w_{\varphi})}}\tilde{\imath}_{x}^{(w_{\varphi})}$. Thus,
\begingroup
\large
\begin{equation}\label{term-phi-1-eta}
\sum_{x=1}^{\jmath_{\varphi}-1}\tilde{\imath}_{x}^{(w_{\varphi})}+\eta_{\varphi-1}=\sum_{x=1}^{m^{(w_{\varphi})}}\tilde{\imath}_{x}^{(w_{\varphi})}=maj\left(\tilde{\pi}^{(w_{\varphi})}\right).
\end{equation}
\endgroup
Thus, Equation \ref{pi-phi-1-ogs} is a standard $OGS$ elementary factorization of
\begingroup
\large
${(\pi')}^{-1}_{\varphi-1}\cdot t_{\varrho_{\varphi-1}}^{\eta_{\varphi-1}}$,
\endgroup
 with the elementary factors
 \begingroup
 \large
  $\tilde{\pi}^{(w_{\varphi-1})}_{\jmath_{\varphi-1}}$, ~$\prod_{x=1}^{m^{(w')}}
t_{\tilde{h}^{(w')}_{x}}^{\tilde{\imath}_{x}^{(w')}}$,
\endgroup
where
\begingroup
\large
$w_{\varphi-1}+1\leq w'\leq w_{\varphi}-1$,
\endgroup
and the terminal elementary factor
\begingroup
\large
 $\prod_{x=1}^{\jmath_{\varphi}-1}
t_{\tilde{h}^{(w_{\varphi})}_{x}}^{\tilde{\imath}_{x}^{(w_{\varphi})}}\cdot t_{\varrho_{\varphi-1}}^{\eta_{\varphi-1}}$.
\endgroup

Notice, similar to the $OGS$ elementary factorization of ${(\pi')}^{-1}_{\varphi}$, we have that all the elementary factors of ${(\pi')}^{-1}_{\varphi-1}$ apart from the first factor $\tilde{\pi}^{(w_{\varphi})}_{\jmath_{\varphi-1}}$ and the last factor,  coincides with the terminal elementary factors ${{(\pi')}^{-1}}^{(w')}$ ~of ~${(\pi')}^{-1}$, for \\ $w_{\varphi-1}+1\leq w'\leq w_{\varphi}-1$.

Notice, by Definition \ref{kappa}, $\kappa_{+;\varphi-1}-\vartheta_{\varphi-1}=\varrho_{\varphi-1}$. By our assumption, \\ $\varrho_{\varphi-1}>\tilde{h}^{(w_{\varphi-1})}_{x}$ for $\jmath_{\varphi-1}\leq x\leq m^{(w_{\varphi-1})}$, ~$\varrho_{\varphi-1}>\tilde{h}^{(w')}_{x}$ for every $w_{\varphi-1}<w'<w_{\varphi}$ and $1\leq x\leq m^{(w')}$, and $\varrho_{\varphi-1}>\tilde{h}^{(w_{\varphi})}_{x}$ for $1\leq x\leq \jmath_{\varphi}-1$. Thus, by Proposition \ref{exchange-elementary2}, for $q=\kappa_{+;\varphi}$ and $i_{q}=\vartheta_{\varphi}$, we have:
\begingroup
\Large
\begin{equation}\label{exchnge-z-1-1}
t_{\kappa_{+;\varphi-1}}^{\vartheta_{\varphi-1}}\cdot {(\pi')}^{-1}_{\varphi-1}\cdot t_{\varrho_{\varphi-1}}^{\eta_{\varphi-1}}=t_{\rho^{(w_{\varphi-1})}_{\jmath_{\varphi-1}}+\vartheta_{\varphi-1}}^{\vartheta_{\varphi-1}}\cdot {(\breve{\pi}')}^{-1}_{\varphi-1}\cdot t_{\kappa_{+;\varphi-1}}^{\vartheta_{\varphi-1}+\eta_{\varphi-1}}.
\end{equation}
\endgroup
where,
\begingroup
\Large
\begin{align*}
{(\breve{\pi}')}^{-1}_{\varphi-1}&=\prod_{x={\jmath}_{\varphi-1}}^{m^{(w_{\varphi-1})}}
t_{\tilde{h}^{(w')}_{x}+\vartheta_{\varphi-1}}^{\tilde{\imath}_{x}^{(w')}} \\ &~\cdot\prod_{w'=w_{\varphi-1}+1}^{w_{\varphi}-1} \left(t_{\rho^{(w')}_{1}+\vartheta_{\varphi-1}}^{\vartheta_{\varphi-1}}\cdot \prod_{x=1}^{m^{(w')}}
t_{\tilde{h}^{(w')}_{x}+\vartheta_{\varphi-1}}^{\tilde{\imath}_{x}^{(w')}}\right) \\ &~\cdot \left(t_{\rho^{(w_{\varphi})}_{1}+\vartheta_{\varphi-1}}^{\vartheta_{\varphi-1}}\cdot \prod_{x=1}^{{\jmath}_{\varphi}-1}
t_{\tilde{h}^{(w_{\varphi})}_{x}+\vartheta_{\varphi-1}}^{\tilde{\imath}_{x}^{(w_{\varphi})}}\right).
\end{align*}
\endgroup
Thus, by Proposition \ref{exchange-elementary2}, the standard $OGS$ of ${(\breve{\pi}')}^{-1}_{\varphi-1}$ is the following:
\begin{itemize}
\item $t_{k_{\varphi_{j}}'+\vartheta_{\varphi-1}}$, for $1\leq j\leq \mu(\varphi-1)$;
\item $t_{maj\left({{(\pi')}^{-1}}^{(w')}\right)+\vartheta_{\varphi-1}}$, for some $w_{\varphi-1}<w'\leq w_{\varphi}$, such that  $$\tilde{h}^{(w'-1)}_{m^{(w'-1)}}<maj\left({{(\pi')}^{-1}}^{(w')}\right)<\tilde{h}^{(w')}_{1},$$
\end{itemize}
where, by the same argument as in the case of ${(\breve{\pi}')}^{-1}_{\varphi}$,
the standard $OGS$ canonical form of ${(\breve{\pi}')}^{-1}_{\varphi-1}$ is a product of exponents of $t_{\kappa^{(v)}_{-;j}}$ and $t_{\kappa^{(v)}_{+;j}}$, such that the following holds:
\begin{itemize}
\item If $k_{\varphi_{j'}}'+\vartheta_{\varphi-1}=\kappa^{(v)}_{+;j}$  ~for some $1\leq v\leq z(\pi)$, ~$1\leq j\leq m^{(v)}$, and $1\leq j'\leq \mu(\varphi-1)$, then
\begingroup
\large
\begin{equation}\label{kappa+exponent-phi-1}
\tilde{i}_{\kappa^{(v)}_{+;j}}=i'_{\kappa^{(v\rightarrow z(\pi)-1)}_{+;j}}=\sum_{\alpha}\left(\vartheta^{(v_{\alpha})}_{j_{\alpha}}+\eta^{(v_{\alpha})}_{j_{\alpha}}\right),
\end{equation}
\endgroup
such that $\kappa^{(v_{\alpha})}_{+;j_{\alpha}}=\kappa^{(v)}_{+;j}$, where
\begingroup
\large
$$\kappa_{-;\varphi-1}=\varrho_{\varphi-2}+\vartheta_{\varphi-1}<\kappa^{(v)}_{+;j}\leq \varrho_{\varphi-1}+\vartheta_{\varphi-1}=\kappa_{+;\varphi-1};$$
\endgroup
\item If $k_{\varphi_{j'}}'+\vartheta_{\varphi-1}=\kappa^{(v)}_{-;j}$  ~for some $1\leq v\leq z(\pi)$, ~$1\leq j\leq m^{(v)}$, and $1\leq j'\leq \mu(\varphi-1)$, then \begingroup
    \large
    \begin{align*}\label{kappa-exponent-phi-1}
    \tilde{i}_{\kappa^{(v)}_{-;j}}=i'_{\kappa^{(v\rightarrow z(\pi)-1)}_{-;j}}+\vartheta_{\varphi-1}&=\kappa^{(v\rightarrow z(\pi)-1)}_{-;j}+\vartheta_{\varphi-1}-\sum_{\alpha}\left(\vartheta^{(v_{\alpha})}_{j_{\alpha}}+\eta^{(v_{\alpha})}_{j_{\alpha}-1}\right) \\&=
    \kappa^{(v)}_{-;j}-\sum_{\alpha}\left(\vartheta^{(v_{\alpha})}_{j_{\alpha}}+\eta^{(v_{\alpha})}_{j_{\alpha}-1}\right),
    \end{align*}
    \endgroup
such that $\kappa^{(v_{\alpha})}_{-;j_{\alpha}}=\kappa^{(v)}_{-;j}$, where
\begingroup
\large
$$\kappa_{-;\varphi-1}=\varrho_{\varphi-2}+\vartheta_{\varphi-1}<\kappa^{(v)}_{-;j}<\varrho_{\varphi-1}+\vartheta_{\varphi-1}=\kappa_{+;\varphi-1};$$
\endgroup
\item If $\tilde{h}^{(w'-1)}_{m^{(w'-1)}}<maj\left({{(\pi')}^{-1}}^{(w')}\right)<\tilde{h}^{(w')}_{1}$, for some $w_{\varphi-1}<w'\leq w_{\varphi}$, then by Proposition \ref{exchange-elementary2}, the exponent of ~$t_{maj\left({{(\pi')}^{-1}}^{(w')}\right)}$ is ~$i_{maj\left({{(\pi')}^{-1}}^{(w')}\right)}=0$ in the standard $OGS$ canonical form of ${(\pi')}^{-1}_{\phi-1}$. We have also $maj\left({{(\pi')}^{-1}}^{(w')}\right)=\kappa^{(v\rightarrow z(\pi)-1)}_{-;j}$ for some $1\leq v\leq z(\pi)$ and $1\leq j\leq m^{(v)}$.
By similar argument as in the case of ${(\breve{\pi}')}^{-1}_{\varphi}$, we conclude
\begingroup
\large
\begin{align*}
    \tilde{i}_{maj\left({{(\pi')}^{-1}}^{(w')}\right)}=\vartheta_{\varphi-1}&=\kappa^{(v\rightarrow z(\pi)-1)}_{-;j}+\vartheta_{\varphi}-\sum_{\alpha}\left(\vartheta^{(v_{\alpha})}_{j_{\alpha}}+\eta^{(v_{\alpha})}_{j_{\alpha}-1}\right) \\&=
    \kappa^{(v)}_{-;j}-\sum_{\alpha}\left(\vartheta^{(v_{\alpha})}_{j_{\alpha}}+\eta^{(v_{\alpha})}_{j_{\alpha}-1}\right),
    \end{align*}
\endgroup
such that $\kappa^{(v_{\alpha})}_{-;j_{\alpha}}=\kappa^{(v)}_{-;j}$, where
\begingroup
\large
$$\kappa_{-;\varphi}=\varrho_{\varphi-1}+\vartheta_{\varphi}<\kappa^{(v)}_{-;j}<\varrho_{\varphi}+\vartheta_{\varphi}=\kappa_{+;\varphi};$$
\endgroup
\end{itemize}
Finally, analogously to Equations \ref{phi-segment}, \ref{phi-segment-eta-zero} we conclude the following:
\begingroup
\Large
\begin{equation}\label{phi-1-segment}
t_{\kappa_{-;\varphi-1}}^{\varrho_{\varphi-2}}\cdot t_{\kappa_{+;\varphi-1}}^{\vartheta_{\varphi-1}}\cdot {(\pi')}^{-1}_{\varphi-1}\cdot t_{\varrho_{\varphi-1}}^{\eta_{\varphi-1}}=t_{\varrho_{\varphi-2}}^{\eta_{\varphi-2}}\cdot t_{\kappa_{-;\varphi-1}}^{\varrho_{\varphi-2}-\eta_{\varphi-2}}\cdot {(\check{\pi}')}^{-1}_{\varphi-1}\cdot t_{\kappa_{+;\varphi-1}}^{\vartheta_{\varphi-1}+\eta_{\varphi-1}},
\end{equation}
\endgroup
where,
\begingroup
\large
$${(\check{\pi}')}^{-1}_{\varphi-1}=\begin{cases} {(\breve{\pi}')}^{-1}_{\varphi-1} & ~~\eta_{\varphi-2}>0 \\ \\ t_{\rho^{(w_{\varphi-1})}_{1}+\vartheta_{\varphi-1}}^{\vartheta_{\varphi-1}}\cdot{(\breve{\pi}')}^{-1}_{\varphi-1} & ~~\eta_{\varphi-2}=0\end{cases}$$
\endgroup
If $\varphi=2$, then $\varrho_{\varphi-2}=0$, and the theorem has been proved. If $\varphi>2$, we continue in the same way, where finally we get
\begingroup
\Large
\begin{equation}
\pi^{-1}=\left({(\check{\pi}')}^{-1}_{1}\cdot t_{\kappa_{+;1}}^{\vartheta_{1}+\eta_{1}}\right)\prod_{j=2}^{\varphi}\left(t_{\kappa_{-;j}}^{\varrho_{j-1}-\eta_{j-1}}\cdot {(\check{\pi}')}^{-1}_{j}\cdot t_{\kappa_{+;j}}^{\vartheta_{j}+\eta_{j}}\right)\cdot \pi^{-1}_{0},
\end{equation}
\endgroup
where, for every $1\leq j\leq \varphi$, ~${(\check{\pi}')}^{-1}_{j}$ is defined analogously to the definition of ${(\check{\pi}')}^{-1}_{\varphi}$.
\\

Hence, the $OGS$ canonical form of $\pi^{-1}$ is as follows:
 $${\pi}^{-1}=\prod_{\tilde{j}=1}^{\tilde{m}}t_{\tilde{k}_{\tilde{j}}}^{\tilde{i}_{\tilde{k}_{\tilde{j}}}},$$
 such that:
 \begin{itemize}
\item  $\tilde{k}_{\tilde{j}}=\kappa^{(v)}_{-;j}$ or $\tilde{k}_{\tilde{j}}=\kappa^{(v)}_{+;j}$, for $1\leq v\leq z(\pi)$, and $1\leq j\leq m^{(v)}$;
\item If $\tilde{k}_{\tilde{j}}=\kappa^{(v)}_{+;j}$, then $\tilde{i}_{\tilde{k}_{\tilde{j}}}=\sum_{\gamma}\left(\vartheta^{(v_{\gamma})}_{j_{\gamma}}+\eta^{(v_{\gamma})}_{j_{\gamma}}\right),$
such that $\kappa^{(v)}_{+;j}=\kappa^{(v_{\gamma})}_{+;j_{\gamma}};$
\item If $\tilde{k}_{\tilde{j}}=\kappa^{(v)}_{-;j}$, then $\tilde{i}_{\tilde{k}_{\tilde{j}}}=\kappa^{(v)}_{-;j}-\sum_{\gamma}\left(\vartheta^{(v_{\gamma})}_{j_{\gamma}}+\eta^{(v_{\gamma})}_{j_{\gamma}-1}\right),$
such that $\kappa^{(v)}_{-;j}=\kappa^{(v_{\gamma})}_{-;j_{\gamma}}.$
\end{itemize}

Thus, $\pi^{-1}$ satisfies all the conditions of Proposition \ref{inverse-factorization}.
Hence,
\begin{itemize}
\item $z(\pi^{-1})$, the number of standard $OGS$ elementary factors of $\pi^{-1}$ equals to the number of different non-zero values $\kappa^{(v)}_{-;j}$, where $1\leq v\leq z(\pi)$, and $1\leq j\leq m^{(v)}$;
\item For every $1\leq w\leq z(\pi^{-1})$, ~$maj\left({\pi^{-1}}^{(w)}\right)=\kappa^{(v)}_{-;j}$ for some $1\leq v\leq z(\pi)$, and for every given non-zero value of $\kappa^{(v)}_{-;j}$, there exists exactly one $w$ such that $maj\left({\pi^{-1}}^{(w)}\right)=\kappa^{(v)}_{-;j}$ (i.e., There is a one-to-one correspondence between the non-zero different values of $\kappa^{(v)}_{-;j}$ and the descent set of ~$\pi^{-1}$).
\end{itemize}

\end{proof}

\begin{example}
Consider $$\pi=t_{9}^{3}\cdot t_{10}^{3}\cdot t_{12}^{5}\cdot t_{13}^{3}\cdot t_{15}^{5}\cdot t_{17}^{2}\cdot t_{18}^{7}.$$

The permutation presentation of $\pi$:
$$\pi=[12; ~13; ~15; ~17; ~2; ~3; ~6; ~9; ~10; ~16; ~4; ~5; ~16; ~18; ~1; ~7; ~8; ~11].$$

Thus, $$des\left(\pi\right)=\{4, ~10, ~14\}.$$

Now, the standard $OGS$ elementary factorization of $\pi$ is as follows:
$$\pi^{(1)}=t_{9}^{3}\cdot t_{10}, ~~\pi^{(2)}=t_{10}^{2}\cdot t_{12}^{5}\cdot t_{13}^{3},  ~~\pi^{(3)}=t_{15}^{5}\cdot t_{17}^{2}\cdot t_{18}^{7}.$$
\\
Notice, the number of elementary factors equals to the number of the descents of $\pi$, where each descent of $\pi$ corresponds to the sum of the exponents of an elementary factor.
\\

The generator-length of the elementary factors are: $m^{(1)}=2$, ~~$m^{(2)}=3$, ~~$m^{(3)}=3$.
\\

Now, we consider the parameters which we use for the computation of $\pi^{-1}$.

 $$h^{(1)}_{1}=9, ~~h^{(1)}_{2}=10,$$
 $$h^{(2)}_{1}=10, ~~h^{(2)}_{2}=12, h^{(2)}_{3}=13,$$
 $$h^{(3)}_{1}=15, ~~h^{(3)}_{2}=17, ~~h^{(3)}_{3}=18.$$
$$\imath^{(1)}_{1}=3, ~~\imath^{(1)}_{2}=1,$$
$$\imath^{(2)}_{1}=2, ~~\imath^{(2)}_{2}=5, ~~\imath^{(2)}_{3}=3,$$
$$\imath^{(3)}_{1}=5, ~~\imath^{(3)}_{2}=7, ~~\imath^{(3)}_{3}=14.$$

Therefore,
$$\varrho^{(1)}_{1}=3,  ~~~~~~~~~~~~~~~~\varrho^{(1)}_{2}=3+1=4=maj\left(\pi^{(1)}\right),$$
$$\varrho^{(2)}_{1}=2, ~~\varrho^{(2)}_{2}=2+5=7, ~~\varrho^{(2)}_{3}=2+5+3=10=maj\left(\pi^{(2)}\right),$$
$$\varrho^{(3)}_{1}=5, ~~\varrho^{(3)}_{2}=5+2=7, ~~\varrho^{(3)}_{3}=5+2+7=14=maj\left(\pi^{(3)}\right).$$

Thus, by the formula $\vartheta^{(v)}_{j}=h^{(v)}_{j}-maj\left(\pi^{(v)}\right)$:
$$\vartheta^{(1)}_{1}=9-4=5, ~~~~~~~~~~~~~\vartheta^{(1)}_{2}=10-4=6,$$
$$\vartheta^{(2)}_{1}=10-10=0, ~~\vartheta^{(2)}_{2}=12-10=2, ~~\vartheta^{(2)}_{3}=13-10=3,$$
$$\vartheta^{(3)}_{1}=15-14=1, ~~\vartheta^{(3)}_{2}=17-14=3, ~~\vartheta^{(3)}_{3}=18-14=4.$$

Hence, by using the formula $\kappa^{(v\rightarrow v)}_{-;j}=\varrho^{(v)}_{j-1}+\vartheta^{(v)}_{j}$  and $\kappa^{(v\rightarrow v)}_{+;j}=\varrho^{(v)}_{j}+\vartheta^{(v)}_{j}$:
$$\kappa^{(1\rightarrow 1)}_{-;1}=0+5=5, ~~\kappa^{(1\rightarrow 1)}_{-;2}=3+6=9,$$
$$\kappa^{(1\rightarrow 1)}_{+;1}=3+5=8, ~~\kappa^{(1\rightarrow 1)}_{+;2}=4+6=10.$$

$$\kappa^{(2\rightarrow 2)}_{-;1}=0+0=0, ~~\kappa^{(2\rightarrow 2)}_{-;2}=2+2=4, ~~\kappa^{(2\rightarrow 2)}_{-;3}=7+3=10,$$
$$\kappa^{(2\rightarrow 2)}_{+;1}=2+0=2, ~~\kappa^{(2\rightarrow 2)}_{+;2}=7+2=9, ~~\kappa^{(2\rightarrow 2)}_{+;3}=10+3=13.$$

$$\kappa^{(3\rightarrow 3)}_{-;1}=0+1=1, ~~\kappa^{(3\rightarrow 3)}_{-;2}=5+3=8, ~~\kappa^{(3\rightarrow 3)}_{-;3}=7+4=11,$$
$$\kappa^{(3\rightarrow 3)}_{+;1}=5+1=6, ~~\kappa^{(3\rightarrow 3)}_{+;2}=7+3=10, ~~\kappa^{(3\rightarrow 3)}_{+;3}=14+4=18.$$

Now, compute $\kappa^{(v\rightarrow r)}_{-;j}$ and $\kappa^{(v\rightarrow r)}_{+;j}$, for $1\leq v\leq 3$, ~$v<r\leq 3$, and ~$1\leq j\leq m^{(v)}$, ~by using the algorithm in Definition \ref{kappa}:
\\
\\

  Since $\varrho^{(2)}_{1}<\kappa^{(1\rightarrow 1)}_{-;1}<\varrho^{(2)}_{2}$: ~~~~$\kappa^{(1\rightarrow 2)}_{-;1}=\kappa^{(1\rightarrow 1)}_{-;1}+\vartheta^{(2)}_{2}=5+2=7.$
 \\

  Since $\varrho^{(3)}_{2}\leq\kappa^{(1\rightarrow 2)}_{-;1}<\varrho^{(3)}_{3}$: ~~~~$\kappa^{(1\rightarrow 3)}_{-;1}=\kappa^{(1\rightarrow 2)}_{-;1}+\vartheta^{(3)}_{3}=7+4=11.$
\\
\\

  Since $\varrho^{(2)}_{2}<\kappa^{(1\rightarrow 1)}_{+;1}<\varrho^{(2)}_{3}$: ~~~~$\kappa^{(1\rightarrow 2)}_{+;1}=\kappa^{(1\rightarrow 1)}_{+;1}+\vartheta^{(2)}_{3}=8+3=11.$
  \\

  Since $\varrho^{(3)}_{2}<\kappa^{(1\rightarrow 2)}_{+;1}<\varrho^{(3)}_{3}$: ~~~~$\kappa^{(1\rightarrow 3)}_{+;1}=\kappa^{(1\rightarrow 2)}_{+;1}+\vartheta^{(3)}_{3}=11+4=15.$
\\
\\

  Since $\varrho^{(2)}_{2}<\kappa^{(1\rightarrow 1)}_{-;2}<\varrho^{(2)}_{3}$: ~~~~$\kappa^{(1\rightarrow 2)}_{-;2}=\kappa^{(1\rightarrow 1)}_{-;2}+\vartheta^{(2)}_{3}=9+3=12.$
  \\

  Since $\varrho^{(3)}_{2}<\kappa^{(1\rightarrow 2)}_{-;2}<\varrho^{(3)}_{3}$: ~~~~$\kappa^{(1\rightarrow 3)}_{-;2}=\kappa^{(1\rightarrow 2)}_{-;2}+\vartheta^{(3)}_{3}=12+4=16.$
\\
\\

  Since $\varrho^{(2)}_{2}<\kappa^{(1\rightarrow 1)}_{+;2}<\varrho^{(2)}_{3}$: ~~~~$\kappa^{(1\rightarrow 2)}_{+;2}=\kappa^{(1\rightarrow 1)}_{+;2}+\vartheta^{(2)}_{3}=10+3=13.$
  \\

  Since $\varrho^{(3)}_{2}<\kappa^{(1\rightarrow 2)}_{+;2}<\varrho^{(3)}_{3}$: ~~~~$\kappa^{(1\rightarrow 3)}_{+;2}=\kappa^{(1\rightarrow 2)}_{+;2}+\vartheta^{(3)}_{3}=13+4=17.$
\\
\\

  Since $\kappa^{(2\rightarrow 2)}_{-;1}<\varrho^{(3)}_{1}$: ~~~~$\kappa^{(2\rightarrow 3)}_{-;1}=\kappa^{(2\rightarrow 2)}_{-;1}+\vartheta^{(3)}_{1}=0+1=1.$
\\
\\

  Since $\kappa^{(2\rightarrow 2)}_{+;1}<\varrho^{(3)}_{1}$: ~~~~$\kappa^{(2\rightarrow 3)}_{+;1}=\kappa^{(2\rightarrow 2)}_{+;1}+\vartheta^{(3)}_{1}=2+1=3.$
\\
\\

  Since $\kappa^{(2\rightarrow 2)}_{-;2}<\varrho^{(3)}_{1}$: ~~~~$\kappa^{(2\rightarrow 3)}_{-;2}=\kappa^{(2\rightarrow 2)}_{-;2}+\vartheta^{(3)}_{1}=4+1=5.$
\\
\\

  Since $\varrho^{(3)}_{2}<\kappa^{(2\rightarrow 2)}_{+;2}<\varrho^{(3)}_{3}$: ~~~~$\kappa^{(2\rightarrow 3)}_{+;2}=\kappa^{(2\rightarrow 2)}_{+;2}+\vartheta^{(3)}_{3}=9+4=13.$
\\
\\

  Since $\varrho^{(3)}_{2}<\kappa^{(2\rightarrow 2)}_{-;3}<\varrho^{(3)}_{3}$: ~~~~$\kappa^{(2\rightarrow 3)}_{-;3}=\kappa^{(2\rightarrow 2)}_{-;3}+\vartheta^{(3)}_{3}=10+4=14.$
\\
\\

  Since $\varrho^{(3)}_{2}<\kappa^{(2\rightarrow 2)}_{+;3}<\varrho^{(3)}_{3}$: ~~~~$\kappa^{(2\rightarrow 3)}_{-;3}=\kappa^{(2\rightarrow 2)}_{-;3}+\vartheta^{(3)}_{3}=13+4=17.$
\\

  Thus,
   $$\kappa^{(1)}_{-;1}=11, ~~~~~~~~~~\kappa^{(1)}_{-;2}=16,$$
   $$\kappa^{(1)}_{+;1}=15, ~~~~~~~~~~\kappa^{(1)}_{+;2}=17.$$

   $$\kappa^{(2)}_{-;1}=1, ~~\kappa^{(2)}_{-;2}=5, ~~\kappa^{(2)}_{-;3}=14,$$
   $$\kappa^{(2)}_{+;1}=3, ~~\kappa^{(2)}_{+;2}=13, ~~\kappa^{(2)}_{+;3}=17.$$

   $$\kappa^{(3)}_{-;1}=1, ~~\kappa^{(3)}_{-;2}=8, ~~\kappa^{(3)}_{-;3}=11,$$
   $$\kappa^{(3)}_{+;1}=6, ~~\kappa^{(3)}_{+;2}=10, ~~\kappa^{(3)}_{+;3}=18.$$

    Thus, the values of $\kappa$ in increasing order are:
    $$\kappa^{(3)}_{-;1}=\kappa^{(2)}_{-;1}=1, ~~\kappa^{(2)}_{+;1}=3, ~~\kappa^{(2)}_{-;2}=5, ~~\kappa^{(3)}_{+;1}=6, ~~\kappa^{(3)}_{-;2}=8, ~~\kappa^{(3)}_{+;2}=10,$$
    $$\kappa^{(3)}_{-;3}=\kappa^{(1)}_{-;1}=11, ~~\kappa^{(2)}_{+;2}=13, ~~\kappa^{(2)}_{-;3}=14, ~~\kappa^{(1)}_{+;1}=15, ~~\kappa^{(1)}_{-;2}=16,$$
    $$\kappa^{(1)}_{+;2}=\kappa^{(2)}_{+;3}=17, ~~\kappa^{(3)}_{+;3}=18.$$

    Now, we find $\chi^{(v)}_{j}$ and $\eta^{(v)}_{j}$  for $1\leq v\leq 3$, and $1\leq j\leq m^{(v)}$, by using Definitions \ref{chi}, \ref{eta}.
    Obviously,
    ~$\chi^{(1)}_{j}=0.$  ~Therefore, ~$\eta^{(1)}_{j}=0$, ~for every $1\leq j\leq 2$.

    Now we compute $\chi^{(v)}_{j}$ and $\eta^{(v)}_{j}$, ~for $2\leq v\leq 3$, and $1\leq j\leq m^{(v)}-1$.
    \\

    $\kappa^{(2)}_{+;1}<\kappa^{(1)}_{-;1}$~
    applies ~$\chi^{(2)}_{1}=1.$
    Therefore,
     ~$\eta^{(2)}_{1}=0.$
\\

    $\kappa^{(1)}_{-;1}<\kappa^{(2)}_{+;2}<\kappa^{(1)}_{+;1}$~
    applies ~$\chi^{(2)}_{2}=0.$
    Therefore,
     ~$\eta^{(2)}_{2}=\vartheta^{(1)}_{1}=5.$
\\

   $\kappa^{(2)}_{-;2}<\kappa^{(3)}_{+;1}<\kappa^{(2)}_{+;2}$~
    and
    ~$\kappa^{(3)}_{+;1}<\kappa^{(1)}_{-;1}$~
    applies ~$\chi^{(3)}_{1}=1.$
    Therefore,
    ~$\eta^{(3)}_{1}=\vartheta^{(2)}_{2}=2.$
\\

   $\kappa^{(2)}_{-;2}<\kappa^{(3)}_{+;2}<\kappa^{(2)}_{+;2}$~
    and
    ~$\kappa^{(3)}_{+;2}<\kappa^{(1)}_{-;1}$~
    applies ~$\chi^{(3)}_{1}=1.$
    Therefore,
    ~$\eta^{(3)}_{2}=\vartheta^{(2)}_{2}=2.$
\\

Thus, by Theorem \ref{inverse-general-sn}, the standard $OGS$ canonical form of $\pi^{-1}$  is a product of all the elements of the form
\begingroup
\large
$$t_{\kappa^{(v)}_{-;j}}^{\kappa^{(v)}_{-;j}-\left(\vartheta^{(v)}_{j}+\eta^{(v)}_{j-1}\right)}, ~~~~t_{\kappa^{(v)}_{+;j}}^{\vartheta^{(v)}_{j}+\eta^{(v)}_{j}}$$
\endgroup
where, $1\leq v\leq 3$ and $1\leq j\leq m^{(v)}$.
\\

Therefore, by putting $\kappa^{(v)}_{-;j}$ and $\kappa^{(v)}_{+;j}$ in increasing order, and substituting the corresponding $\vartheta^{(v)}_{j}$, and $\eta^{(v)}_{j}$ or $\eta^{(v)}_{j-1}$ in the exponent for every $1\leq v\leq 3$ and $1\leq j\leq m^{(v)}$,  we get the following standard $OGS$ canonical form of $\pi^{-1}$:
\begin{align*}
\pi^{-1}&=t_{3}^{0}\cdot t_{5}^{5-2}\cdot t_{6}^{1+2}\cdot t_{8}^{8-(3+2)}\cdot t_{10}^{3+2}\cdot t_{11}^{11-(4+2+5)}\cdot t_{13}^{2+5}\cdot t_{14}^{14-(3+5)}\cdot t_{15}^{5}\cdot t_{16}^{16-6}\cdot t_{17}^{6+3}\cdot t_{18}^{4} \\ \\ &= t_{5}^{3}\cdot t_{6}^{3}\cdot t_{8}^{3}\cdot t_{10}^{5}\cdot t_{13}^{7}\cdot t_{14}^{6}\cdot t_{15}^{5}\cdot t_{16}^{10}\cdot t_{17}^{9}\cdot t_{18}^{4}.
\end{align*}
The standard $OGS$ elementary factorization of $\pi^{-1}$:

The $6$  different values of $\kappa^{(v)}_{-;j}$, such that $1\leq v\leq 3$, and $1\leq j\leq m^{(v)}$, implies $6$ different $OGS$ elementary factors as follows:
\begin{itemize}
\item $\kappa^{(3)}_{-;1}=\kappa^{(2)}_{-;1}=1$ ~implies ~$maj\left({\pi^{-1}}^{(1)}\right)=1$;
\item $\kappa^{(2)}_{-;2}=5$ ~implies ~$maj\left({\pi^{-1}}^{(2)}\right)=5$;
\item $\kappa^{(3)}_{-;2}=8$ ~implies ~$maj\left({\pi^{-1}}^{(3)}\right)=8$;
\item $\kappa^{(3)}_{-;3}=\kappa^{(1)}_{-;1}=11$ ~implies ~$maj\left({\pi^{-1}}^{(4)}\right)=11$;
\item $\kappa^{(2)}_{-;3}=14$ ~implies ~$maj\left({\pi^{-1}}^{(5)}\right)=14$;
\item $\kappa^{(1)}_{-;3}=16$ ~implies ~$maj\left({\pi^{-1}}^{(6)}\right)=16$;
\end{itemize}
$${\pi^{-1}}^{(1)}=t_{5}, ~~~~{\pi^{-1}}^{(2)}=t_{5}^{2}\cdot t_{6}^{3}, ~~~~{\pi^{-1}}^{(3)}=t_{8}^{3}\cdot t_{10}^{5},$$
$${\pi^{-1}}^{(4)}=t_{13}^{7}\cdot t_{14}^{4}, ~~~~{\pi^{-1}}^{(5)}=t_{14}^{2}\cdot t_{15}^{5}\cdot t_{16}^{7}, ~~~~{\pi^{-1}}^{(6)}=t_{16}^{3}\cdot t_{17}^{9}\cdot t_{18}^{4}.$$

Indeed, the permutation presentation of $\pi^{-1}$:

$$\pi^{-1}=[15; ~5; ~6; ~11; ~12; ~7; ~16; ~17; ~8; ~9; ~18; ~1; ~2; ~13; ~3; ~10; ~4; ~14].$$

Where, $$des\left(\pi^{-1}\right)=\{1, ~5, ~8, ~11, ~14, ~16\}.$$

\end{example}

\section{Conclusions and future plans}

In the paper, we introduced a quite interesting generalization of the fundamental theorem for abelian groups to two important and very elementary families of non-abelian Coxeter groups, the $I$-type (where the number of vertices on the Coxeter graph is always two, the smallest possibly number of vertices for a non-abelian Coxeter group, although the lace of the connecting edge depends on the group, and can be any integer), and the $A$-type (where the number of vertices depends on the group, and can be any integer greater than one, although the lace of the connecting edges are simply laced, and with a minimal number of connections such that the Coxeter graph is connected and non-abelian). We showed canonical forms, with very interesting exchange laws, and quite interesting properties concerning the Coxeter lengths of the elements. The interesting results for the two elementary families of non-abelian Coxeter groups motivate generalization for further families of Coxeter and generalized Coxeter groups, which have an importance in the classification of Lie-Algebras and the Lie-type simple groups, and in other fields of mathematics, such as algebraic geometry for classification of fundamental groups of Galois covers of surfaces \cite{alst}. In the first step it is interesting to generalize the standard $OGS$ canonical forms and the exchange laws for the finite classical families of $B$ and $D$-type, which have presentations as signed permutations,  then to the affine classical families $\tilde{A}$, $\tilde{B}$, $\tilde{C}$, and $\tilde{D}$, and also to other generalizations of the mentioned Coxeter groups, as the complex reflection groups $G(r,p,n)$ \cite{ST} or the generalized affine classical groups, the definition of which is described in \cite{rtv}, \cite{ast}.

\end{document}